\documentclass[10pt]{amsart}
\usepackage{amsmath,amssymb,amsthm}
\usepackage{bbm}
\usepackage{graphicx,tikz,mathtools}
\usepackage[left=2.5cm,right=2.5cm,top=2.5cm,bottom=2.5cm]{geometry}
\usepackage[pdftex]{hyperref}
\usepackage{cite}
\usepackage{mathrsfs} 
\usepackage{verbatim}
\usepackage{caption}
\usepackage{subcaption}
\usepackage{graphicx}
\hypersetup{
	colorlinks=true,
	linkcolor=blue, 
	citecolor=blue,
	filecolor=blue,
	urlcolor=blue,
}

\makeatletter
\@namedef{subjclassname@2020}{\textup{2020} Mathematics Subject Classification}
\makeatother

\newtheorem{theorem}{Theorem}[section]
\newtheorem{proposition}[theorem]{Proposition}
\newtheorem{corollary}[theorem]{Corollary}

\newtheorem{lemma}[theorem]{Lemma}
\newtheorem*{conjecture*}{Conjecture}

\newcommand{\triple}[1]{{\left\vert\kern-0.25ex\left\vert\kern-0.25ex\left\vert #1
        \right\vert\kern-0.25ex\right\vert\kern-0.25ex\right\vert}}

\theoremstyle{definition}

\newtheorem{remark}[theorem]{Remark}

\newcommand{\al}{\alpha}
\newcommand{\be}{\beta}
\newcommand{\de}{\delta}

\newcommand{\te}{\theta}
\newcommand{\vp}{\varphi}

\newcommand{\Om}{\Omega}

\newcommand{\phitl}{\widetilde{\phi}_{\infty}}
\def\NN{\mathbb{N}}
\def\RR{\mathbb{R}}

\def\bK{\mathbf{K}}

\newcommand{\cB}{{\mathcal B}}
\newcommand{\cC}{{\mathcal C}}

\newcommand{\cE}{{\mathcal E}}
\newcommand{\cF}{{\mathcal F}}

\newcommand{\cH}{{\mathcal H}}

\newcommand{\cK}{{\mathcal K}}

\newcommand{\cL}{{\mathcal L}}

\newcommand{\cM}{{\mathcal M}}
\newcommand{\cN}{{\mathcal N}}

\newcommand{\cR}{{\mathcal R}}

\newcommand{\cX}{{\mathcal X}}

\newcommand{\pd}{\partial}
\newcommand\minus\backslash

\newcommand{\Hud}{\dot{H}^{\frac{1}{2}}}
\newcommand{\Dud}{(-\Delta)^{\frac{1}{2}}}
\newcommand{\Sjk}{\sum_{j=1}^k}
\newcommand{\ls}{\lesssim}

\newcommand\lan\langle
\newcommand\ran\rangle

\newcommand{\supp}{\operatorname{supp}}

\renewcommand\leq\leqslant
\renewcommand\geq\geqslant

\numberwithin{equation}{section}

\begin{document}
 
\title[ Smooth nonradial  stationary solutions to the SQG equation]{ Smooth nonradial  stationary solutions \\ to SQG  via the half-Yamabe equation}

\author[\'A. Castro]{\'{A}ngel Castro}
\address{    
\newline
\textbf{{\small \'{A}ngel Castro}}
\newline
Corresponding author
\vspace{0.15cm}
\newline \indent Instituto de Ciencias Matem\'aticas, Consejo Superior de Investigaciones Cient\'\i ficas, 28049 Madrid, Spain}
\email{angel\_castro@icmat.es}

\author[A. J. Fern\'andez]{Antonio J.\ Fern\'andez}
\address{ \vspace{-0.4cm}
\newline 
\textbf{{\small Antonio J. Fern\'andez}} 
\vspace{0.15cm}
\newline \indent Departamento de Matem\'aticas, Universidad Aut\'onoma de Madrid, 28049 Madrid, Spain}
 \email{antonioj.fernandez@uam.es}
 
\author[C. Garc\'ia]{Claudia Garc\'ia}
\address{ \vspace{-0.4cm}
\newline
\textbf{{\small Claudia Garc\'ia}}
\vspace{0.15cm}
\newline \indent Departamento de Matem\'atica Aplicada, Universidad de Granada, 18071 Granada, Spain}
\email{claudiagarcia@ugr.es}

%
%

\begin{abstract}
    We prove the existence of infinitely many smooth nonradial stationary solutions to the surface quasi-geostrophic (SQG) equation with finite kinetic energy. Our construction is based on a family of nonradial sign-changing solutions to the two-dimensional half-Yamabe equation, obtained via a Lyapunov--Schmidt reduction and concentrated at the vertices of a regular polygon. As the number of vertices tends to infinity, the associated stationary SQG solutions converge to a radial stationary profile centered at the origin, together with a lower-order vortex sheet correction.
\end{abstract}

\maketitle
 
\section{Introduction}

 The surface quasi-geostrophic (SQG) equation, which models the temperature on a fluid surface under the effect of the Coriolis force, is a fundamental model used in geophysical fluid dynamics. It reads as follows:
\begin{equation} \label{E.SQG}
\left\{
\begin{aligned}
\, & \pd_t \theta + \nabla^{\perp} \psi \cdot \nabla \theta = 0 \quad && \textup{in } \RR^2 \times (0,T)\,,\\
& \psi = (-\Delta)^{-\frac12} \theta && \textup{in } \RR^2 \times [0,T)\,, \\
& \theta(\cdot,0) = \theta_0 && \textup{in } \RR^2\,.
\end{aligned}
\right.
\end{equation}
Here, $\theta$ represents the temperature or potential vorticity, and the associated velocity field $u:=\nabla^\perp\psi \equiv (\partial_{x_2} \psi, -\partial_{x_1}\psi)$, which is solenoidal, is linked to the potential vorticity through the second equation in \eqref{E.SQG}. In addition to its relevance in geophysical fluid dynamics, the SQG equation is closely related to the incompressible 3D Euler equations. Indeed, it is widely regarded as a good two-dimensional model for the vorticity evolution in 3D Euler \cite{CMT94}.

The SQG equation is locally well-posed in $H^s$, for $s > 2$, see \cite{CMT94}; and in $C^{k,\alpha}$, for $k \geq 1$ and $\al \in (0,1)$, see  \cite{W05}. However, unlike 2D Euler, the formation of singularities in finite time from smooth initial data remains open. 
 
The global existence of weak solutions holds in $L^p$, for $p>\frac{4}{3}$, see \cite{R95, M08}. Nevertheless, very recently, many results about nonuniqueness of weak solutions were obtained, see \cite{BSV19, CFMS25, CKL21, DP23}. Finally, note that, in \cite{CordobaZoroa22, jeongkimsqg}, the authors find critical exponents where ill-posedness occurs, with solutions losing regularity instantaneously in spaces where uniqueness holds. See also \cite{CZOpp,CLZpp} for ill-posedness in supercritical spaces.

There are very few results concerning the existence of smooth global solutions to the SQG equation. Moreover, most known examples are either traveling or rotating solutions evolving by rigid motions, see \cite{ADdPMW21, CCGS20, GS19}. The only known nontrivial global smooth solution to the SQG equation that does not evolve by rigid motion was only recently constructed in \cite[Corollary 1.3]{CFMS25}. 

In this paper, we are mainly concerned with the existence of smooth finite-energy stationary solutions to the SQG equation in $\RR^2$. Radial examples are easy to construct: indeed, it suffices to pick any smooth radially symmetric function $\te$ with enough decay at infinity. Then, choosing $\psi$ as in the second equation in \eqref{E.SQG}, it is immediate to check that $\te$ is a smooth stationary solution to \eqref{E.SQG}. The natural question here is therefore the existence (or not) of smooth nonradial stationary solutions to the SQG equation with finite kinetic energy. Let us highlight that the existence of  homogeneous (unbounded) nonradial stationary solutions with infinite energy has been recently proven both in \cite{abejeongserrano} and \cite{pacual-caballo}.  

In recent years there has been an emergence of rigidity results for steady two-dimensional fluids. We refer in particular to \cite{GSPSY21} for outstanding rigidity results concerning the SQG equation, both in the smooth setting and for patch type solutions. Nevertheless, the existence of smooth nonradial finite-energy stationary solutions to the SQG equation remains a well-known open problem. Our first main result answers this question affirmatively. 

We construct the first smooth nonradial stationary solutions to the SQG equation in $\RR^2$ with finite kinetic energy. More precisely, our first main result reads as follows:

\begin{theorem} \label{T.mainThmSQG}
There exist infinitely many smooth nonradial stationary solutions to \eqref{E.SQG} with finite kinetic energy. Furthermore, if we index these solutions as $(\theta_k)_{k=k_\star}^{\infty}$, it follows that, up to a subsequence,
\begin{equation} \label{E.distributionalConvergence}
\log k \Big( \theta_k + U^3\Big) \xrightharpoonup{\hspace{0.4cm}} \frac{\pi}{\sqrt{2} } \, \sigma_{\partial \mathbb{D}} + R \quad \textup{\textit{in} } \mathcal{D}'(\RR^2)\,, \quad \textup{\textit{as} } k \to \infty\,.
\end{equation}
Here:
\begin{itemize}
    \item $U$ denotes the standard bubble, which is given by $U(x):= (1+|x|^2)^{-\frac12}$. 
    \item $\sigma_{\partial \mathbb{D}}$ denotes the arc-length measure on $\partial \mathbb{D} := \{x \in \RR^2: |x|=1\}$.
    \item $R \in \dot{H}^{s} \cap L^p$ for all $1 \leq p < \infty$ and all $0 < s < \frac12$.
\end{itemize}
\end{theorem}

The stationary solutions to the SQG equation constructed in Theorem \ref{T.mainThmSQG}, $(\theta_k)_{k=k_\star}^\infty$, are nonradial perturbations of a negative radial function $-U^3 = -(1+|x|^2)^{-\frac32}$ plus a circular vortex sheet type correction at order $(\log k)^{-1}$. Note that the more regular remainder $R$ is to some extent explicit, and given in terms of the perturbations $(\te_k)_{k=k_\star}^\infty$ (see \eqref{E.Rdistributional}). Moreover, it is worth pointing out that the kinetic energy of these solutions is not uniformly bounded in $k$ as $k \to \infty$.  

The proof of Theorem \ref{T.mainThmSQG} relies on the following classical observation: if $\psi$ is a smooth solution to an autonomous elliptic PDE of the form
\begin{equation} \label{E.autonomousPDE}
(-\Delta)^{\frac12} \psi - F(\psi) = 0\quad \textup{in } \RR^2\,,
\end{equation}
for some smooth function $F$, then $\theta = F(\psi)$ defines a smooth, stationary solution to \eqref{E.SQG}. There exist many examples of functions $F$ for which one can easily find smooth solutions to \eqref{E.autonomousPDE}. The key difficulty here is precisely finding nonradial solutions, for some smooth function $F$. We choose and fix $F(t) = t^3$. This choice immediately drives us to the well-known two-dimensional half-Yamabe equation, namely
\begin{equation} \label{E.Yamabe}
(-\Delta)^{\frac12} \psi = \psi^3 \quad \textup{in } \RR^2\,.
\end{equation}

It is well known (see, for instance, \cite{CLO06, LZ95, L04}) that, under suitable mild assumptions, all positive solutions to \eqref{E.Yamabe} are of the form
\begin{equation*}  
U_{\de,\xi}(x) := \frac{1}{\sqrt{\de}}\, U \Big( \frac{x-\xi}{\de} \Big) = \frac{\sqrt{\de}}{(\de^2 + |x-\xi|^2)^\frac12}\,, \quad \textup{for some } \de > 0 \textup{ and } \xi \in \RR^2\,,
\end{equation*}
where we recall that $U$ is the standard bubble, which is given by
\begin{equation*} 
U(x) := \frac{1}{(1+|x|^2)^{\frac12}}\,.
\end{equation*}
Consequently, if we want to rely on \eqref{E.Yamabe} in order to prove Theorem \ref{T.mainThmSQG}, we should look for smooth nonradial finite-energy sign-changing solutions to \eqref{E.Yamabe}. We therefore seek this kind of solutions to \eqref{E.Yamabe}, and our second main result establishes their existence.

More precisely, our second main result, which concerns the two-dimensional half-Yamabe equation \eqref{E.Yamabe}, and is completely independent of the SQG equation, reads as follows:

\begin{theorem} \label{T.mainThmYamabe}
There exists $k_0 \in \NN$ such that, for all $k \geq k_0$, \eqref{E.Yamabe} has a solution of the form
$$
\psi_k(x) = \sum_{j=1}^k \frac{1}{\widehat{\delta}_k^{\,\frac12}}\, U \bigg( \frac{x-\xi_j^k}{\widehat{\delta}_k} \bigg) - U(x) + \phi_k(x)\,.
$$
Here, $U$ is the standard bubble given in Theorem \ref{T.mainThmSQG}, and $\phi_k \in \dot{H}^{\frac12}$, $(\xi_j^k)_{j=1}^k \subset \RR^2$ and $\widehat{\delta}_k > 0$ satisfy
\begin{align*}
\|\phi_k\|_{\dot{H}^{\frac12}} \to 0\,, \quad \textup{as } k \to \infty\,,
\end{align*}
and
$$
\xi_j^k \sim \Big( \cos \big( \tfrac{2\pi(j-1)}{k} \big), \, \sin \big( \tfrac{2\pi (j-1)}{k} \big) \Big)\,, \qquad \widehat{\delta}_k \sim \frac{1}{(k \log k)^2}\,. 
$$
\end{theorem}

We refer to Figures \ref{k5} and \ref{k12} for a schematic picture of the solutions constructed in Theorem \ref{T.mainThmYamabe} and its level curves, with $k = 5$ and $k = 12$, respectively.

Note that, having Theorem \ref{T.mainThmYamabe} at hand, the proof of Theorem \ref{T.mainThmSQG} reduces to showing that, for all $k$ sufficiently large, the solutions to \eqref{E.Yamabe} constructed in Theorem \ref{T.mainThmYamabe} belong to $L^6 \cap C^{\infty}$, and are nonradial.

Theorem \ref{T.mainThmYamabe} provides the first finite-$\Hud$-energy sign-changing solutions to \eqref{E.Yamabe}. This result is inspired by the construction in \cite{GM16}, which in turn was motivated by the seminal paper \cite{dPMPP11}. Nevertheless, our result substantially differs from \cite{GM16, dPMPP11} from a technical point of view. While \cite{GM16, dPMPP11} use the so-called \textit{inner-outer gluing} method, we directly attack the problem using a Lyapunov-Schmidt reduction in $\Hud$. Our approach is actually inspired by the proof of \cite[Theorem 1.2]{PV19}. 

The paper \cite{GM16} deals with the more general fractional Yamabe problem
\begin{equation} \label{E.YamabeMonica}
(-\Delta)^s \psi = |\psi|^{\frac{4s}{n-2s}}\psi \quad \textup{in } \RR^n\,,
\end{equation}
with $s \in (0,1)$ and $n > 2s$. We refer to \cite{CG11,GQ13} for a detailed introduction to the fractional Yamabe problem. In \cite{GM16}, the authors construct nonradial sign-changing solutions to \eqref{E.YamabeMonica} in the case where $n \geq 3$ and $s \in (\frac12,1)$. We believe that their approach could be easily adapted to handle the case where $n =2$ and $s \in (\frac12,1)$. Nevertheless, as explicitly pointed out in \cite[p. 89]{GM16}, their \textit{gluing} scheme breaks down at the endpoint $s = \frac12$. The present work resolves precisely this endpoint case. It is also worth pointing out that, in \cite[p. 89]{GM16}, the authors conjecture that the rate
$\widehat{\de}_k \sim k^{-2}$ should hold for all $s \in (0,1)$. However, Theorem \ref{T.mainThmYamabe} shows that, for $n = 2$ and $s = \frac12$, this is not the case. 

\subsection{Strategy of the proofs} 
The proof of Theorem \ref{T.mainThmSQG} is divided into two steps. We first prove Theorem \ref{T.mainThmYamabe} using a Lyapunov-Schmidt reduction argument. Since the solutions to the two-dimensional half-Yamabe equation constructed in Theorem \ref{T.mainThmYamabe} yield stationary solutions to the SQG equation, to conclude the proof of Theorem \ref{T.mainThmSQG}, we are only missing to show that the resulting solutions are indeed smooth and nonradial, have finite kinetic energy, and satisfy the distributional convergence in \eqref{E.distributionalConvergence}. This is precisely the second step, which is developed in Section \ref{S.nonradial}. 


We now discuss the proof of Theorem \ref{T.mainThmYamabe}. We start off with the approximate solution 
\begin{equation}\label{intro:Ukt}
U_{k,t}(x) = \sum_{j=1}^k U_{j,k,t}(x) - U(x)\,,
\end{equation}
where $U$ is the standard bubble, and $U_{1,k,t}, \ldots, U_{k,k,t}$ are scaled and translated bubbles concentrating at the $k$ vertices of a regular polygon (see Figures \ref{k5} and \ref{k12} for a representation of $U_{k,t}$). Here, $k \in \NN$ has to be sufficiently large, while $t\in[a,b]$, for some fixed $0<a<b$, is a free parameter to be adjusted. 


The gist of the proof is to construct, for all $k$ sufficiently large, an exact solution to \eqref{E.Yamabe} of the form
$$
U_{k,t}+\phi_{k,t}\,, \quad \quad \textup{with} \quad \quad  \phi_{k,t} \xrightarrow{k \to \infty} 0 \textup{ in } \dot{H}^\frac12\,.
$$
We will address the problem by linearizing around $U_{k,t}$ and solving
the corresponding problem for $\phi_{k,t}$. It turns out that this is possible if the concentration rate $\delta_k$ and the parameter $t \in [a,b]$ are suitably chosen. 

Linearizing around $U_{k,t}$, and applying the inverse half-Laplacian operator, the two-dimensional half-Yamabe equation \eqref{E.Yamabe} can be rewritten as
\begin{equation}\label{intro:eq}
\cL_{k,t}^{*} \phi = \cE_{k,t}^{*} + \cN_{k,t}^{*} (\phi)\,, \quad \phi \in \Hud\,,
\end{equation}
where
\begin{equation*}
 \cL_{k,t}^{*} \phi:=\phi-(-\Delta)^{-\frac12}(3U^2_{k,t}\phi)\,, \quad    \cE_{k,t}^{*}:=(-\Delta)^{-\frac12}U^3_{k,t}-U_{k,t}\,, \quad \cN_{k,t}^{*} (\phi):=(-\Delta)^{-\frac12}(3U_{k,t}\phi^2+\phi^3)\,.
\end{equation*}
The main challenge to solve \eqref{intro:eq} is that the operator $\cL_{k,t}^{*}$ has an approximated nonempty kernel as $k \to \infty$. This prevents the use of any standard fixed point argument. We will obtain a solution to \eqref{intro:eq}, for all $k$ sufficiently large, using a Lyapunov-Schmidt reduction argument. 

The idea of using the number of bubbles $k$ as a parameter of the problem was pioneered in \cite{WY10}, and successfully used in \cite{dPMPP11,PV19,MM19,GM16}. Our paper is however the first time this idea is implemented in a two-dimensional problem. Taking into account that the existence of concentrating solutions in critical problems is strongly affected by the dimension, we believe this to be of independent interest. 
 
One of the key difficulties we have to face in proving Theorem \ref{T.mainThmYamabe} is to establish suitable a priori (uniform in $k$) estimates for a projected version of $\cL_{k,t}^{*}$, see Subsection \ref{SS.LT}. Although our proof is inspired by \cite{PV19}, it requires several new ingredients. Let us highlight the improved Sobolev type inequality for $k$-fold symmetric functions we prove in Proposition \ref{P.ImprovedSobolev}.

\subsection{Organization of the paper} In Section \ref{S.PIT} we prove several integral estimates that will be crucial in the proof of Theorem \ref{T.mainThmYamabe}. In particular, we establish an improved Sobolev type inequality in $\dot{H}^1(\RR^3)$ which we believe is of independent interest (cf. Proposition \ref{P.ImprovedSobolev}). In Section \ref{S.2} we prove Theorem \ref{T.mainThmYamabe} using a Lyapunov-Schmidt reduction argument. Section \ref{S.nonradial} is then devoted to proving Theorem \ref{T.mainThmSQG}. The first part of the section, which is actually independent of Theorem \ref{T.mainThmSQG}, is devoted to showing some extra properties of the solutions to \eqref{E.Yamabe} constructed in Theorem \ref{T.mainThmYamabe}. Having these extra properties at hand, in the second part, we finish the proof of Theorem \ref{T.mainThmSQG}. The paper concludes with an appendix containing some useful calculations for the proof of Theorem \ref{T.mainThmYamabe}.

\subsection{Notation} The space  $\dot{H}^{\frac12} := \dot{H}^{\frac12}(\RR^2)$ denotes the completion of $C_c^{\infty}(\RR^2)$ with respect to the norm  
$$
\|\phi\|_{\dot{H}^\frac12} := \left(\int_{\RR^2} |(-\Delta)^{\frac14}\phi|^2 dx\right)^{\frac12}\,, \quad \textup{for all } \phi \in \dot{H}^{\frac12}\,.
$$
Note that $\dot{H}^{\frac12}$ is a Hilbert space endowed with the scalar product 
$$
\langle \phi,\, \psi \rangle := \int_{\RR^2} (-\Delta)^{\frac14}\phi (-\Delta)^{\frac14}\psi\,  dx \,, \quad \textup{for all } \phi,\, \psi \in \dot{H}^{\frac12}\,.
$$
Throughout the paper, the expression $A\lesssim B$ means that there exists a constant $C$, independent of $k$ and $t$, such that $A\leq C B$. We use the same convention for `` $\gtrsim$ '' and `` $\sim$ ''. Also, if we do not specify the domain of an integral, we assume it is $\mathbb{R}^2.$ Likewise, unless specified otherwise, $L^p$ norms are considered in the whole of $\RR^2$. Note as well that we use here the convention $(a,b)^{\perp} = (b,-a)$ for all $(a,b) \in \RR^2\,.$

\section{Preliminary integral estimates} \label{S.PIT}

This section is devoted to proving some integral estimates that will be crucial throughout the paper. In particular, we prove an improved Sobolev type inequality for $k$-fold symmetric functions, which is of independent interest. See Proposition \ref{P.ImprovedSobolev} for the precise statement. 

As pointed out in the introduction, it is well known that, under suitable mild assumptions, all positive solutions to \eqref{E.Yamabe} are of the form
\begin{equation} \label{E.bubbleDeltaXi}
U_{\de,\xi}(x) := \frac{1}{\sqrt{\de}}\, U \Big( \frac{x-\xi}{\de} \Big) = \frac{\sqrt{\de}}{(\de^2 + |x-\xi|^2)^\frac12}\,, \quad \textup{for some } \de > 0 \textup{ and } \xi \in \RR^2\,,
\end{equation}
where we recall that
\begin{equation*} 
U(x) := \frac{1}{(1+|x|^2)^{\frac12}}\,.
\end{equation*}
Then, for $t > 0$, $k \in \NN$ with $k \geq 2$ and $j \in \{1, \ldots, k\}$, we define
\begin{equation} \label{E.bubblesAnsatz}
U_{j,k,t}(x) := U_{t\delta_k,\xi_{j,k,t}}(x) = \frac{\sqrt{t\delta_k}}{(t^2 \delta_k^2 + |x-\xi_{j,k,t}|^2)^{\frac12}}\,, 
\end{equation}
with 
\begin{equation} \label{E.deltaXi}
\delta_k := \frac{1}{(k\log k)^2} \quad \textup{and} \quad \xi_{j,k,t} := \sqrt{1-t^2\delta_k^2}\, \Big( \cos \big( \tfrac{2\pi(j-1)}{k} \big), \, \sin\big( \tfrac{2\pi(j-1)}{k} \big) \Big)\,.
\end{equation}

Having this notation at hand, we prove the following estimate in the spirit of \cite[Lemma A.1]{PV19}. The proof is similar to the one of \cite{PV19}, but for the benefit of the reader we provide the details. This result will allow us to simplify many calculations in the rest of the paper.

\begin{lemma} \label{L.PremoselliVetois}
Let 
\begin{equation} \label{E.Omega1}
\begin{aligned}
\Om_1 := \big\{ x \in \RR^2: |x - \xi_{1,k,t}| < |x - \xi_{j,k,t}| \textup{ for all } j \in \{2, \ldots,k\} \big\}\,.
\end{aligned}
\end{equation}
For every $\al, \be \geq 0$ such that $\al + \be \leq 4$ and all $0 < a < b$, there exist constants $C > 0$ and $k_0 \in \NN$ such that
$$
\int_{\Om_1} U^{4-\al-\be}\, U_{1,k,t}^{\al} \bigg( \sum_{j=2}^k U_{j,k,t}\bigg)^{\be} dx \leq C \de_k^{ \frac{\al + \be}{2}}  \Big( k^{\be-1} f_1(k,\al,\be) + (k \log k)^{\be} f_2(\delta_k, k,\al) \Big)\,,
$$
for all $k \geq k_0$ and all $t \in [a,b]$. Here,
$$ 
f_1(k,\al,\be)
:= \left\{
\begin{aligned}
& 1 && \textup{if } \al < 1\,,\\
& (\log k)^{\be+1} && \textup{if } \al = 1\,,\\
& k^{\al -1} (\log k)^{\be} && \textup{if } \al > 1\,,
\end{aligned}
\right.
\quad \textup{ and } \quad 
f_2(\de_k,k,\al) := \left\{
\begin{aligned}
&  k^{\al-2} && \textup{if } \al < 2\,,\\
&  |\log \de_k| && \textup{if } \al = 2\,,\\
& \de_k^{2-\al}  && \textup{if } \al > 2\,. 
\end{aligned}
\right.
$$
\end{lemma}

\begin{proof}
We divide the integral into three terms and estimate them separately. For that, we define
\begin{align*}
 \Omega_{1,1}:=&B_{\delta_k}(\xi_{1,k,t}),\qquad  \Omega_{1,2}:=\Omega_1\cap (B_{2\sqrt{1-t^2\delta_k^2}}(0)\backslash B_{\delta_k}(\xi_{1,k,t})),\qquad \Omega_{1,3}:=\Omega_1\backslash B_{2\sqrt{1-t^2\delta_k^2}}(0),
\end{align*}
so that
$$
\Omega_1=\Omega_{1,1}\cup \Omega_{1,2}\cup \Omega_{1,3}.
$$
We refer to Figure \ref{Omega1} for an illustration of the decomposition of $\Omega_1$.
\begin{figure}[h]
 \centering
\includegraphics[width=0.6\linewidth]{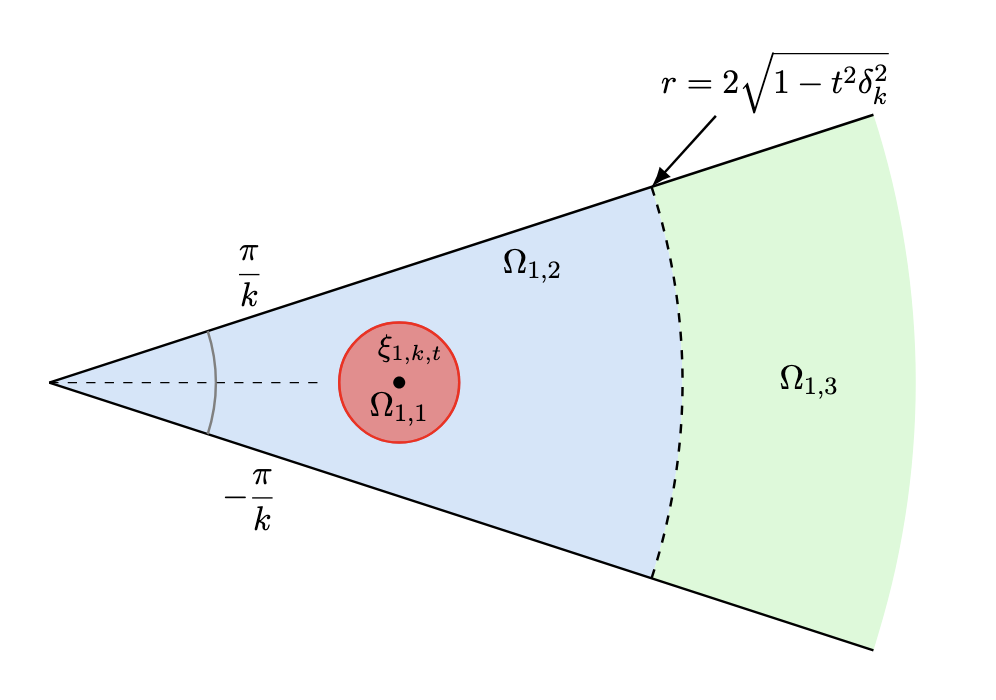}
\caption{Schematic picture of the decomposition of $\Om_1$.}
\label{Omega1}
\end{figure}

Then, the integral can be divided as
\begin{align*}
& \int_{\Om_1} U^{4-\al-\be}\, U_{1,k,t}^{\al} \bigg( \sum_{j=2}^k U_{j,k,t}\bigg)^{\be} dx=  \int_{\Om_{1,1}} U^{4-\al-\be}\, U_{1,k,t}^{\al} \bigg( \sum_{j=2}^k U_{j,k,t}\bigg)^{\be} dx\\
& \quad +\int_{\Om_{1,2}} U^{4-\al-\be}\, U_{1,k,t}^{\al} \bigg( \sum_{j=2}^k U_{j,k,t}\bigg)^{\be} dx  +\int_{\Om_{1,3}} U^{4-\al-\be}\, U_{1,k,t}^{\al} \bigg( \sum_{j=2}^k U_{j,k,t}\bigg)^{\be} dx =: I_1+I_2+I_3.
\end{align*}

We start dealing with $I_1$. Note that
\begin{align*}
    I_1 &  =  t^\frac{\alpha+\beta}{2} \delta_k^\frac{\alpha+\beta}{2} \int_{\Omega_{1,1}}\frac{1}{(1+|x|^2)^\frac{4-\alpha-\beta}{2}}\frac{1}{(t^2\delta_k^2+|x-\xi_{1,k,t}|^2)^\frac{\alpha}{2}}\bigg(\sum_{j=2}^k\frac{1}{(t^2\delta_k^2+|x-\xi_{j,k,t}|^2)^\frac12}\bigg)^{\beta} dx\\
    & \lesssim \delta_k^\frac{\beta-\alpha}{2} \int_{\Omega_{1,1}}\bigg(\sum_{j=2}^k\frac{1}{|x-\xi_{j,k,t}|}\bigg)^\beta dx.
\end{align*}
Moreover, by \cite[(A.6)]{PV19} with $n=3$, we know that
\begin{equation}\label{id-intklogk}
\sum_{j=2}^k\frac{1}{|x-\xi_{j,k,t}|} \lesssim k\log k \quad \textup{ in } \Om_1.
\end{equation}
Hence, 
\begin{equation} \label{E.I1}
    I_1\lesssim \delta_k^\frac{\beta-\alpha}{2} (k\log k)^\beta\int_{\Omega_{1,1}} dx=\delta_k^\frac{\beta-\alpha}{2} (k\log k)^\beta|\Omega_{1,1}|\lesssim\delta_k^{2+\frac{\beta-\alpha}{2}} (k\log k)^\beta.
\end{equation}

Next, we analyze $I_2$. First, we immediately see that
\begin{equation} \label{E.I2}
\begin{aligned}
    I_2 & = t^\frac{\alpha+\beta}{2} \delta_k^\frac{\alpha+\beta}{2} \int_{\Omega_{1,2}}\frac{1}{(1+|x|^2)^\frac{4-\alpha-\beta}{2}}\frac{1}{(t^2\delta_k^2+|x-\xi_{1,k,t}|^2)^\frac{\alpha}{2}}\bigg(\sum_{j=2}^k\frac{1}{(t^2\delta_k^2+|x-\xi_{j,k,t}|^2)^\frac12}\bigg)^\beta dx\\
    & \lesssim \delta_k^\frac{\alpha+\beta}{2} \int_{\Om_{1,2}}\frac{1}{|x-\xi_{1,k,t}|^\alpha}\bigg(\sum_{j=2}^k\frac{1}{|x-\xi_{j,k,t}|}\bigg)^\beta dx.
\end{aligned}
\end{equation}
Now, we define
$$
\Gamma_2:=\left[-\frac{1}{k}+\sqrt{1-t^2\delta_k^2},\ \frac{1}{k}+\sqrt{1-t^2\delta_k^2} \right]\times\RR,
$$
and separate $I_2$ into two integrals. First, using again \eqref{id-intklogk}, we see that
\begin{align*}
    I_2 & \lesssim \delta_k^\frac{\alpha+\beta}{2} \int_{\Om_{1,2}\cap \Gamma_2}\frac{1}{|x-\xi_{1,k,t}|^\alpha}\bigg(\sum_{j=2}^k\frac{1}{|x-\xi_{j,k,t}|}\bigg)^\beta dx+\delta_k^\frac{\alpha+\beta}{2} \int_{\Om_{1,2}\backslash\Gamma_2}\frac{1}{|x-\xi_{1,k,t}|^\alpha}\bigg(\sum_{j=2}^k\frac{1}{|x-\xi_{j,k,t}|}\bigg)^\beta dx \\
    & \lesssim \delta_k^\frac{\alpha+\beta}{2} (k\log k)^\beta \int_{\Om_{1,2}\cap \Gamma_2}\frac{1}{|x-\xi_{1,k,t}|^\alpha}dx+\delta_k^\frac{\alpha+\beta}{2} \int_{\Om_{1,2}\backslash\Gamma_2}\frac{1}{|x-\xi_{1,k,t}|^\alpha}\bigg(\sum_{j=2}^k\frac{1}{|x-\xi_{j,k,t}|}\bigg)^\beta dx.
\end{align*}
Also, observe that
$$
\Om_{1,2}\cap \Gamma_2\subset \Big(\left[-\frac{1}{k}+\sqrt{1-t^2\delta_k^2},\frac{1}{k}+\sqrt{1-t^2\delta_k^2} \right]\times \left[-\frac{2\sqrt{1-t^2\delta_k^2}\pi}{k},\frac{2\sqrt{1-t^2\delta_k^2}\pi}{k}\right]\Big)\backslash B_{\delta_k}(\xi_{1,k,t}),
$$
and so that
\begin{align} \label{E.I21}
    I_2 \lesssim \delta_k^\frac{\al + \be}{2} (k\log k)^\beta  \int_{B_{\frac{20}{k}}(0)\backslash B_{\frac{\delta_k}{10}}(0)} \frac{1}{|x|^\alpha}dx+\delta_k^\frac{\alpha+\beta}{2} \int_{\Om_{1,2}\backslash  \Gamma_2}\frac{1}{|x-\xi_{1,k,t}|^\alpha}\bigg(\sum_{j=2}^k\frac{1}{|x-\xi_{j,k,t}|}\bigg)^\beta dx,
\end{align}
where
\begin{equation} \label{E.I22}
    \int_{B_{\frac{20}{k}}(0)\backslash B_{\frac{\delta_k}{10}}(0)} \frac{1}{|x|^\alpha}dx\lesssim\left\{\begin{array}{ll}
        k^{\alpha-2}, &   \alpha<2,\\
        |\log \delta_k|, &  \alpha=2,\\
        \delta_k^{2-\alpha}, &\alpha>2.
    \end{array}\right.
\end{equation}

For the second integral, we will use  \cite[(A.7)]{PV19} with $n=3$, namely
\begin{equation} \label{E.sumklogabs}
\sum_{j=2}^k\frac{1}{|x-\xi_{j,k,t}|} \lesssim k\log\left(1+\frac{1}{|x-\xi_{1,k,t}|}\right) \quad \textup{in } \Om_1.
\end{equation}
Indeed, using this inequality, we find that
\begin{equation} \label{E.I23}
    \delta_k^\frac{\alpha+\beta}{2} \int_{\Om_{1,2}\backslash  \Gamma_2}\frac{1}{|x-\xi_{1,k,t}|^\alpha}\bigg(\sum_{j=2}^k\frac{1}{|x-\xi_{j,k,t}|}\bigg)^\beta dx \lesssim \de_k^{\frac{\al+\be}{2}} k^{\be-1} \int_{1/k}^1 \rho^{-\al} \left(\log\left(1+\frac{1}{\rho}\right)\right)^\beta d\rho\,,
\end{equation}
where
\begin{equation} \label{E.I24}
   \int_{1/k}^1 \rho^{-\al} \left(\log\left(1+\frac{1}{\rho}\right)\right)^\beta d\rho \lesssim\left\{\begin{array}{ll}
        1, &   \alpha<1,\\
        (\log k)^{\be+1}, &  \alpha=1,\\
        k^{\al-1} (\log k)^{\be}, &\alpha>1.
    \end{array}\right.
\end{equation}

Finally, we estimate $I_3$. Using that
$$
|x-\xi_{i,k,t}| \geq \frac{|x|}{2}\quad \textup{in } \Om_{1,3}\,, 
$$
for all $i \in \{1, \ldots,k\}$, we get that
\begin{equation} \label{E.I3}
    \begin{aligned}
        I_3 \lesssim \de_k^{\frac{\al+\be}{2}} k^\be \int_{\Om_{1,3}} \frac{1}{|x|^4} \,dx \lesssim \de_k^{\frac{\al+\be}{2}} k^\be \int_{-\frac{\pi}{k}}^{\frac{\pi}{k}} \int_1^{\infty} \rho^{-3} d\rho \lesssim \de_k^{\frac{\al+\be}{2}} k^{\be-1}\,.
    \end{aligned}
\end{equation}
The result follows combining \eqref{E.I1}, \eqref{E.I21}, \eqref{E.I22}, \eqref{E.I23}, \eqref{E.I24} and \eqref{E.I3}.
\end{proof}

The rest of this section is devoted to the improved Sobolev type inequality. We will prove it in $\dot{H}^1(\RR^3)$, but the proof can be adapted to higher dimensions. For $t > 0$, $k \in \NN$ with $k \geq 2$ and $j \in \{1, \ldots,k\}$, we define
$$
\Xi_{j,k,t} := (\xi_{j,k,t}, 0) = \sqrt{1-t^2\delta_k^2}\, \Big( \cos \big( \tfrac{2\pi(j-1)}{k} \big), \, \sin\big( \tfrac{2\pi(j-1)}{k} \big),\, 0 \Big)  \in \RR^3\,, 
$$
and let $\dot{H}^1_k(\RR^3)$ be the subset of all functions $v \in \dot{H}^1(\RR^3)$ such that $v$ is $k$-fold symmetric in $x_1, x_2$, that is,
\begin{equation} \label{E.k-foldsymmetryR3}
v(r \cos(\theta), r \sin(\theta),x_3) = v\big( r \cos(\theta + \tfrac{2\pi}{k}), r \sin(\theta + \tfrac{2\pi}{k} ),x_3\big)\,, \quad \textup{for all } r > 0, \ \theta \in [0,2\pi) \textup{ and } x_3 \in \RR\,.
\end{equation}

The improved Sobolev inequality reads as follows:

\begin{proposition} \label{P.ImprovedSobolev}
    For every $0 < a < b < \infty$, there exist $k_{\star} \in \NN$ and $\cC_{\star} > 0$ such that, for all $k \geq k_{\star}$, all $t \in [a,b]$, and all $v \in \dot{H}^1_k(\RR^3)$, it follows that
    \begin{equation}
        \frac{1}{t\de_k} \int_{B_{2\sqrt{t\de_k}}(\Xi_{1,k,t})} |v|^2 dx \leq \frac{\cC_{\star}}{k} \int_{\RR^3} |\nabla v|^2 dx = \cC_{\star} \int_{\Om_1 \times \RR} |\nabla v|^2 dx\,.
    \end{equation}
\end{proposition}

The proof of this result is rather long, so we split it into several lemmas. We start with a key quasi-additivity result for the Newtonian capacity of suitable sets in this $k$-fold setting. Here and throughout this section, we denote by ${\rm cap}(E)$ the Newtonian capacity of a compact set $E \subset \RR^3$.  

\begin{lemma} \label{L.QA}
      For every $0 < a < b < \infty$, there exist $k_{\star} \in \NN$ and $c_1> 0$ such that, for all $k \geq k_{\star}$, all $t \in [a,b]$, and all 
      $$
      K = \bigcup_{j=1}^k K_j\,,
      $$
      with $K_j \subset B_{2\sqrt{t\de_k}}(\Xi_{j,k,t})$, $j \in \{1, \ldots,k\}$, compact sets, it follows that
      $$
      {\rm cap}(K) \geq c_1 \sum_{j=1}^k {\rm cap}(K_j)\,.
      $$
\end{lemma}

\begin{proof}
    We will use the dual characterization of the Newtonian capacity: for $E \subset \RR^3$ compact
    $$
    {\rm cap}(E) = \sup_{\mu \in \cM_E^{+}} \frac{\mu(E)^2}{I(\mu)}\,,
    $$
    where $\cM_E^{+} := \{\mu: \mu$ is a positive finite Radon measure and $\supp(\mu) \subset E\}$ and
    $$
    I(\mu) := \frac{1}{4\pi}\iint_{\RR^3 \times \RR^3} \frac{1}{|x-y|} d\mu(x) d\mu(y) = \frac{1}{4\pi} \iint_{E\times E}  \frac{1}{|x-y|} d\mu(x) d\mu(y)\,.
    $$
    We refer for instance to \cite[(6.66)]{TaylorII} for this classical characterization. Actually, let us explicitly point out that it is not exactly the same characterization as in \cite[(6.66)]{TaylorII}. However, doing a simple scaling argument, one can check they are equivalent.

    Now, for each $K_j$, we choose a measure $\mu_j \in \cM_{K_j}^{+}$ such that
    $$
    \mu_j(K_j) = {\rm cap}(K_j) \quad \textup{and} \quad I(\mu_j) \leq 2 {\rm cap}(K_j)\,,
    $$
    and set
    $$
    \mu := \sum_{j=1}^k \mu_j \in \cM_{K}^{+}\,,
    $$
    so that
    \begin{equation} \label{E.QA1}
        \mu(K) = \sum_{j=1}^k \mu_j(K_j) = \sum_{j=1}^k {\rm cap}(K_j)\,.
    \end{equation}
    Note that here we are using that $K_i \cap K_j = \emptyset$ for all $i \neq j$. For $k$ sufficiently large, this immediately follows from the definition of $\Xi_{i,k,t}$, $\Xi_{j,k,t}$ and the choice of $\de_k$. 
    Moreover, observe that
    \begin{align*}
       4\pi I(\mu) & = \int_K \int_K \frac{1}{|x-y|} d \mu(x) d \mu(y) = \Sjk \sum_{i=1}^k \int_{K_j} \int_{K_i} \frac{1}{|x-y|} d \mu_j(x) d \mu_i(y) \\
       & = 4\pi\Sjk I(\mu_j) + \Sjk \sum_{i \neq j}  \int_{K_j} \int_{K_i} \frac{1}{|x-y|} d \mu_j(x) d \mu_i(y) \\
       & \leq 8\pi \Sjk {\rm cap}(K_j) + \Sjk \sum_{i \neq j}  \int_{K_j} \int_{K_i} \frac{1}{|x-y|} d \mu_j(x) d \mu_i(y)\,.
    \end{align*}
    Also, using again the definition of $\Xi_{i,k,t}$, $\Xi_{j,k,t
    }$ and the choice of $\de_k$, we infer that, for $k$ sufficiently large, 
    $$
    \int_{K_j} \int_{K_i} \frac{1}{|x-y|} d \mu_j(x) d \mu_i(y) \leq \frac{100}{|\Xi_{i,k,t}-\Xi_{j,k,t}|} \int_{K_j} \int_{K_i} d\mu_j(x) d\mu_i(y) = \frac{100 \,{\rm cap}(K_j) {\rm cap}(K_i)}{|\Xi_{i,k,t}-\Xi_{j,k,t}|}\,.
    $$
    Hence, using that 
    $$
    \sum_{i \neq j} \frac{1}{|\Xi_{i,k,t}-\Xi_{j,k,t}|} \lesssim k \log k \quad \textup{and} \quad 
    {\rm cap} ( B_{2\sqrt{t\de_k}} (\Xi_{i,k,t})) \lesssim \sqrt{\de_k}\,,
    $$
    we get that
    \begin{equation} \label{E.QA2}
    \begin{aligned}
        4\pi I(\mu) & \lesssim \sum_{j=1}^k {\rm cap} (K_j) + \Sjk {\rm cap}(K_j) \sum_{i \neq j} \frac{{\rm cap}(K_i)}{|\Xi_{i,k,t}-\Xi_{j,k,t}|} \\
        & \leq \sum_{j=1}^k {\rm cap} (K_j) + \Sjk {\rm cap}(K_j) \sum_{i \neq j} \frac{{\rm cap}(B_{2\sqrt{t\de_k}} (\Xi_{i,k,t}))}{|\Xi_{i,k,t}-\Xi_{j,k,t}|} \\
        & \lesssim \sum_{j=1}^k {\rm cap} (K_j) + \sqrt{\de_k} \Sjk {\rm cap}(K_j) \sum_{i \neq j} \frac{1}{|\Xi_{i,k,t}-\Xi_{j,k,t}|} \\
        & \lesssim   \sum_{j=1}^k {\rm cap} (K_j) + \sqrt{\de_k}\, k \log k \Sjk {\rm cap}(K_j) \lesssim \Sjk {\rm cap}(K_j)\,.
    \end{aligned}
    \end{equation}

    Finally, combining \eqref{E.QA1} and \eqref{E.QA2}, we conclude that
    $$ 
    {\rm cap}(K) \geq \frac{\mu(K)^2}{I(\mu)} \gtrsim \frac{ \left( \Sjk {\rm cap}(K_j) \right)^2}{\Sjk {\rm cap}(K_j)} = \Sjk {\rm cap}(K_j)\,,
    $$
    as desired. The proof is concluded.
\end{proof}

Having the previous lemma at hand, one can easily prove the following:

\begin{lemma} \label{L.QAcorollary}
    For every $0 < a < b < \infty$, let $k_\star \in \NN$ be as in Lemma \ref{L.QA}. There exists $c_2 > 0$  such that, for all $k \geq k_\star$, $t \in [a,b]$ and compact $K_1 \subset B_{2\sqrt{t\de_k}}(\Xi_{1,k,t})$, if we set $K_j := R_{\frac{2\pi(j-1)}{k}} K_1$ for $j \in \{1, \ldots,k\}$ and
    $$
    K := \bigcup_{j=1}^k K_j\,,
    $$
    then
    $$
    |K| \leq c_2 \de_k {\rm cap}(K)\,.
    $$
    Here, $R_\te$ denotes the rotation matrix of angle $\te$ around the $x_3$-axis. 
\end{lemma}

\begin{proof}
    First note that $|K| = k |K_1|$ and that ${\rm cap}(K_j) = {\rm cap}(K_1)$ for all $j \in \{2, \ldots,k\}$. Moreover, taking into account the variational characterization of the Newtonian capacity, see for instance \cite[(6.64)]{TaylorII}, and using Sobolev's inequality, one can easily see that
    $$
    |E|^{\frac13} \lesssim {\rm cap}(E)\,, \quad \textup{for every compact set } E \subset \RR^3\,.
    $$
    Thus, using Lemma \ref{L.QA}, we get that
    $$
    |K|^{\frac13} = k^{\frac13} |K_1|^{\frac13} \lesssim k^{\frac13} {\rm cap}(K_1) = k^{\frac13} \frac{1}{k} \Sjk {\rm cap}(K_j) \lesssim \frac{1}{k^{\frac23}} {\rm cap}(K)\,.
    $$
    On the other hand, it is straightforward to check that
    $$
    |K| = k |K_1| \leq k |B_{2\sqrt{t\de_k}}(\Xi_{1,k,t})| \lesssim k \de_k^{\frac32}\,.
    $$
    Combining both inequalities we conclude that
    $$
    |K| \lesssim  \frac{1}{k^{\frac23}} {\rm cap}(K) k^{\frac23} \de_k = \de_k {\rm cap}(K)\,.
    $$
\end{proof}

At this point, we can use a capacitary strong type inequality to conclude the proof of Proposition \ref{P.ImprovedSobolev}.

\begin{proof}[Proof of Proposition \ref{P.ImprovedSobolev}]
    First observe that, by density, it is enough to prove the result for $v \in C_c^{\infty}(\RR^3)$ satisfying \eqref{E.k-foldsymmetryR3}. Hence, let $v \in C_c^{\infty}(\RR^3)$ satisfying \eqref{E.k-foldsymmetryR3} be fixed but arbitrary. We set
    $$
    \cB_k := \bigcup_{j=1}^k B_{2\sqrt{t\de_k}}(\Xi_{j,k,t})\,,
    $$
    and observe that
    \begin{align*}
        k \int_{B_{2\sqrt{t\de_k}}(\Xi_{1,k,t})} |v|^2 dx = \int_{\cB_k} |v|^2 dx = 2 \int_0^{\infty} s \,|\{|v| \geq s\} \cap \cB_k | \,ds\,.
    \end{align*}
    Thus, using Lemma \ref{L.QAcorollary}, we get that
    \begin{align*}
        k \int_{B_{2\sqrt{t\de_k}}(\Xi_{1,k,t})} |v|^2 dx \lesssim \de_k \int_0^{\infty} s\, {\rm cap}(\{|v| \geq s\} \cap \cB_k ) ds \leq \de_k \int_0^{\infty} s\, {\rm cap}(\{|v| \geq s\} ) ds \,.
    \end{align*}

    On the other hand, by \cite[Sect. 2.3.1, Theorem]{Mazya}, we know that
    $$
    \int_0^{\infty}  s\, {\rm cap}(\{|v| \geq s\} ) ds \leq 2 \int_{\RR^3} |\nabla v|^2 dx \,.
    $$
    Hence, combining this inequality with the one above, we get that
    $$
    \frac{1}{t\de_k} \int_{B_{2\sqrt{t\de_k}}(\Xi_{1,k,t})} |v|^2 dx \lesssim \frac{1}{k} \int_{\RR^3} |\nabla v|^2 dx =  \int_{\Om_1 \times \RR} |\nabla v|^2 dx\,.
    $$
    The proof is concluded. 
\end{proof}

We set $\RR_+^3 := \{x \in \RR^3: x_3 > 0\}$ and denote by $\dot{H}_k^1(\RR_+^3)$ the subset of all functions $u \in \dot{H}^1(\RR_+^3)$ such that $u$ satisfies \eqref{E.k-foldsymmetryR3}. Likewise, for every $R > 0$ and $\Xi \in \RR^3$, we use the notation $B_R^+(\Xi) := B_R(\Xi) \cap \RR_+^3.$ An immediate (but useful for our purposes) corollary of the previous proposition is the following:

\begin{corollary} \label{C.improvedSobolev}
    For every $0 < a < b < \infty$, there exist $k_{\star} \in \NN$ and $\cC_{\star} > 0$ such that, for all $k \geq k_{\star}$, all $t \in [a,b]$, and all $v \in \dot{H}^1_k(\RR^3_+)$, it follows that
    \begin{equation}
        \frac{1}{t\de_k} \int_{B_{2\sqrt{t\de_k}}^+(\Xi_{1,k,t})} |v|^2 dx \leq \frac{\cC_{\star}}{k} \int_{\RR^3_+} |\nabla v|^2 dx = \cC_{\star} \int_{\Om_1 \times (0,\infty)} |\nabla v|^2 dx\,.
    \end{equation}
\end{corollary}

\begin{proof}
    For every $v \in \dot{H}^1_k(\RR^3_+)$  we denote by $v_\sharp$ the even extension of $v$ to $\RR^3$, that is
    $$
    v_\sharp(x_1,x_2,x_3) = v(x_1,x_2,|x_3|) \quad \textup{for all } x \in \RR^3\,.
    $$
    Then, it follows that $v_\sharp \in \dot{H}^1_k(\RR^3)$ and that
    $$
    \|\nabla v\|_{L^2(\RR_+^3)}^2 = \frac12 \|\nabla v_\sharp\|_{L^2(\RR^3)}^2 \quad \textup{and} \quad  \int_{B_{2\sqrt{t\de_k}}^+(\Xi_{1,k,t})} |v|^2 dx = \frac12  \int_{B_{2\sqrt{t\de_k}}(\Xi_{1,k,t})} |v_\sharp|^2 dx\,.
    $$
    The result immediately follows from Proposition \ref{P.ImprovedSobolev}.
\end{proof}

\section{The half-Yamabe theorem. Proof of Theorem \ref{T.mainThmYamabe}} \label{S.2}

Having the notation introduced at the beginning of Section \ref{S.PIT} at hand, see \eqref{E.bubbleDeltaXi}--\eqref{E.deltaXi}, we can readily present our approximate solution
\begin{equation} \label{E.ansatz}
U_{k,t}(x) := \sum_{j=1}^k U_{j,k,t}(x) - U(x)\,.
\end{equation}
We refer to Figures \ref{k5} and \ref{k12} for a representation of $U_{k,t}$ and its level curves when $k = 5$ and $k = 12$ respectively. Actually, this is how the solutions we will construct should look like.

\begin{figure}[h]
 \centering
\begin{subfigure}{.5\textwidth}
  \centering
  \includegraphics[width=1.3\linewidth]{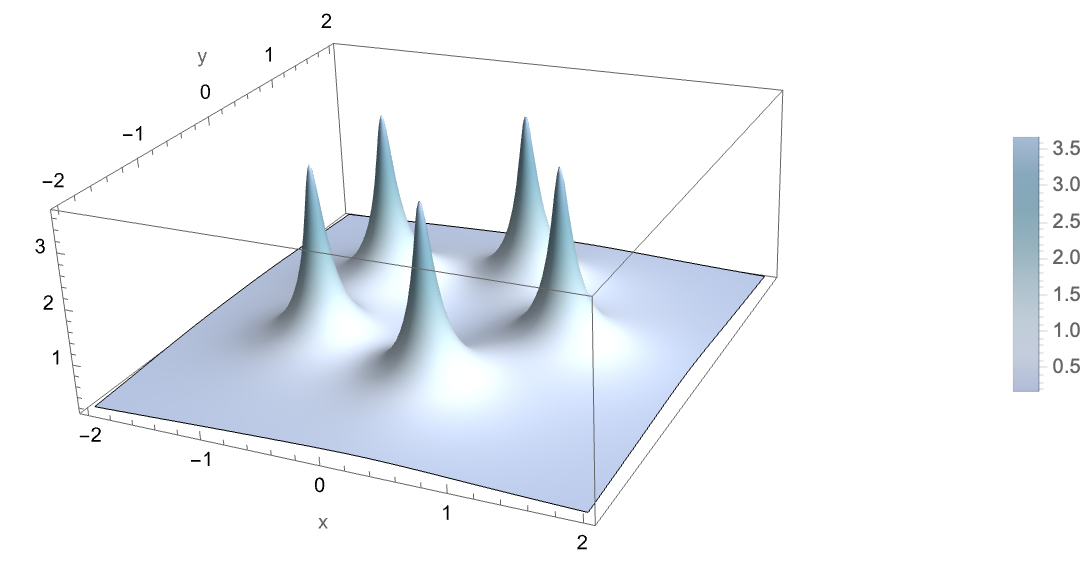}
\end{subfigure}%
\begin{subfigure}{.5\textwidth}
  \centering
  \includegraphics[width=1\linewidth]{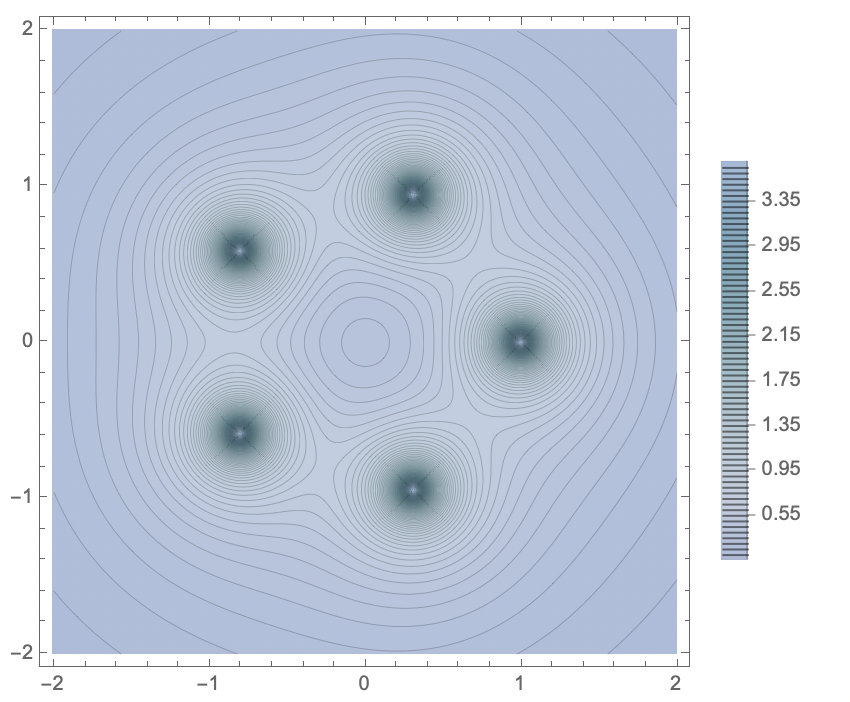}
\end{subfigure}
\caption{Schematic picture of $U_{k,t}$ and its level curves when $k = 5.$}
\label{k5}
\end{figure}

\begin{figure}[h]
 \centering
\begin{subfigure}{.5\textwidth}
  \centering
\includegraphics[width=1\linewidth]{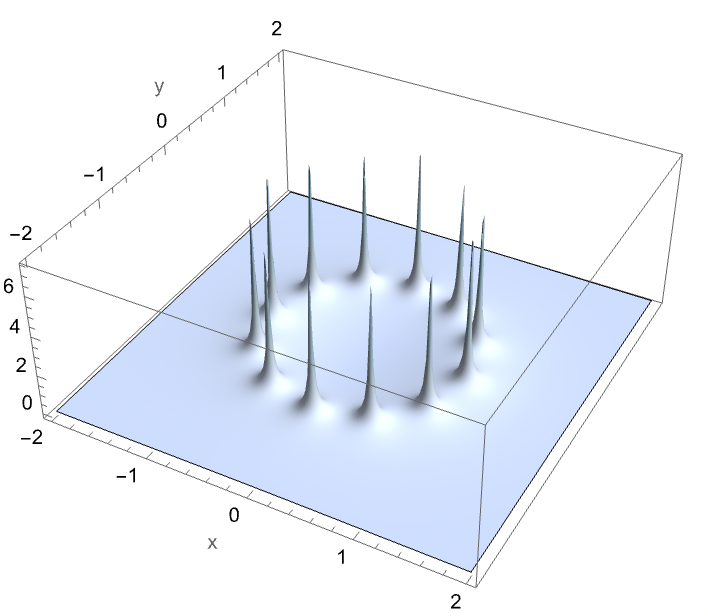}
\end{subfigure}%
\begin{subfigure}{.5\textwidth}
  \centering
\includegraphics[width=1\linewidth]{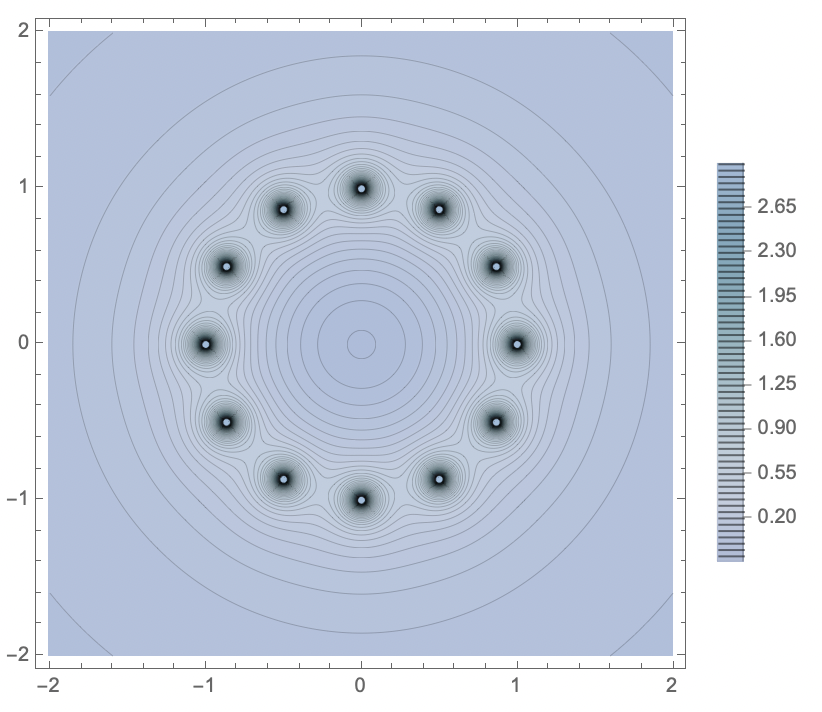}
\end{subfigure}
\caption{Schematic picture of $U_{k,t}$ and its level curves when $k = 12.$}
\label{k12}
\end{figure}

As a first step towards the proof of Theorem \ref{T.mainThmYamabe}, we will measure how well this approximate solution ``solves'' \eqref{E.Yamabe}, but first we are going to reformulate \eqref{E.Yamabe} using the inverse of the half-Laplacian. By Sobolev inequality, the embedding $\dot{H}^{\frac12} \hookrightarrow L^4$ is continuous. Hence, using the Riesz representation theorem, we can define the continuous operator
\begin{align}\label{43H12}
(-\Delta)^{-\frac12} : L^{\frac43} \to \dot{H}^{\frac12}\,, \quad f \mapsto u\,,
\end{align}
where $u \in \dot{H}^{\frac12}$ is the unique solution to
$$
(-\Delta)^{\frac12} u = f \quad \textup{in } \RR^2\,.
$$
Actually, we explicitly have that
\begin{align*}
    (-\Delta)^{-\frac12} f(x) :=\frac{1}{2\pi}\int_{\mathbb{R}^2}\frac{f(y)}{|x-y|}dy\,. 
\end{align*}

Then, we can reformulate \eqref{E.Yamabe} as
\begin{equation} \label{E.YamabeInverse}
u = (-\Delta)^{-\frac12} (u^3) \quad \textup{in } \RR^2\,.
\end{equation}
Having our approximate solution $U_{k,t}$ and this reformulation at hand, for $k \in \NN$ with $k \geq 2$ and $t > 0$, we set
\begin{equation}\label{E.error}
\cE_{k,t} := U_{k,t}^3 - (-\Delta)^{\frac12} U_{k,t} \quad \textup{and} \quad \cE_{k,t}^{*} := (-\Delta)^{-\frac12} \cE_{k,t}\,,
\end{equation}
and prove the following.


\begin{lemma} \label{L.error}
For every $0 < a < b < \infty$, there exists a constant $\cC_0 > 0$ such that
\begin{equation}
\|\cE_{k,t}^{*}\|_{\dot{H}^{\frac12}} \leq \frac{\cC_0}{\log k}\,,
\end{equation}
for all $k \geq 2$ and all $t \in [a,b]$. 
\end{lemma}

\begin{proof}
We first notice that
\begin{align*}
\|\cE^*_{k,t}\|_{\Hud}=\|(-\Delta)^{-\frac{1}{2}}\cE_{k,t}\|_{\Hud}=\|(-\Delta)^{-\frac{1}{4}}\cE_{k,t}\|_{L^2}\lesssim \|\cE_{k,t}\|_{L^\frac{4}{3}}.
\end{align*}
Hence, it suffices to show that
$$
\|\cE_{k,t}\|_{L^{\frac43}} \lesssim \frac{1}{\log k}\,.
$$
By definition and using that \eqref{E.bubbleDeltaXi} solves \eqref{E.Yamabe}, we have that
\begin{align*}
    \cE_{k,t} & =U_{k,t}^3-\Dud U_{k,t}=\bigg(\sum_{j=1}^k U_{j,k,t}-U\bigg)^3 - \sum_{j=1}^k U_{j,k,t}^3 + U^3 \\
    & =\underbrace{\bigg(\Sjk U_{j,k,t}\bigg)^3 - \sum_{j=1}^k U_{j,k,t}^3}_{\cE_{k,t}^{(1)}}-\underbrace{3U\bigg(\Sjk U_{j,k,t}\bigg)^2}_{\cE_{k,t}^{(2)}}+\underbrace{3U^2\Sjk U_{j,k,t}}_{\cE_{k,t}^{(3)}}.
\end{align*}
Moreover, we denote by $\Om_1$ the sector
\begin{align*}
\Om_1  & := \big\{(r\cos(\theta),\,r\sin(\theta))\in \mathbb{R}^2\,:\, r\in \mathbb{R}^+,\,\theta\in (-\tfrac{\pi}{k},\tfrac{\pi}{k})\big\} \\
& \ =\big\{ x\in \mathbb{R}^2\,:\, |x-\xi_{1,k,t}| < |x-\xi_{j,k,t}|,\,\forall j\in\{2,...,k\}\big\}.\end{align*}
Thus, from the symmetry of $\cE_{k,t}$, it follows that
\begin{equation} \label{E.errorSector}
    \| \cE_{k,t}\|_{L^\frac{4}{3}}= k^\frac{3}{4}\| \cE_{k,t}\|_{L^\frac{4}{3}(\Omega_1)},
\end{equation}
and that
\begin{equation} \label{E.errorDecomposition}
\|\cE_{k,t}\|_{L^\frac{4}{3}(\Omega_1)}\leq\|\cE^{(1)}_{k,t}\|_{L^\frac{4}{3}(\Omega_1)}+\|\cE^{(2)}_{k,t}\|_{L^\frac{4}{3}(\Omega_1)}+\|\cE^{(3)}_{k,t}\|_{L^\frac{4}{3}(\Omega_1)}. 
\end{equation}

We estimate the terms $\cE^{(i)}_{k,t}$ for $i \in \{1,2,3\}$ separately using Lemma \ref{L.PremoselliVetois} appropriately. We first deal with $\cE^{(3)}_{k,t}$. Note that
\begin{align*}
    \|\cE^{(3)}_{k,t}\|^\frac{4}{3}_{L^\frac{4}{3}(\Omega_1)}\lesssim&\int_{\Omega_1}U^\frac{8}{3}\bigg(\Sjk U_{j,k,t}\bigg)^\frac{4}{3} dx
    \lesssim  \int_{\Omega_1}U^\frac{8}{3}U_{1,k,t}^\frac{4}{3}dx +\int_{\Omega_1}U^\frac{8}{3}\bigg(\sum_{j=2}^k U_{j,k,t}\bigg)^\frac{4}{3}dx.
\end{align*}
For the first integral we use Lemma \ref{L.PremoselliVetois} with $\alpha=4/3$ and $\beta=0$ to get
\begin{align*}
    \int_{\Omega_1}U^\frac{8}{3}U_{1,k,t}^\frac{4}{3}dx\lesssim (\delta_k)^\frac{2}{3}(k^{-1}+k^{-\frac{2}{3}})\lesssim \frac{1}{k^2\log^\frac{4}{3}(k)}.
\end{align*}
For the second one we use Lemma \ref{L.PremoselliVetois} with $\alpha=0$ and $\beta=4/3$, which yields
\begin{align*}
\int_{\Omega_1}U^\frac{8}{3}\bigg(\sum_{j=2}^k U_{j,k,t}\bigg)^\frac{4}{3}dx \lesssim \delta_k ^\frac{2}{3}(k^\frac{1}{3}+(k\log(k))^\frac{4}{3}k^{-2})\lesssim \frac{1}{k\log^\frac{4}{3}(k)} .
\end{align*}
Therefore, it follows that
\begin{equation} \label{E.Error3}
 \|\cE^{(3)}_{k,t}\|_{L^\frac{4}{3}(\Omega_1)}\lesssim\frac{1}{k^\frac{3}{4}\log(k)}. 
\end{equation}

Next, we consider $\cE^{(2)}_{k,t}$. Note that
\begin{align*}
    \|\cE_{k,t}^{(2)}\|_{L^\frac{4}{3}(\Omega_1)}^\frac{4}{3}\lesssim \int_{\Omega_1}U^\frac{4}{3}\bigg(\Sjk U_{j,k,t}\bigg)^\frac{8}{3}dx\lesssim \int_{\Omega_1} U^\frac{4}{3}U_{1,k,t}^\frac{8}{3}dx+\int_{\Omega_1}U^\frac{4}{3}
\bigg(\sum_{j=2}^kU_{j,k,t}\bigg)^\frac{8}{3}dx.
\end{align*}
For the first integral we use Lemma \ref{L.PremoselliVetois} with $\alpha=8/3$ and $\beta=0$ to get
\begin{align*}
    \int_{\Omega_1}U^\frac{4}{3}U_{1,k,t}^\frac{8}{3}dx\ls \delta_k^\frac{4}{3}(k^{-\frac{1}{3}}+\delta_k^{-\frac{2}{3}})\ls \delta_k^\frac{2}{3}\ls
    \frac{1}{k^\frac{4}{3}\log^\frac{4}{3}(k)}.
\end{align*}
For the second we use Lemma \ref{L.PremoselliVetois} with $\alpha=0$ and $\beta=8/3$. Thus, we get that
\begin{align*}
    \int_{\Omega_1}U^\frac{4}{3}
\bigg(\sum_{j=2}^kU_{j,k,t}\bigg)^\frac{8}{3}dx\sim \delta_k^\frac{4}{3}\left(k^\frac{5}{3}+(k\log(k))^\frac{8}{3}k^{-2}\right)\ls \delta_k^\frac{4}{3}k^\frac{5}{3}\ls \frac{1}{k\log^\frac{8}{3}(k)},
\end{align*}
and so that
\begin{equation} \label{E.error2}
    \|\cE^{(2)}_{k,t}\|_{L^\frac{4}{3}(\Omega_1)}\lesssim \frac{1}{k^\frac{3}{4}\log^2(k)}. 
\end{equation}
 
Finally, for $\cE_{k,t}^{(1)}$, we first note that
\begin{align*}
    \bigg(\Sjk U_{j,k,t}\bigg)^3& =\bigg(U_{1,k,t}+\sum_{j=2}^k U_{j,k,t}\bigg)^3 \\
    & =U_{1,k,t}^3+\bigg(\sum_{j=2}^k U_{j,k,t}\bigg)^3+3U_{1,k,t}^2\bigg(\sum_{j=2}^k U_{j,k,t}\bigg)+3U_{1,k,t}\bigg(\sum_{j=2}^k U_{j,k,t}\bigg)^2,
\end{align*}
and so that
\begin{align*}
    \bigg(\Sjk U_{j,k,t}\bigg)^3-\Sjk U_{j,k,t}^3=\bigg(\sum_{j=2}^k U_{j,k,t}\bigg)^3+3U_{1,k,t}^2\bigg(\sum_{j=2}^k U_{j,k,t}\bigg)+3U_{1,k,t}\bigg(\sum_{j=2}^k U_{j,k,t}\bigg)^2-\sum_{j=2}^k U_{j,k,t}^3\,.
\end{align*}
Hence, it follows that
\begin{align*}
    \|\cE_{k,t}^{(1)}\|_{L^{\frac43}(\Om_1)}^{\frac43} \lesssim  &\ \int_{\Omega_1}\bigg(\sum_{j=2}^k U_{j,k,t}\bigg)^4 dx +\int_{\Omega_1} U_{1,k,t}^\frac{8}{3}\bigg(\sum_{j=2}^k U_{j,k,t}\bigg)^\frac{4}{3} dx \\
    & +\int_{\Omega_1}U_{1,k,t}^\frac{4}{3} \bigg(\sum_{j=2}^kU_{j,k,t}\bigg)^\frac{8}{3}dx+\int_{\Omega_1}\bigg(\sum_{j=2}^k U_{j,k,t}^3\bigg)^\frac{4}{3} dx.
\end{align*}
For the first integral we use Lemma \ref{L.PremoselliVetois} with $\alpha=0$ and $\beta=4,$ and get
\begin{align*}
    \int_{\Omega_1}\bigg(\sum_{j=2}^k U_{j,k,t}\bigg)^4 dx\ls \delta_k^2\left(k^3+(k\log(k))^4k^{-2}\right)\ls \frac{1}{k\log^4(k)}.
\end{align*}
For the second we use Lemma \ref{L.PremoselliVetois} with $\alpha=8/3$ and $\beta=4/3.$ We then get that
\begin{align*}
    \int_{\Omega_1} U_{1,k,t}^\frac{8}{3}\bigg(\sum_{j=2}^k U_{j,k,t}\bigg)^\frac{4}{3}dx\ls \delta_k^2(k \log^\frac{4}{3}(k)+k^\frac{4}{3}\log^\frac{4}{3}(k)\delta_k^{-\frac{2}{3}})\ls \delta_k^\frac{4}{3}k^\frac{4}{3}\log^\frac{4}{3}(k) = \frac{1}{k^\frac{4}{3}\log^\frac{4}{3}(k)}.
\end{align*}
For the third one we use Lemma \ref{L.PremoselliVetois} with $\alpha=4/3$ and $\beta=8/3,$ which yields
\begin{align*}
\int_{\Omega_1}U_{1,k,t}^\frac{4}{3} \bigg(\sum_{j=2}^kU_{j,k,t}\bigg)^\frac{8}{3}dx\leq \delta_k^2( k^\frac{5}{3}+k^\frac{8}{3}\log^\frac{8}{3}(k)k^{-\frac{2}{3}})\ls \delta_k^2 k^2\log^\frac{8}{3}(k) = \frac{1}{k^2\log^\frac{4}{3}(k)}.
\end{align*}

Finally, since the $U_{j,k,t}$ are positive functions, we can use the estimate of the first integral for the fourth one. Therefore, it follows that
\begin{equation} \label{E.error1}
    \|\cE_{k,t}^{(1)}\|_{L^\frac{4}{3}(\Omega_1)}\ls \frac{1}{k^\frac{3}{4}\log^3(k)}. 
\end{equation}
Taking into account \eqref{E.errorSector} and \eqref{E.errorDecomposition}, the result follows from \eqref{E.Error3}, \eqref{E.error2} and \eqref{E.error1}.
\end{proof}

Next, we define the function spaces for the perturbations $\phi$ of the approximate solution $U_{k,t}$. For every $k \in \NN$ with $k \geq 2$, we let $\cH_k$ be the subset of $\phi \in \dot{H}^{\frac12}$ such that:
\begin{itemize}
\item[(i)] $\phi$ is even with respect to $x_2$,
\item[(ii)] $\phi$ is $k$-fold symmetric,
\item[(iii)] $\phi$ is invariant under the action of the Kelvin transform $\bK$.
\end{itemize}
More explicitly, we introduce the symmetries
\begin{align}
& \phi(x_1,x_2) = \phi(x_1,-x_2) \,, \quad \textup{for all } x = (x_1,x_2) \in \RR^2\,, \label{E.even}\\
& \phi(r \cos(\theta), r \sin(\theta)) = \phi\big( r \cos(\theta + \tfrac{2\pi}{k}), r \sin(\theta + \tfrac{2\pi}{k} )\big)\,, \quad \textup{for all } r > 0, \ \theta \in [0,2\pi)\,, \label{E.kfold} \\ 
& \phi(x) = \bK[\phi](x) := \tfrac{1}{|x|} \phi \Big( \tfrac{x}{|x|^2} \Big)\,, \quad \textup{for all } x \in \RR^2 \setminus \{0\}\,, \label{E.kelvin}
\end{align}
and set
\begin{equation}
\cH_k := \big\{ \phi \in \dot{H}^{\frac12} : \phi \textup{ satisfies  \eqref{E.even}, \eqref{E.kfold} and \eqref{E.kelvin}} \big\}\,.
\end{equation} 
Note that $\cH_k$ is a Hilbert space endowed with the same scalar product and the same norm as $\dot{H}^{\frac12}$. Also, let us stress that $U_{k,t} \in \cH_k$ for all $t > 0$ and all $k \in \NN$ with $k \geq 2$.

Let us list some properties of the Kelvin transform which will be used in the paper:
\begin{enumerate}
    \item $\bK^2=\mathrm{Id}.$
    \item If $f$ is $k-$fold symmetric and even with respect to $x_2$, then so is $\bK[f]$.
    \item Let $f\in C^\infty_c(\RR^2)$. Then 
    \begin{align*}(-\Delta)^\frac{1}{2}\bK [f](x)=\ &|x|^{-3}(-\Delta )^\frac{1}{2}f\left(\frac{x}{|x|^2}\right),\\
    \bK[(-\Delta)^{-\frac{1}{2}} f](x)=\ & (-\Delta)^{-\frac{1}{2}}\left(|\cdot|^{-2}\bK[f](\cdot)\right)(x).\end{align*}
    \item Let $f$ and $g\in \dot{H}^\frac{1}{2}$. Then 
    \begin{align*}
        \langle f, \bK[g]\rangle=\langle \bK[f], g\rangle, 
    \end{align*}
    and $\bK$ is an isometry in $\Hud$. Here $\langle \, \cdot, \cdot \rangle$ denotes the scalar product in $\Hud$.
    \item If $f\in \cH_k$ and $g\in \Hud$ then
    \begin{align*}
        \langle f, \bK[g]\rangle =\langle f, g\rangle.
    \end{align*}
\end{enumerate}

Then, we look for a solution to \eqref{E.Yamabe} which is a perturbation of $U_{k,t}$:
\begin{equation}\label{E.solutionForm}
\psi_k = U_{k,t} + \phi\,,
\end{equation}
with $\phi \in \cH_k$ for some $k \in \NN$ sufficiently large. It is not difficult to see that $U_{k,t}\in \cH_k$. To check that it is in $\Hud$ we  notice that $U\in L^4(\RR^2)$ and that the scaling $\delta^{-\frac{1}{2}}U(x\delta^{-1})$ is invariant in $\Hud$. The reason why we have introduced the factor $\sqrt{1-t^2\delta_k^2}$ in the vectors $\xi_{j,k,t}$ is to get that $U_{k,t}$ is Kelvin invariant. The $k-$fold and reflection symmetries follow from the construction.

Thus, \eqref{E.Yamabe} is reduced to 
\begin{equation}\label{E.YamabePhi}
\cL_{k,t} \phi = \cE_{k,t} + \cN_{k,t} (\phi)\,, \quad \phi \in \cH_k\,,
\end{equation}
where
\begin{equation*} 
\cL_{k,t} \phi := (-\Delta)^{\frac12} \phi - 3 U_{k,t}^2 \phi \quad \textup{and} \quad \cN_{k,t}(\phi) := 3 U_{k,t} \phi^2 + \phi^3\,. 
\end{equation*}
Since we want to use the inverse formulation given in \eqref{E.YamabeInverse} instead of \eqref{E.Yamabe}, we set
$$
\cL_{k,t}^{*}: \cH_k \to \cH_k, \quad \phi \mapsto \phi - (-\Delta)^{-\frac12} (3 U_{k,t}^2 \phi) 
$$
and
$$
\cN_{k,t}^{*}: \cH_k \to \cH_k, \quad    \phi \mapsto (-\Delta)^{-\frac12} (\cN_{k,t}(\phi))\,,
$$
and reformulate \eqref{E.YamabePhi} as
\begin{equation} \label{E.YamabePhiInverse}
\cL_{k,t}^{*} \phi = \cE_{k,t}^{*} + \cN_{k,t}^{*} (\phi)\,, \quad \phi \in \cH_k\,.
\end{equation}

The fact that $\cL_{k,t}^{*}$  and $\cN_{k,t}^{*}$ map $\cH_k$ to $\cH_k$ follow from \eqref{43H12}, the embedding $\Hud \subset L^4$ and the property iii) of the Kelvin transform. Again the $k-$fold and reflection symmetries are easy to check.

The general idea of the proof would be to invert the linear operator $\cL_{k,t}^{*}$, apply its inverse in \eqref{E.YamabePhiInverse} and then solve the equation via a fixed point argument. However, that is not possible in our case. As already pointed out in the introduction, this operator has a nonempty approximated kernel as $k \to \infty$. Consequently, one cannot get a uniform estimate of the form
$$
\|\phi\|_{\Hud}\lesssim \|\cL_{k,t}^{*} \phi\|_{\Hud}\,, \quad \textup{for all } \phi \in \cH_k\,,
$$
and thus one can not find a uniformly bounded inverse in $k$. If one wants to consider $k$ sufficiently large, as we want, this will prevent any fixed point argument from working. To overcome this, we will use a Lyapunov-Schmidt reduction argument, that is, decompose the equation \eqref{E.YamabePhiInverse} into a system of two equations by projecting in some appropriate space. 

By  \cite[Theorem 1.1]{DdPS13} and \cite[Lemma A.1]{CKW25} we know that the set of $\dot{H}^{\frac12}$-solutions to 
$$
(-\Delta)^{\frac12} \phi - 3U^2 \phi = 0 \quad \textup{in } \RR^2\,,
$$
is  ${\rm span}\{Z^{(0)},\, Z^{(1)},\, Z^{(2)}\}$ where
\begin{equation} \label{E.Z012}
Z^{(0)}(x) = \frac12 \frac{|x|^2-1}{(1+|x|^2)^{\frac32}} \quad \textup{and} \quad Z^{(i)}(x) =  \frac{x_i}{(1+|x|^2)^{\frac32}} \quad \textup{for }i \in \{1,2\}\,.
\end{equation}
For $\delta > 0$ and $\xi \in \RR^2$, we set
$$
Z^{(0)}_{\de,\xi}(x):= \frac{1}{\sqrt{\delta}} Z^{(0)}\Big( \frac{x-\xi}{\de} \Big)\,,
$$
and, for $t > 0$, $k \in \NN$ with $k \geq 2$ and $j \in \{1, \ldots,k\}$, we define
\begin{equation}\label{E.Zkt}
Z_{k,t}(x) := \sum_{j=1}^k Z_{t\delta_k,\xi_{j,k,t}}^{(0)} (x) = \frac{\sqrt{t\delta_k}}{2} \sum_{j=1}^k \frac{|x-\xi_{j,k,t}|^2-t^2\de_k^2}{(t^2\de_k^2+|x-\xi_{j,k,t}|^2)^{\frac32}} \,.
\end{equation}
For later purposes, we also introduce the notation
$$
Z_{j,k,t}(x) := Z_{t\de_k, \xi_{j,k,t}}^{(0)}(x)\,.
$$
\noindent Before going any further, let us stress that $Z_{k,t}\in \Hud$ is $k-$fold symmetric and even with respect to $x_2$ but is  neither invariant under the action of the Kelvin transform $\bK$ nor anti-invariant. Indeed, we have the following result, which states that $Z_{k,t}$ is almost Kelvin invariant:

\begin{lemma} \label{KZkt} For $t > 0$ and $k \in \NN$ with $k \geq 2$, the Kelvin transform of $Z_{k,t}$ is given by 
\begin{align*}\bK[Z_{k,t}](x)= Z_{k,t}(x) + (t\de_k)^{\frac52} (1-|x|^2) \Sjk \frac{1}{\left((t\delta_k)^2+|x-\xi_{j,k,t}|^2\right)^\frac{3}{2}}.
\end{align*}
\end{lemma}
\begin{proof}
The proof is just direct calculus. Indeed, note that
\begin{align*}
    \bK[Z^{(0)}_{t\delta_k,\xi_{j,k,t}}](x)&=\frac{\sqrt{t\delta_k}}{2|x|}\frac{\left|\frac{x}{|x|^2}-\xi_{j,k,t}\right|^2-(t\delta_k)^2}{\left((t\delta_k)^2+\left|\frac{x}{|x|^2}-\xi_{j,k,t}\right|^2\right)^\frac{3}{2}}
    =\frac{\sqrt{t\delta_k}|x|}{2}\frac{|x-|x|^2\xi_{j,k,t}|^2-|x|^4(t\delta_k)^2}{\left((t\delta_k)^2|x|^4+|x-\xi_{j,k,t}|x|^2|^2\right)^\frac{3}{2}}\\
    &=\frac{\sqrt{t\delta_k}|x|}{2}\frac{|x|^2+|x|^4(|\xi_{j,k,t}|^2-(t\delta_k)^2)-2|x|^2x\cdot\xi_{j,k,t}}{\left(|x|^2+|x|^4(|\xi_{j,k,t}|^2+(t\delta_k)^2)-2|x|^2x\cdot\xi_{j,k,t}\right)^\frac{3}{2}}\\
    &=\frac{\sqrt{t\delta_k}|x|}{2}\frac{|x|^2+|x|^4(1-2(t\delta_k)^2)-2|x|^2x\cdot\xi_{j,k,t}}{\left(|x|^2+|x|^4-2|x|^2x\cdot\xi_{j,k,t}\right)^\frac{3}{2}}\\
    &=\frac{\sqrt{t\delta_k}}{2}\frac{1+|x|^2-2x\cdot\xi_{j,k,t}-2|x|^2(t\delta_k)^2}{(1+|x|^2-2x\cdot\xi_{j,k,t})^\frac{3}{2}}\\
    &=\frac{\sqrt{t\delta_k}}{2}\frac{|x-\xi_{j,k,t}|^2+(t\delta_k)^2-2|x|^2(t\delta_k)^2}{(|x-\xi_{j,k,t}|^2+(t\delta_k)^2)^\frac{3}{2}}\\
    &=Z^{(0)}_{t\delta_k,\xi_{j,k,t}}(x)+(t\delta_k)^\frac{5}{2}\frac{1-|x|^2}{(|x-\xi_{j,k,t}|^2+(t\delta_k)^2)^\frac{3}{2}}\,.
\end{align*}
Hence, the result immediately follows from the definition of $Z_{k,t}$. 
\end{proof}

Having $Z_{k,t}$, the lemma above at hand  and emphasizing that $\bK[Z_{k,t}]\in \Hud$ is $k-$fold symmetric and even with respect to $x_2$ thanks to properties ii) and iv) of $\bK$, we introduce
\begin{equation} \label{E.K-Kperp}
\cK_{k,t} := {\rm span}\big\{\tfrac12 \big(Z_{k,t}+ \bK[Z_{k,t}]\big)\big\} \quad \textup{and} \quad \cK_{k,t}^{\perp} := \{\phi \in \cH_k : \langle \phi, Z_{k,t} \rangle = 0 \}\,,
\end{equation}
and the orthogonal projections
\begin{equation}
\Pi: \cH_k \to \cK_{k,t} \quad \textup{and} \quad \Pi^{\perp}: \cH_k \to \cK_{k,t}^{\perp}\,,
\end{equation}
so that
$$
\cH_k = \cK_{k,t} \oplus \cK_{k,t}^{\perp}\,.
$$
Let us explicitly point out that here we are using that, for all $\phi \in \cH_k$,  $\langle \phi, Z_{k,t} \rangle = \langle \phi, \frac12 (Z_{k,t} + \bK[Z_{k,t}]) \rangle$. Then, we rewrite \eqref{E.YamabePhiInverse} as the system
\begin{equation}\label{E.phiSystem}
\left\{
\begin{aligned}
\, & \Pi\big[ \cL_{k,t}^{*} \phi - \cE_{k,t}^{*} - \cN_{k,t}^{*} (\phi)  \big]  = 0\,, \\ 
& \Pi^{\perp}\big[ \cL_{k,t}^{*} \phi - \cE_{k,t}^{*} - \cN_{k,t}^{*} (\phi)  \big]= 0\,, 
\end{aligned}
\right. \qquad \phi \in \cH_k\,.
\end{equation}
For later purposes, we also introduce the shortened notation
\begin{equation}
\cL_{k,t}^{\perp,*} := \Pi^{\perp} \cL_{k,t}^{*} \,, \quad \cE_{k,t}^{\perp,*} :=  \Pi^{\perp} \cE_{k,t}^{*} \quad \textup{and} \quad \cN_{k,t}^{\perp,*} := \Pi^{\perp} \cN_{k,t}^{*}\,.
\end{equation}
 We will check that, although $\cL_{k,t}^{*}$ is not uniformly bounded in $k$, the projected operator $\cL_{k,t}^{\perp,*}$ does. That allows us to perform a classical fixed point argument for the second equation in \eqref{E.phiSystem}. Once the second equation is solved in Proposition \ref{P.contraction}, we will come back to the first equation in Section \ref{sec:proof-main-result1}.

\subsection{Linear theory} \label{SS.LT} This subsection is devoted to dealing with the linearized operator $\cL_{k,t}^{\perp,*}$. Our main aim here is to prove the following key uniform (in $k$) estimate:

\begin{lemma}\label{L.aprioriEstimate}
For every $0 < a < b < \infty$, there exist $k_1 \in \NN$ and $\cC_1 > 0$ such that, for all $k \geq k_1$ and all $t \in [a,b]$, it follows that
$$
\|\phi\|_{\dot{H}^{\frac12}} \leq \mathcal{C}_1 \|\cL_{k,t}^{\perp,*}\phi\|_{\dot{H}^{\frac12}},  \quad \textup{for all } \phi \in \cK_{k,t}^{\perp}\,. 
$$ 
\end{lemma}

\begin{proof} 
We will argue by contradiction. Assume that the inequality is not true. Then, there exist sequences $(t_k)_k \subset [a,b]$ and $(\phi_k)_k \subset \mathcal{K}^\perp_{k,t_k}$ such that
    $$
    \|\phi_k\|_{\dot{H}^{\frac12}}^2=k \quad \textup{and} \quad \|\cL_{k,t_k}^{\perp,*}\phi_k\|_{\dot{H}^{\frac12}}^2=o(k)\,, \quad \textup{as } k \to \infty\,.
    $$
More explicitly, setting
    $$
    f_k:=\cL_{k,t_k}^{\perp,*}\phi_k\,,
    $$
so that
    \begin{align}\label{ecuphi1}
    \phi_k-(-\Delta)^{-\frac12}(3U^2_{k,t_k}\phi_k)=f_k+2c_k \left(Z_{k,t_k}+\bK[Z_{k,t_k}]\right),
    \end{align}
or equivalently
        \begin{align}\label{ecuphi2}
    (-\Delta)^{\frac12}\phi_k-3U^2_{k,t_k}\phi_k=(-\Delta)^{\frac12} f_k+ 2 c_k(-\Delta)^\frac12 \left(Z_{k,t_k}+\bK[Z_{k,t_k}]\right),
    \end{align}
with
    \begin{align}\label{ck}
    c_k := \frac{\langle \phi_k- (-\Delta)^{-\frac12} (3 U_{k,t_k}^2 \phi_k), Z_{k,t_k} \rangle}{\|Z_{k,t_k}+\bK[Z_{k,t_k}]\|_{\dot{H}^{\frac12}}^2}\,,
    \end{align}
we are assuming that
    $$
    \|\phi_k\|_{\dot{H}^{\frac12}}^2=k \quad \textup{and} \quad \|f_k\|_{\dot{H}^{\frac12}}^2=o(k)\,, \quad \textup{as } k \to \infty\,.
    $$

We will split the proof into several steps. However, before going any further, let us point out that,  since $\phi_k$ is $k$-fold symmetric (see \eqref{E.kfold}) then $(-\Delta)^\frac14 \phi_k$ is $k$-fold and then
    $$
    \int_{\Omega_1}|(-\Delta)^\frac{1}{4}\phi_k|^2dx=1 \quad \textup{and}
    \quad 
    3\int_{\RR^2} U_{k,t_k}^2\phi_k^2 dx=3k\int_{\Omega_1} U_{k,t_k}^2\phi_k^2dx\,.
    $$
Moreover, we should also stress that, for all $\psi \in \Hud$, using the weak formulation of \eqref{ecuphi2} we find
\begin{equation} \label{E.startingR2}
    \langle \phi_k, \psi \rangle - 3 \int_{\RR^2} U_{k,t_k}^2 \phi_k \psi dx = o\big(\sqrt{k} \|\psi\|_{\Hud}\big)\,, \quad \textup{as } k \to \infty\,.
\end{equation}
Note that here we are using Lemma \ref{L.normZ+KZ} and the following estimate: 
\begin{equation} \label{E.ckestimate}
    |c_k| \lesssim \frac{1}{\sqrt{k} \log k}\,,
\end{equation}
which is proved in Appendix \ref{appendix}.
 
\medbreak
\noindent \textbf{Step 1.} \textit{It follows that}
\begin{align}\label{uphi}
\int_{\Omega_1}U^2_{k,t_k}\phi_k^2 dx=\frac{1}{3}+o(1)\,, \quad \textup{\textit{as} } k \to \infty\,.
\end{align}

In order to prove \eqref{uphi} we notice that, multiplying \eqref{ecuphi2} by $\phi_k$ and integrating, it follows that
\begin{align*}
& \|\phi_k\|^2_{\Hud}-3\int_{\RR^2}U_{k,t_k}^2\phi_k^2 dx =k \int_{\Omega_1} \Big( |(-\Delta)^\frac{1}{4}\phi_k|^2 -3U_{k,t_k}^2\phi_k^2 \Big) dx\\
&\quad =\langle f_k,\phi_k\rangle +2 c_k \langle Z_{k,t_k}+\bK[Z_{k,t_k}],\phi_k\rangle =\langle f_k,\phi_k\rangle +4 c_k \langle Z_{k,t_k},\phi_k\rangle=\langle f_k, \phi_k\rangle\,.
\end{align*}
Also, we have that
\begin{align*}
    \frac{1}{k}|\langle f_k,\phi_k\rangle| \leq \frac{1}{k}\|f_k\|_{\Hud}\|\phi_k\|_{\Hud}=\frac{1}{\sqrt{k}}\|f_k\|_{\Hud}=o(1)\,, \quad \textup{as } k \to \infty\,.
\end{align*}
Hence, \eqref{uphi} follows.

\medbreak 
\noindent \textbf{Step 2.} \textit{It follows that}
\begin{equation} \label{E.Step2LT}
\int_{\Om_1} U_{k,t_k}^2 \phi_k^2 dx = \int_{B_{\sqrt{t\de_k}}(\xi_{1,k,t_k})} U_{1,k,t_k}^2 \phi_k^2 dx + \int_{\Om_1} U^2 \phi_k^2 dx + O \Big( \frac{1}{\log k} \Big)\,, \quad \textup{\textit{as} } k \to \infty\,.
\end{equation}

First of all, by definition
\begin{equation} \label{E.Step2Expansion}
\begin{aligned}
    & \int_{\Omega_1}U^2_{k,t_k}\phi^2_k dx = \int_{\Omega_1}\bigg(\Sjk U_{j,k,t_k}-U\bigg)^2 \phi_k^2 dx =\int_{\Omega_1}\bigg(\bigg(\Sjk U_{j,k,t_k}\bigg)^2+U^2-2U\Sjk U_{j,k,t_k}\bigg)\phi_k^2 dx\\
& \quad = \int_{\Omega_1}\bigg(U_{1,k,t_k}^2+\bigg(\sum_{j=2}^k U_{j,k,t_k}\bigg)^2 + 2 U_{1,k,t_k} \sum_{j=2}^k U_{j,k,t_k}
+U^2-2UU_{1,k,t_k}-2U\sum_{j=2}^kU_{j,k,t_k}\bigg)\phi_k^2 dx \\
& \quad = \int_{B_{\sqrt{t_k\de_k}}(\xi_{1,k,t_k})} U_{1,k,t_k}^2 \phi_k^2 dx + \int_{\Om_1} U^2 \phi_k^2 dx + \int_{\Om_1 \setminus B_{\sqrt{t_k\de_k}}(\xi_{1,k,t_k})} U_{1,k,t_k}^2 \phi_k^2 dx \\
& \qquad + \int_{\Om_1} \bigg(\bigg(\sum_{j=2}^k U_{j,k,t_k}\bigg)^2 + 2 U_{1,k,t_k} \sum_{j=2}^k U_{j,k,t_k}
-2UU_{1,k,t_k}-2U\sum_{j=2}^kU_{j,k,t_k}\bigg)\phi_k^2 dx\,.
\end{aligned}
\end{equation}

Having this expression at hand, to prove \eqref{E.Step2LT}, we deal with the ``remainder'' terms separately. Since
\begin{align*}
    k\|\phi_k\|_{L^4(\Omega_1)}^4=\|\phi_k\|^4_{L^4}\ls \|\phi_k\|_{\Hud}^4=k^2\,,
\end{align*}
we obtain that
\begin{equation} \label{E.phikSectorL4}
    \|\phi_k\|_{L^4(\Omega_1)}\ls k^\frac{1}{4}\,,
\end{equation}
which in turn yields
\begin{align*}
    \int_{\Omega_1}\bigg(\sum_{j=2}^k U_{j,k,t_k}\bigg)^2\phi_k^2\, dx\ls \|\phi_k\|^2_{L^4(\Omega_1)}\bigg(\int_{\Omega_1}\bigg(\sum_{j=2}^k U_{j,k,t_k}\bigg)^4 dx \bigg)^\frac{1}{2} \ls \bigg(k\int_{\Omega_1}\bigg(\sum_{j=2}^k U_{j,k,t_k}\bigg)^4 dx\bigg)^\frac{1}{2} \ls \frac{1}{\log^2(k)}\,. 
\end{align*}
Note that the last inequality follows from Lemma \ref{L.PremoselliVetois} applied with $\alpha=0$ and $\beta=4.$

Likewise, we have that
\begin{align*}
& \int_{\Om_1} U_{1,k,t_k} \sum_{j=2}^k U_{j,k,t_k} \phi_k^2 dx \leq \|\phi_k\|_{L^4(\Om_1)}^2 \bigg( \int_{\Om_1} U_{1,k,t_k}^2 \bigg(\sum_{j=2}^k U_{j,k,t_k} \bigg)^2 dx \bigg)^{\frac12} \\
& \quad \lesssim \bigg( k \int_{\Om_1} U_{1,k,t_k}^2 \bigg(\sum_{j=2}^k U_{j,k,t_k} \bigg)^2 dx \bigg)^{\frac12} \lesssim\frac{1}{\sqrt{k \log k}}\,,
\end{align*}
and that
\begin{align*}
    \int_{\Omega_1}U U_{1,k,t_k}\phi_k^2 dx \ls \|\phi_k\|^2_{L^4(\Omega_1)}\bigg( \int_{\Om_1} U^2 U_{1,k,t_k}^2 dx\bigg)^\frac{1}{2} \ls \left(k\int_{\Omega_1} U^2 U_{1,k,t_k}^2 dx\right)^\frac{1}{2}\ls \frac{1}{\sqrt{k\log k}},
\end{align*}
where we have applied Lemma \ref{L.PremoselliVetois} with $\alpha = \beta = 2$, and $\alpha=2$ and $\beta=0,$ respectively. 

Similarly, we infer that
\begin{align*}
    \int_{\Omega_1}U\sum_{j=2}^kU_{j,k,t_k}\phi_k^2 dx \ls  \|\phi\|^2_{L^4(\Omega_1)}\bigg(\int_{\Om_1} U^2\bigg(\sum_{j=2}^k U_{j,k,t_k}\bigg)^2 dx\bigg)^\frac{1}{2}
    \ls  \bigg(k\int_{\Omega_1}U^2\bigg(\sum_{j=2}^k U_{j,k,t_k}\bigg)^2 dx\bigg)^\frac{1}{2}\ls \frac{1}{\log k}\,.
\end{align*}
Note that the last inequality follows from Lemma \ref{L.PremoselliVetois} applied with $\alpha=0$ and $\beta=2$.

Finally, arguing as in the proof of Lemma \ref{L.PremoselliVetois} when dealing with $I_2$ for $\al = 4$ (there is a factor $\sqrt{\delta_k}$ here instead of $\delta_k$ there but otherwise the estimate is the same), we get that
\begin{align*}
& \int_{\Omega_1\setminus B_{\sqrt{t_k\de_k}}(\xi_{1,k,t_k})} U_{1,k,t_k}^2\phi_k^2\, dx \ls  \|\phi_k\|_{L^4(\Omega_1)}^2\left(\int_{\Omega_1\setminus B_{\sqrt{t_k\de_k}}(\xi_{1,k,t_k})} U_{1,k,t_k}^4 dx\right)^\frac{1}{2} \\
& \quad \ls  \bigg( k \int_{\Omega_1\setminus B_{\sqrt{t_k\de_k}}(\xi_{1,k,t_k})} U_{1,k,t_k}^4\bigg)^\frac{1}{2}\ls \frac{1}{\sqrt{k}\log k}\,.
\end{align*}
Substituting the estimates of the ``remainder'' terms into \eqref{E.Step2Expansion}, we get \eqref{E.Step2LT}, and Step 2 follows.

\medbreak
\noindent
\textbf{Step 3.} \textit{The rescaled and translated sequence}  $(\widetilde{\phi}_k)_{k}$ \textit{given by}
\begin{equation} \label{E.phikTilde}
    \widetilde{\phi}_k(y):=\sqrt{t_k\delta_k}\eta(\sqrt{t_k\delta_k} |y|)\phi_k(t_k\delta_k y +\xi_{1,k,t_k})\,, \quad \textup{\textit{for all }} k \in \NN\,,
\end{equation}
\textit{is bounded in }$\Hud$. \textit{Here, $\eta: [0,\infty) \to [0,1]$ is a smooth cutoff function such that $\eta \equiv 1$ in $[0,1)$ and $\eta \equiv 0$ in $[2,\infty)$.}

In order to bound the $\Hud-$norm of $\widetilde{\phi}_k$ we will use its harmonic extension to $\RR^3_+ := \RR^2\times (0,\infty)$. For all $k \in \NN$, we denote by $\Phi_k$ the harmonic extension to $\RR_+^3$ of $\phi_k$, and by $\widetilde{\Phi}_k$ the harmonic extension to $\RR_+^3$  of $\widetilde{\phi}_k$. Also, we introduce
$$
\widehat{\Phi}_k(Y) := \sqrt{t_k \de_k} \, \eta (\sqrt{t_k\de_k}|Y|) \Phi_k(t_k \de_k Y + \Xi_{1,k,t_k})\,, \quad \textup{for all } k \in \NN\,,
$$
where
$$
\Xi_{1,k,t_k}:= (\xi_{1,k,t_k},0)\,, \quad \textup{for all } k \in \NN\,.
$$

Now, observe that
\begin{align*}
    & \int_{\RR_+^3} |\nabla \widehat{\Phi}_k(Y)|^2 dY \\
    & \quad = \int_{\RR_+^3} t_k \de_k \bigg| t_k \de_k \nabla \Phi_k(t_k \de_k Y + \Xi_{1,k,t_k}) \eta(\sqrt{t_k \de_k}|Y|) + \sqrt{t_k \de_k} \Phi_k(t_k \de_k Y + \Xi_{1,k,t_k}) \eta'(\sqrt{t_k \de_k} |Y|) \frac{Y}{|Y|} \bigg|^2 dY \\
    & \quad \leq 2 \int_{\RR_+^3} (t_k \de_k)^3 \eta(\sqrt{t_k \de_k}|Y|)^2 |\nabla \Phi_k(t_k \de_k Y + \Xi_{1,k,t_k})|^2 dY \\
    & \qquad + 2 \int_{\RR_+^3} (t_k \de_k)^2 \eta'(\sqrt{t_k \de_k}|Y|)^2 \Phi_k (t_k \de_k Y + \Xi_{1,k,t_k})^2 dY =: \widehat{I}_1 + \widehat{I}_2\,. 
\end{align*}
On the one hand, by direct calculations
$$
\widehat{I}_1 = 2 \int_{B^+_{2\sqrt{t_k\de_k}}(\Xi_{1,k,t_k})} \eta\bigg(\frac{|X - \Xi_{1,k,t_k}|}{\sqrt{t_k\de_k}}\bigg)^2 |\nabla \Phi_k(X)|^2 dX \leq 2 \int_{\Om_1 \times (0,\infty)} |\nabla \Phi_k(X)|^2 dX = 2\,.
$$
Note that here and through the proof, we use the notation $B_R^+(\Xi):= B_R(\Xi)\cap \RR_+^3$. On the other hand, by Corollary \ref{C.improvedSobolev}, 
$$
\widehat{I}_2 = \frac{2}{t_k \de_k} \int_{B^+_{2\sqrt{t_k\de_k}}(\Xi_{1,k,t_k})} \eta'\bigg(\frac{|X - \Xi_{1,k,t_k}|}{\sqrt{t_k\de_k}}\bigg)^2 \Phi_k(X)^2 dX \leq  \frac{2\|\eta'\|^2_{L^{\infty}}}{t_k \de_k} \int_{B^+_{2\sqrt{t_k\de_k}}(\Xi_{1,k,t_k})}  \Phi_k(X)^2 dX \lesssim 1\,.
$$
Hence, it follows that
\begin{equation} \label{E.Step3LT1}
\|\nabla \widehat{\Phi}_k\|_{L^2(\RR_+^3)} \lesssim 1\,.
\end{equation}

Next, observe that
\begin{align*}
    \Delta \widehat{\Phi}_k(Y) = &\ (t_k \de_k)^{\frac32}  \Phi_k(t_k \de_k Y + \Xi_{1,k,t_k}) \Big( \eta''(\sqrt{t_k \de_k}|Y|) + \frac{2}{\sqrt{t_k \de_k}|Y|} \eta'(\sqrt{t_k \de_k} |Y|) \Big) \\
    & + 2(t_k \de_k)^{2} \eta'(\sqrt{t_k \de_k} |Y|)  \nabla \Phi_k(t_k \de_k Y + \Xi_{1,k,t_k}) \cdot \frac{Y}{|Y|} =: F(Y)\,,
\end{align*}
and so that
\begin{equation} \label{E.differenceExtension}
-\Delta (\widetilde{\Phi}_k - \widehat{\Phi}_k) = F \textup{ in } \RR_+^3\,, \qquad \widetilde{\Phi}_k - \widehat{\Phi}_k = 0 \textup{ on } \partial\RR_+^3 \equiv \RR^2\,.
\end{equation}
Hence, it follows that
\begin{equation}  \label{E.Step3LT2}
\|\nabla (\widetilde{\Phi}_k - \widehat{\Phi}_k)\|_{L^2(\RR_+^3)} \lesssim \|F\|_{L^{\frac65}(\RR_+^3)} 
\end{equation}

Finally, observe that
\begin{align*}
    & (t_k \de_k)^{\frac95} \int_{\RR_+^3} |\Phi_k(t_k \de_k Y + \Xi_{1,k,t_k})|^{\frac65}  \Big| \eta''(\sqrt{t_k \de_k}|Y|) + \frac{2}{\sqrt{t_k \de_k}|Y|} \eta'(\sqrt{t_k \de_k} |Y|) \Big|^{\frac65} dY \\
    & \quad = \frac{1}{(t_k\de_k)^{\frac65}} \int_{\RR_+^3} |\Phi_k(X)|^{\frac65}  \bigg| \eta''\bigg( \frac{|X-\Xi_{1,k,t_k}|}{\sqrt{t_k\de_k}}  \bigg) + \frac{2\sqrt{t_k\de_k}}{|X-\Xi_{1,k,t_k}|} \eta'\bigg( \frac{|X-\Xi_{1,k,t_k}|}{\sqrt{t_k\de_k}}  \bigg) \bigg|^{\frac65} dX \\
    & \quad \leq \bigg( \frac{1}{t_k\de_k} \int_{B^+_{2\sqrt{t_k \de_k}}(\Xi_{1,k,t_k})} \Phi_k(X)^2 dX \bigg)^{\frac35} \\
    & \qquad\ \times \bigg( \frac{1}{(t_k\de_k)^{\frac32}}  \int_{B^+_{2\sqrt{t_k \de_k}}(\Xi_{1,k,t_k})} \bigg| \eta''\bigg( \frac{|X-\Xi_{1,k,t_k}|}{\sqrt{t_k\de_k}}  \bigg) + \frac{2\sqrt{t_k\de_k}}{|X-\Xi_{1,k,t_k}|} \eta'\bigg( \frac{|X-\Xi_{1,k,t_k}|}{\sqrt{t_k\de_k}}  \bigg) \bigg|^3 dX \bigg)^{\frac25} \\
    & \quad =  \bigg( \frac{1}{t_k\de_k} \int_{B^+_{2\sqrt{t_k \de_k}}(\Xi_{1,k,t_k})} \Phi_k(X)^2 dX \bigg)^{\frac35} \bigg( \int_{B_2^{+}(0)} \Big| \eta''(|Z|) + \frac{2}{|Z|} \eta'(|Z|) \Big|^3 dZ \bigg)^{\frac25} \lesssim 1\,.
\end{align*}
Note that in the last step we have again used Corollary \ref{C.improvedSobolev}. Similarly, it follows that
\begin{align*}
    & (t_k \de_k)^{\frac{12}{5}} \int_{\RR_+^3} |\eta'(\sqrt{t_k \de_k}|Y|)|^{\frac65} |\nabla \Phi_k(t_k\de_k Y + \Xi_{1,k,t_k})|^{\frac65} dY  = \frac{1}{(t_k \de_k)^{\frac35}} \int_{\RR_+^3} |\nabla \Phi_k(X)|^{\frac65} \bigg| \eta'\bigg( \frac{|X-\Xi_{1,k,t_k}|}{\sqrt{t_k\de_k}}  \bigg)\bigg|^{\frac65} dX  \\
    &  \quad\leq \bigg( \int_{B^+_{2\sqrt{t_k \de_k}}(\Xi_{1,k,t_k})} |\nabla\Phi_k(X)|^2 dX \bigg)^{\frac35} \bigg( \frac{1}{(t_k\de_k)^{\frac32}}  \int_{B^+_{2\sqrt{t_k \de_k}}(\Xi_{1,k,t_k})} \bigg|\eta'\bigg( \frac{|X-\Xi_{1,k,t_k}|}{\sqrt{t_k\de_k}}  \bigg)\bigg|^3 dX \bigg)^{\frac25} \\
    & \quad = \bigg( \int_{B^+_{2\sqrt{t_k \de_k}}(\Xi_{1,k,t_k})} |\nabla\Phi_k(X)|^2 dX \bigg)^{\frac35} \bigg( \int_{B_2^{+}(0)}  |\eta'(|Z|)|^3 dZ \bigg)^{\frac25}  \lesssim 1\,.
\end{align*}
Then, we have that
\begin{equation}  \label{E.Step3LT3}
\|F\|_{L^{\frac65}(\RR_+^3)} \lesssim 1\,.
\end{equation}

At this point, using that
$$
\|\widetilde{\phi}_k\|_{\Hud} = \|\nabla \widetilde{\Phi}_k\|_{L^2(\RR_+^3)}\,,
$$
the third step follows from \eqref{E.Step3LT1}, \eqref{E.Step3LT2} and \eqref{E.Step3LT3}.

\medbreak 
\noindent 
\textbf{Step 4.}\textit{ Up to a subsequence,} \textit{the sequence }$(\widetilde{\phi}_k)_{k}$
\textit{given in \eqref{E.phikTilde} satisfies}
\begin{align*}
    \widetilde{\phi}_k \rightharpoonup 0 \quad \text{\textit{in }$\dot{H}^{\frac{1}{2}}$} \quad \textup{\textit{and}} \quad
    \widetilde{\phi}_k \to 0 \quad \text{\textit{in} $L^p_{\text{loc}}$}\,,\ 2 \leq p < 4\,, \quad \textup{\textit{as} } k \to \infty\,.
\end{align*}

By Step 3, we know that $(\widetilde{\phi}_k)_k$ is a bounded sequence in $\Hud$. Hence, up to a subsequence, it follows that
$$
\widetilde{\phi}_k \rightharpoonup \widetilde{\phi}_\infty  \quad \textup{in }\dot{H}^{\frac{1}{2}} \quad \textup{and} \quad
    \widetilde{\phi}_k \to \widetilde{\phi}_\infty  \quad \textup{in }L^p_{\text{loc}}\,,\ 2 \leq p < 4\,, \quad \textup{as } k \to \infty\,,
$$
for some $\widetilde{\phi}_\infty \in \Hud$.

\medbreak
\noindent \textit{\underline{Step 4.1}.} $\widetilde{\phi}_{\infty} \in \Hud$ \textit{is a solution to}
$$
(-\Delta)^{\frac12} \widetilde{\phi}_{\infty} - 3 U^2 \widetilde{\phi}_{\infty} = 0 \quad \textup{\textit{in} } \RR^2\,.
$$

Let $\psi \in C_c^{\infty}(\RR^2)$ be fixed but arbitrary and let $R > 0$ be sufficiently large so that $\supp \psi \subset B_R(0)$. We denote by $\Psi$ its harmonic extension to $\RR_+^3$ and point out that
\begin{equation} \label{E.decayPsi}
|\Psi(Y)| \lesssim \frac{t}{|Y|^3} \,\|\psi\|_{L^1}
\quad \textup{and} \quad 
|\nabla \Psi(Y)| \lesssim \frac{1}{|Y|^3}\, \|\psi\|_{L^1}\,, \quad \textup{for } |Y| = |(y,t)| \geq 2R\,.
\end{equation}
Also, we introduce the notation
$$
\widetilde{U}_{k,t_k}(y) := \sqrt{t_k \de_k}\, U_{k,t_k}(t_k \de_k y + \xi_{1,k,t_k})\,.
$$

First, since $\supp \psi \subset B_R(0)$ and 
$$
\Big\| \frac{1}{\sqrt{t_k \de_k}} \psi \Big( \frac{\cdot - \xi_{1,k,t_k}}{t_k \de_k} \Big) \Big\|_{\Hud} = \|\psi\|_{\Hud}\,,
$$ 
we get 
\begin{equation} \label{E.psiktildepotential}
    \begin{aligned}
    & \int_{\RR^2} \widetilde{U}_{k,t_k}(y)^2 \widetilde{\phi}_k(y) \psi(y) dy = \int_{B_{R t_k \de_k}(\xi_{1,k,t_k})} U_{k,t_k}(x)^2 \phi_k(x) \frac{1}{\sqrt{t_k \de_k}} \psi \Big( \frac{x-\xi_{1,k,t_k}}{t_k \de_k} \Big) dx \\
    & \quad  =\int_{\Om_1} U_{k,t_k}(x)^2 \phi_k(x)  \frac{1}{\sqrt{t_k \de_k}} \psi \Big( \frac{x-\xi_{1,k,t_k}}{t_k \de_k} \Big) dx\,.
    \end{aligned}
\end{equation}

Next, note that
\begin{align*}
    \langle \widetilde{\phi}_k, \psi \rangle = \int_{\RR_+^3} \nabla \widetilde{\Phi}_k(Y) \cdot \nabla \Psi(Y) \, dY = \int_{\RR_+^3} \nabla \widehat{\Phi}_k(Y) \cdot \nabla \Psi(Y)\, dY + \int_{\RR^3_+} \nabla (\widetilde{\Phi}_k(Y) - \widehat{\Phi}_k(Y) ) \cdot \nabla \Psi(Y) \, dY\,.
\end{align*}
We analyze the two terms on the right hand side separately. Observe that the first one can be rewritten as
\begin{align*}
& \int_{\RR_+^3} \nabla \widehat{\Phi}_k(Y) \cdot \nabla \Psi(Y) \, dY = \int_{\RR_+^3} (t_k \de_k)^{\frac32}  \eta(\sqrt{t_k \de_k}|Y|) \nabla \Phi_k(t_k \de_k Y + \Xi_{1,k,t_k}) \cdot \nabla \Psi(Y)\, dY \\
& \qquad + \int_{\RR_+^3} t_k \de_k \Phi_k(t_k \de_k Y + \Xi_{1,k,t_k}) \eta'(\sqrt{t_k \de_k} |Y|) \frac{Y}{|Y|} \cdot \nabla \Psi(Y) \, dY =: \widehat{J}_1 + \widehat{J}_2\,.
\end{align*}
Hence, we treat $\widehat{J}_1$ and $\widehat{J}_2$ separately. First, we deal with $\widehat{J}_2$. Combining Corollary \ref{C.improvedSobolev} and \eqref{E.decayPsi}, we get
\begin{align*}
     |\widehat{J}_2| &\lesssim t_k \de_k \int_{B_{\frac{2}{\sqrt{t_k\de_k}}}^{+}(0) \setminus B_{\frac{1}{\sqrt{t_k\de_k}}}^{+}(0)} \frac{ \|\psi\|_{L^1}}{|Y|^3} |\Phi_k(t_k \de_k Y + \Xi_{1,k,t_k})| |\eta'(\sqrt{t_k \de_k}|Y|)| dY \\
    & \lesssim  \|\psi\|_{L^1} (t_k \de_k)^{\frac52} \int_{B_{\frac{2}{\sqrt{t_k\de_k}}}^{+}(0) \setminus B_{\frac{1}{\sqrt{t_k\de_k}}}^{+}(0)} |\Phi_k(t_k \de_k Y + \Xi_{1,k,t_k})| |\eta'(\sqrt{t_k \de_k}|Y|)| dY \\
    & \leq  \|\psi\|_{L^1} \frac{1}{\sqrt{t_k \de_k}} \int_{B_{2\sqrt{t_k\de_k}}^{+}(\Xi_{1,k,t_k})} |\Phi_k(X)| \bigg|\eta'\bigg( \frac{|X-\Xi_{1,k,t_k}|}{\sqrt{t_k\de_k}}  \bigg)\bigg| dX \\
    &\leq  \|\psi\|_{L^1} \bigg( \frac{1}{t_k \de_k} \int_{B_{2\sqrt{t_k\de_k}}^{+}(\Xi_{1,k,t_k})} \Phi_k(X)^2 dX \bigg)^{\frac12} \bigg( \int_{B_{2\sqrt{t_k\de_k}}^{+}(\Xi_{1,k,t_k})} \eta'\bigg( \frac{|X-\Xi_{1,k,t_k}|}{\sqrt{t_k\de_k}}  \bigg)^2 dX \bigg)^{\frac12} \\
    &  =  \|\psi\|_{L^1} (t_k \de_k)^{\frac34} \bigg( \frac{1}{t_k \de_k} \int_{B_{2\sqrt{t_k\de_k}}^{+}(\Xi_{1,k,t_k})} \Phi_k(X)^2 dX \bigg)^{\frac12} \bigg( \int_{B_2^{+}(0)} \eta'(|Z|)^2 dZ \bigg)^{\frac12} \lesssim \frac{ \|\psi\|_{L^1}}{(k \log k)^{\frac32}}\,.
\end{align*}
Now, to deal with $\widehat{J}_1$, let us start by pointing out that
\begin{align*}
    \widehat{J}_1 = &\  \int_{B_{\sqrt{t_k\de_k}}^{+}(\Xi_{1,k,t_k})} \nabla \Phi_k(X) \cdot \frac{1}{(t_k \de_k)^{\frac32}} \nabla \Psi \bigg( \frac{X-\Xi_{1,k,t_k}}{t_k \de_k} \bigg) dX \\
    & + (t_k \de_k)^{\frac32}  \int_{B_{\frac{2}{\sqrt{t_k\de_k}}}^{+}(0) \setminus B_{\frac{1}{\sqrt{t_k\de_k}}}^{+}(0)} \eta( \sqrt{t_k \de_k}|Y|) \nabla \Phi_k(t_k \de_k Y + \Xi_{1,k,t_k} ) \cdot \nabla \Psi(Y)\, dY\,.
\end{align*}
Arguing exactly as we did to handle $\widehat{J}_2$, we get that
$$
\bigg| (t_k \de_k)^{\frac32}  \int_{B_{\frac{2}{\sqrt{t_k\de_k}}}^{+}(0) \setminus B_{\frac{1}{\sqrt{t_k\de_k}}}^{+}(0)} \eta( \sqrt{t_k \de_k}|Y|) \nabla \Phi_k(t_k \de_k Y + \Xi_{1,k,t_k} ) \cdot \nabla \Psi(Y)\, dY \bigg| \lesssim \frac{ \|\psi\|_{L^1}}{(k \log k)^{\frac32}}\,.
$$
Likewise, it follows that
\begin{align*}
    & \int_{\Om_1 \times (0,\infty) \setminus B_{\sqrt{t_k\de_k}}^{+}(\Xi_{1,k,t_k})}   \nabla \Phi_k(X) \cdot \frac{1}{(t_k \de_k)^{\frac32}} \nabla \Psi \bigg( \frac{X-\Xi_{1,k,t_k}}{t_k \de_k} \bigg) dX \\
    & \quad \leq \bigg( \int_{\Om_1 \times (0,\infty)} |\nabla \Phi_k(X)|^2 dX \bigg)^{\frac12} \bigg( \int_{\RR_+^3 \setminus B_{\frac{1}{\sqrt{t_k\de_k}}}^{+}(0) } |\nabla \Psi(Y)|^2 dY \bigg)^{\frac12}\lesssim \frac{ \|\psi\|_{L^1}}{(k \log k)^{\frac32}}\,.
\end{align*}
Hence, we conclude that 
\begin{equation} \label{E.bracketPsihat}
     \int_{\RR_+^3} \nabla \widehat{\Phi}_k(Y) \cdot \nabla \Psi(Y) \, dY = \int_{\Om_1 \times (0,\infty)}  \nabla \Phi_k(X) \cdot \frac{1}{(t_k \de_k)^{\frac32}} \nabla \Psi \bigg( \frac{X-\Xi_{1,k,t_k}}{t_k \de_k} \bigg) dX + O \bigg( \frac{\|\psi\|_{L^1}}{(k \log k)^{\frac32}} \bigg)\,. 
\end{equation}

Next, taking into account \eqref{E.differenceExtension} and using the notation $Y = (y,t)$, we get
$$
\int_{\RR^3_+} \nabla (\widetilde{\Phi}_k(Y) - \widehat{\Phi}_k(Y) ) \cdot \nabla \Psi(Y) \, dY = \int_{\RR_+^3} F(Y) \Psi(Y)\, dY - \int_{\RR^2} \psi(y) \partial_t(\widetilde{\Phi}_k(y,t) - \widehat{\Phi}_k(y,t))\big|_{t=0}\, dy\,. 
$$
Also, note that
$$
\widetilde{\Phi}_k(Y) - \widehat{\Phi}_k(Y)  = \frac{1}{4\pi} \int_{\RR_+^3} \bigg( \frac{1}{|Y-Z|} - \frac{1}{|Y-Z^*|} \bigg) F(Z) \, dZ\,, \quad \textup{where } Z^* := (z_1,z_2,-s)\,.
$$
Moreover, since 
$$
\supp \psi \subset B_R(0) \quad \textup{and} \quad \supp F \subset B_{\frac{2}{\sqrt{t_k\de_k}}}^{+}(0) \setminus B_{\frac{1}{\sqrt{t_k\de_k}}}^{+}(0)\,,$$
we can differentiate under the sign integral to get that
\begin{align*}
    & 4\pi \partial_t(\widetilde{\Phi}_k(y,t) - \widehat{\Phi}_k(y,t))\big|_{t=0} = \int_{\RR_+^3} \frac{2s}{(|y-z|^2+s^2)^{\frac32}} F(Z) \, dZ \\
    & \quad = \int_{B_{\frac{2}{\sqrt{t_k\de_k}}}^{+}(0) \setminus B_{\frac{1}{\sqrt{t_k\de_k}}}^{+}(0)} \frac{2s}{(|y-z|^2+s^2)^{\frac32}} F(Z) \, dZ\,, \quad \textup{for all } y \in \supp \psi\,.
\end{align*}
Here, we are using the notation $Y = (y,t)$ and $Z = (z,s)$. Moreover, observe that for $y \in \supp \psi$ and $Z \in \supp F$
$$
|y-z|^2+s^2 \geq |z|^2+s^2+|y|^2-2|y||z| \geq \frac{|z|^2}{2}+s^2 - |y|^2 \geq \frac{|Z|^2}{2}-|y|^2 \geq \frac{1}{4t_k\de_k} \quad \textup{and} \quad 0 < s \leq \frac{2}{\sqrt{t_k\de_k}}\,.
$$
Hence, it follows that
\begin{align*}
& \Big| \partial_t(\widetilde{\Phi}_k(y,t) - \widehat{\Phi}_k(y,t))\big|_{t=0} \Big| \lesssim t_k \de_k \int_{B_{\frac{2}{\sqrt{t_k\de_k}}}^{+}(0) \setminus B_{\frac{1}{\sqrt{t_k\de_k}}}^{+}(0)} |F(Z)| \, dZ \\
& \quad \lesssim (t_k\de_k)^{\frac52} \int_{B_{\frac{2}{\sqrt{t_k\de_k}}}^{+}(0) \setminus B_{\frac{1}{\sqrt{t_k\de_k}}}^{+}(0)} |\Phi_k(t_k \de_k Z + \Xi_{1,k,t_k})| \Big|\eta''(\sqrt{t_k \de_k}|Z|) + \frac{2}{\sqrt{t_k \de_k}|Z|} \eta'(\sqrt{t_k \de_k} |Z|) \Big| dZ \\
& \qquad + (t_k\de_k)^3  \int_{B_{\frac{2}{\sqrt{t_k\de_k}}}^{+}(0) \setminus B_{\frac{1}{\sqrt{t_k\de_k}}}^{+}(0)} |\eta'(\sqrt{t_k \de_k} |Z|)|| \nabla \Phi_k(t_k \de_k Z + \Xi_{1,k,t_k})| dZ\,.
\end{align*}
The two terms on the right hand side can be handled exactly as we did $\widehat{J}_1$ and $\widehat{J}_2$. Hence, we conclude that
\begin{equation*}
    \bigg|\int_{\RR^2} \psi(y) \partial_t(\widetilde{\Phi}_k(y,t) - \widehat{\Phi}_k(y,t))\big|_{t=0}\, dy\bigg| \lesssim \frac{\|\psi\|_{L^1}}{(k \log k)^{\frac32}}\,.
\end{equation*}
Similarly, taking \eqref{E.decayPsi} into account, it follows that
\begin{align*}
& \bigg| \int_{\RR_+^3} F(Y) \Psi(Y)\, dY\bigg| \lesssim \|\psi\|_{L^1} \int_{B_{\frac{2}{\sqrt{t_k\de_k}}}^{+}(0) \setminus B_{\frac{1}{\sqrt{t_k\de_k}}}^{+}(0)} \frac{t}{|Y|^3} |F(Y)| dY \\
& \quad \lesssim \|\psi\|_{L^1} t_k \de_k \int_{B_{\frac{2}{\sqrt{t_k\de_k}}}^{+}(0) \setminus B_{\frac{1}{\sqrt{t_k\de_k}}}^{+}(0)} |F(Z)| \, dZ\,,
\end{align*}
and so that
 
\begin{equation*}
     \bigg| \int_{\RR_+^3} F(Y) \Psi(Y)\, dY\bigg| \lesssim \frac{\|\psi\|_{L^1}}{(k \log k)^{\frac32}}\,.
\end{equation*}
Hence, we conclude that
\begin{equation} \label{E.bracketCommutator}
    \bigg|\int_{\RR^3_+} \nabla (\widetilde{\Phi}_k(Y) - \widehat{\Phi}_k(Y) ) \cdot \nabla \Psi(Y) \, dY \bigg| \lesssim \frac{\|\psi\|_{L^1}}{(k \log k)^{\frac32}}\,.
\end{equation}
Then, combining \eqref{E.bracketPsihat} and \eqref{E.bracketCommutator}, we conclude that
\begin{equation} \label{E.psiktildeH12}
    \langle \widetilde{\phi}_k, \psi \rangle =  \int_{\Om_1 \times (0,\infty)}  \nabla \Phi_k(X) \cdot \frac{1}{(t_k \de_k)^{\frac32}} \nabla \Psi \bigg( \frac{X-\Xi_{1,k,t_k}}{t_k \de_k} \bigg) dX + O \bigg( \frac{\|\psi\|_{L^1}}{(k \log k)^{\frac32}} \bigg)\,. 
\end{equation}

Moreover, using that $\phi_k$ is $k$-fold symmetric and even in $x_2$ (see \eqref{E.even}-\eqref{E.kfold}), we get that
$$
\int_{\Om_1 \times (0,\infty)}  \nabla \Phi_k(X) \cdot \frac{1}{(t_k \de_k)^{\frac32}} \nabla \Psi \bigg( \frac{X-\Xi_{1,k,t_k}}{t_k \de_k} \bigg) dX = \int_{\RR_+^3} \nabla \Phi_k(X) \cdot \nabla\Psi_k^{\rm sym}(X)\,dX \,,
$$
and that
$$
\int_{\Om_1} U_{k,t_k}(x)^2 \phi_k(x)  \frac{1}{\sqrt{t_k \de_k}} \psi \Big( \frac{x-\xi_{1,k,t_k}}{t_k \de_k} \Big) dx = \int_{\RR^2} U_{k,t_k}(x)^2 \phi_k(x) \psi_k^{\rm sym}(x) dx\,,
$$
Here,
$$
\psi_k^{\rm sym}(x):= \frac{1}{k} \Sjk \frac{1}{\sqrt{t_k \de_k}}\, \psi \bigg( R_{\frac{2\pi(j-1)}{k}}^{-1} \bigg( \frac{x-\xi_{j,k,t_k}}{t_k \de_k}\bigg) \bigg)\,,
$$
and $\Psi_k^{\rm sym}$ denotes its harmonic extension to $\RR_+^3$. 

At this point, plugging \eqref{E.psiktildepotential} and \eqref{E.psiktildeH12} into \eqref{E.startingR2}, we get that
\begin{align*}
    \langle \widetilde{\phi}_k, \psi \rangle - 3 \int_{\RR^2} \widetilde{U}_{k,t_k}^2 \widetilde{\phi}_k \psi dx = \langle \phi_k, \psi_k^{\rm sym} \rangle - 3 \int_{\RR^2} U_{k,t_k}^2 \phi_k \psi_k^{\rm sym} dx = o (\sqrt{k} \|\psi_k^{\rm sym}\|_{\Hud}) +  O \bigg( \frac{\|\psi\|_{L^1}}{(k \log k)^{\frac32}} \bigg)\,. 
\end{align*}
Moreover, observe that
\begin{align*}
    & \|\psi_k^{\rm sym}\|_{\Hud}^2  = \|\nabla \Psi_{k}^{\rm sym}\|_{L^2(\RR_+^3)}^2 \\
    &\quad = \frac{1}{k} \|\nabla\Psi\|_{L^2(\RR_+^3)}^2 + \frac{1}{k^2} \Sjk \sum_{i \neq j} \int_{\RR_+^3} \frac{1}{(t_k \de_k)^3} \overline{R}_j \nabla \Psi \bigg( \overline{R}_j^{-1} \bigg( \frac{X-\Xi_{j,k,t_k}}{t_k \de_k} \bigg) \bigg) \cdot \overline{R}_i \nabla \Psi \bigg( \overline{R}_i^{-1} \bigg( \frac{X-\Xi_{i,k,t_k}}{t_k \de_k} \bigg) \bigg) dX\\
    & \quad =: \frac{1}{k} \|\psi\|_{\Hud}^2 + \frac{1}{k^2} \Sjk \sum_{i \neq j} I_{ij}\,.
\end{align*}
Note that we are using the shortened notation
$
\overline{R}_j := R_{\frac{2\pi(j-1)}{k}}, 
$
with $R_\theta$ as in Lemma \ref{L.QAcorollary}. Likewise, we set
$$
D_{ij} := \frac{|\Xi_{j,k,t_k}-\Xi_{i,k,t_k}|}{t_k\de_k}\,.
$$
Then, observe that
\begin{align*}
    |I_{ij}| & \leq \int_{\RR_+^3} |\nabla \Psi(Y)|\bigg|\nabla \Psi\bigg(\overline{R}_i^{-1} \overline{R}_j Y + \overline{R}_i^{-1} \frac{\Xi_{j,k,t_k}-\Xi_{i,k,t_k}}{t_k\de_k} \bigg)\bigg| dY \\
    & = \int_{B_{\frac{D_{ij}}{2}}^{+}(Y)}  |\nabla \Psi(Y)|\bigg|\nabla \Psi\bigg(\overline{R}_i^{-1} \overline{R}_j Y + \overline{R}_i^{-1} \frac{\Xi_{j,k,t_k}-\Xi_{i,k,t_k}}{t_k\de_k} \bigg)\bigg| dY \\
    & \quad + \int_{\RR_+^3 \setminus B_{\frac{D_{ij}}{2}}^{+}(Y)}  |\nabla \Psi(Y)|\bigg|\nabla \Psi\bigg(\overline{R}_i^{-1} \overline{R}_j Y + \overline{R}_i^{-1} \frac{\Xi_{j,k,t_k}-\Xi_{i,k,t_k}}{t_k\de_k} \bigg)\bigg| dY =: I_{ij}^1 + I_{ij}^2\,.
\end{align*}
We estimate the two terms on the right hand side separately. To deal with the first one, we stress that
$$
\bigg| \overline{R}_i^{-1} \overline{R}_j Y + \overline{R}_i^{-1} \frac{\Xi_{j,k,t_k}-\Xi_{i,k,t_k}}{t_k\de_k}\bigg| \geq D_{ij}-|Y| \geq \frac{D_{ij}}{2} \quad \textup{in } B_{\frac{D_{ij}}{2}}^{+}(Y)\,.
$$
Thus, using \eqref{E.decayPsi}, one can easily check that
$$
I_{ij}^1 \lesssim (t_k \de_k)^{\frac32} \, \frac{\|\nabla \Psi\|_{L^2(\RR_+^3)} \|\psi\|_{L^1}}{|\Xi_{j,k,t_k}- \Xi_{i,k,t_k}|^{\frac32}}\,.
$$
Similarly, it follows that
\begin{align*}
    I_{ij}^2 & \lesssim \bigg( \int_{\RR_+^3 \setminus B_{\frac{D_{ij}}{2}}^{+}(Y)} \bigg|\nabla \Psi\bigg(\overline{R}_i^{-1} \overline{R}_j Y + \overline{R}_i^{-1} \frac{\Xi_{j,k,t_k}-\Xi_{i,k,t_k}}{t_k\de_k} \bigg)\bigg|^2 dY \bigg)^{\frac12} \bigg( \int_{\RR_+^3 \setminus B_{\frac{D_{ij}}{2}}^{+}(Y)} |\nabla \Psi(Y)|^2 dY \bigg)^{\frac12} \\
    & \lesssim \frac{\|\psi\|_{L^1}}{D_{ij}^{\frac32}}  \bigg( \int_{\RR_+^3} \bigg|\nabla \Psi\bigg(\overline{R}_i^{-1} \overline{R}_j Y + \overline{R}_i^{-1} \frac{\Xi_{j,k,t_k}-\Xi_{i,k,t_k}}{t_k\de_k} \bigg)\bigg|^2 dY \bigg)^{\frac12} \lesssim (t_k \de_k)^{\frac32} \, \frac{\|\nabla \Psi\|_{L^2(\RR_+^3)} \|\psi\|_{L^1}}{|\Xi_{j,k,t_k}- \Xi_{i,k,t_k}|^{\frac32}}\,.
\end{align*}
Hence, it follows that
$$
|I_{ij}| \lesssim (t_k \de_k)^{\frac32} \, \frac{\|\nabla \Psi\|_{L^2(\RR_+^3)} \|\psi\|_{L^1}}{|\Xi_{j,k,t_k}- \Xi_{i,k,t_k}|^{\frac32}}\,,
$$
and so, by \cite[(3.73)]{PV19}, that
$$
\frac{1}{k^2} \Sjk \sum_{i\neq j} I_{ij} \lesssim \frac{\|\psi\|_{\Hud} \|\psi\|_{L^1}}{k^{\frac52} \log^3(k)}\,.
$$

Taking into account this estimate for $\psi_k^{\rm sym}$, we conclude that
$$
\langle \widetilde{ \phi}_k, \psi \rangle - 3 \int_{\RR^2} \widetilde{U}_{k,t_k}^2 \widetilde{\phi}_k \psi \, dx = o (\|\psi\|_{\Hud}) +  O \bigg( \frac{\|\psi\|_{L^1}}{(k \log k)^{\frac32}} \bigg)\,, \quad \textup{as } k \to \infty\,.
$$
Since $\widetilde{U}_{k,t_k} \to U$ in $C_{\rm loc}^0(\RR^2)$ as $k \to \infty$, sending $k \to \infty$ we conclude that
$$
\langle \widetilde{ \phi}_\infty, \psi \rangle - 3 \int_{\RR^2} U^2 \widetilde{\phi}_\infty \psi \, dx = 0\,,
$$
and thus the proof of Step 4.1.

\medbreak\noindent \textit{\underline{Step 4.2}.} $\widetilde{\phi}_{\infty} \equiv 0$.

By  \cite[Theorem 1.1]{DdPS13} and \cite[Lemma A.1]{CKW25} we know that $\widetilde{\phi}_\infty \in {\rm span}\{Z^{(0)}, Z^{(1)},Z^{(2)}\}$, where we recall that the functions $Z^{(\ell)}$, $\ell \in \{0,1,2\}$, were given \eqref{E.Z012}. Also, observe that
\begin{equation*}
\langle \phitl, Z^{(\ell)} \rangle = 3 \int_{\RR^2} U^2 Z^{(\ell)} \phitl\, dx \,, \quad \textup{for } \ell \in \{0,1,2\}\,.
\end{equation*}
To prove that $\phitl \equiv 0$ we will show that, for all $\ell \in \{0,1,2\}$, $\langle \phitl, \, Z^{(\ell)} \rangle = 0$. We first deal with $\ell = 2$. Using that $(\widetilde{\phi}_k)_k$ is even in $x_2$ for all $k$, we infer that $\phitl$ is even in $x_2.$ Thus, it is immediate to check that
$$
3 \int_{\RR^2} U^2 Z^{(2)} \phitl \, dx = 3 \int_{\RR^2} \underbrace{\frac{1}{1+|x|^2} \, \phitl(x)}_{\textup{Even in }x_2} \, \underbrace{\frac{x_2}{(1+|x|^2)^{\frac32}}}_{\textup{Odd in }  x_2} dx = 0\,,
$$
and so that
\begin{equation} \label{E.phitlZ2}
    \langle \phitl,\, Z^{(2)} \rangle = 0\,.
\end{equation}

Next, let move to $\ell=0$. Observe that $(\phi_k)_k \subset \cK_{k,t_k}^{\perp} \subset \cH_k$. Thus, it follows that
\begin{align*}
    0 & = \langle \phi_k, Z_{k,t_k} \rangle = 3 \int_{\RR^2} \bigg( \Sjk U_{j,k,t_k}^2 Z_{j,k,t_k} \bigg) \phi_k \, dx = 3k \int_{\RR^2} U_{1,k,t_k}^2 Z_{1,k,t_k} \phi_k\, dx \\
    & = 3k \int_{B_{\frac{1}{\sqrt{t_k \de_k}}}(0)} U(y)^2 Z^{(0)}(y) \widetilde{\phi}_k (y) \, dy + 3k \int_{\RR^2 \setminus \frac{1}{\sqrt{t_k\de_k}}(0)} U(y)^2 Z^{(0)}(y) \sqrt{t_k \de_k} \phi_k (t_k \de_k y + \xi_{1,k,t_k})\, dy\,.  
\end{align*}
Also, observe that
\begin{equation} \label{E.auxUZ0}
   \bigg| \int_{\RR^2 \setminus \frac{1}{\sqrt{t_k\de_k}}(0)} U(y)^2 Z^{(0)}(y) \sqrt{t_k \de_k} \phi_k (t_k \de_k y + \xi_{1,k,t_k})\, dy \bigg| \lesssim \|\phi_k\|_{L^4} \bigg( \int_{k\log k}^{\infty} \frac{d\rho}{\rho^3} \bigg)^{\frac34} \lesssim \frac{1}{k \log^{\frac32}(k)}\,.
\end{equation}
Sending $k \to \infty$ we conclude that
$$
3 \int_{\RR^2} U^2 Z^{(0)} \phitl \, dx = 0\,,
$$
and thus
\begin{equation} \label{E.phitlZ0}
    \langle \phitl,\, Z^{(0)} \rangle = 0\,.
\end{equation}
For later purposes, let us point out that, combining \eqref{E.auxUZ0} and \eqref{E.phitlZ0}, one can easily prove that
\begin{equation} \label{E.phitlZ0corollary}
 \int_{\RR^2} U(y)^2 Z^{(0)}(y) \sqrt{t_k \de_k} \phi_k(t_k\de_k y + \xi_{1,k,t_k}) dy = o(1)\,, \quad \textup{as } k \to \infty\,.
\end{equation}

Finally, consider $\ell=1$. As in \eqref{E.auxUZ0}, it follows that
$$
\bigg|\int_{\RR^2} U(y)^2 Z^{(1)}(y) (\eta(\sqrt{t_k \de_k}|y|)-1) \sqrt{t_k \de_k} \phi_k(t_k \de_k y + \xi_{1,k,t_k})\, dy\bigg| \lesssim \frac{1}{k^2 \log^{\frac52}(k)}\,.
$$
and so that
\begin{equation} \label{E.phitlZ11}
    \int_{\RR^2} U^2 Z^{(1)} \widetilde{\phi}_k\, dy = \int_{\RR^2} U(y)^2 Z^{(1)}(y) \sqrt{t_k \de_k} \phi_k(t_k\de_k y + \xi_{1,k,t_k}) \, dy + O \bigg( \frac{1}{k^2 \log^{\frac52}(k)}\bigg)\,, \quad \textup{as } k \to \infty\,.
\end{equation}
Henceforth, we focus on the first term of the right hand side. Since $\phi_k \subset \cK_{k,t_k}^{\perp} \subset \cH_k$ is invariant under the action of the Kelvin transform (see \eqref{E.kelvin}), it follows that
\begin{align*}
    & \int_{\RR^2} U(y)^2 Z^{(1)}(y)  \sqrt{t_k \de_k} \phi_k(t_k\de_k y + \xi_{1,k,t_k}) \, dy =  \int_{\RR^2} \frac{y_1}{(1+|y|^2)^{\frac52}} \sqrt{t_k \de_k} \phi_k(t_k \de_k y + \xi_{1,k,t_k})\, dy \\
    & \ = (t_k \de_k)^{\frac52} \int_{\RR^2} \frac{x_1-\sqrt{1-t_k^2\de_k^2}}{((t_k\de_k)^2+|x-\xi_{1,k,t_k}|^2)^{\frac52}} \phi_k(x)\, dx = (t_k \de_k)^{\frac52} \int_{\RR^2}  \frac{1}{|z|^3}  \frac{\frac{z_1}{|z|^2}-\sqrt{1-t_k^2\de_k^2}}{\big((t_k\de_k)^2+ \big|\frac{z}{|z|^2}-\xi_{1,k,t_k}\big|^2\big)^{\frac52}} \phi_k(z) \, dz \\
    & \  = (t_k \de_k)^{\frac52} \int_{\RR^2}  \frac{z_1-\sqrt{1-t_k^2\de_k^2} + (1-|z|^2) \sqrt{1-t_k^2\de_k^2}}{((t_k \de_k)^2+|z-\xi_{1,k,t_k}|^2)^{\frac52}}\, \phi_k(z)\, dz \\
    & \  = \int_{\RR^2} U(y)^2 Z^{(1)}(y) \sqrt{t_k \de_k} \phi_k(t_k\de_k y + \xi_{1,k,t_k}) \, dy + (t_k \de_k)^{\frac52} \sqrt{1-t_k^2\de_k^2} \int_{\RR^2} \frac{1-|z|^2}{((t_k \de_k)^2+|z-\xi_{1,k,t_k}|^2)^{\frac52}}\, \phi_k(z)\, dz\,,
\end{align*}
and so that
\begin{align*}
    0 & = (t_k \de_k)^{\frac52} \sqrt{1-t_k^2\de_k^2} \int_{\RR^2} \frac{1-|z|^2}{((t_k \de_k)^2+|z-\xi_{1,k,t_k}|^2)^{\frac52}}\, \phi_k(z)\, dz \\
    & = \frac{ \sqrt{1-t_k^2\de_k^2}}{t_k \de_k} \int_{\RR^2} \frac{1-|t_k\de_k y + \xi_{1,k,t_k}|^2}{(1+|y|^2)^{\frac52}} \sqrt{t_k\de_k} \phi_k(t_k \de_k y + \xi_{1,k,t_k})\, dy \\
    & = t_k \de_k \, \sqrt{1-t_k^2\de_k^2} \int_{\RR^2} \frac{1-|y|^2}{(1+|y|^2)^{\frac52}} \sqrt{t_k\de_k} \phi_k(t_k \de_k y + \xi_{1,k,t_k})\, dy \\
    & \quad - 2 (1-(t_k\de_k)^2) \int_{\RR^2} \frac{y_1}{(1+|y|^2)^{\frac52}} \sqrt{t_k\de_k} \phi_k(t_k \de_k y + \xi_{1,k,t_k})\, dy \\
    & = -2  t_k \de_k \, \sqrt{1-t_k^2\de_k^2} \int_{\RR^2} U(y)^2 Z^{(0)}(y) \sqrt{t_k\de_k} \phi_k(t_k \de_k y + \xi_{1,k,t_k})\, dy \\
    & \quad - 2 (1-(t_k\de_k)^2) \int_{\RR^2} U(y)^2 Z^{(1)}(y) \sqrt{t_k\de_k} \phi_k(t_k \de_k y + \xi_{1,k,t_k})\, dy.
\end{align*}
Sending $k \to \infty$, and taking \eqref{E.phitlZ0corollary} into account, we conclude that
$$
\int_{\RR^2} U(y)^2 Z^{(1)}(y) \sqrt{t_k \de_k} \phi_k(t_k\de_k y + \xi_{1,k,t_k}) dy = o(1)\,, \quad \textup{as } k \to \infty\,.
$$

Finally, using \eqref{E.phitlZ11}, we conclude that
$$
\int_{\RR^2} U^2 Z^{(1)} \phitl\, dx = 0\,,
$$
which in turn implies that
\begin{equation} \label{E.phitlZ1}
    \langle \phitl,\, Z^{(1)} \rangle = 0\,.
\end{equation}
Step 4.2 follows from \eqref{E.phitlZ2}, \eqref{E.phitlZ0} and \eqref{E.phitlZ1}. This also concludes the proof of Step 4. 

\medbreak
\noindent \textbf{Step 5.} \textit{Conclusion}. 

First of all, observe that
$$
 \int_{B_{\sqrt{t\de_k}}(\xi_{1,k,t_k})} U_{1,k,t_k}^2 \phi_k^2 dx = \int_{B_{\frac{1}{\sqrt{t_k \de_k}}}(0)} U(y)^2 \widetilde{\phi}_k(y)^2 dy\,.
$$
Hence, taking into account Step 4, we infer that
$$
\int_{B_{\sqrt{t\de_k}}(\xi_{1,k,t_k})} U_{1,k,t_k}^2 \phi_k^2 dx \to 0\,, \quad \textup{as } k \to \infty\,,
$$
and thus, by Step 2,
    \begin{equation} \label{E.Step5LTbeginning}
\int_{\Om_1} U_{k,t_k}^2 \phi_k^2 dx =  \int_{\Om_1} U^2 \phi_k^2 dx + o(1)\,, \quad \textup{as } k \to \infty\,.
\end{equation}

Now, note that
$$
\int_{\Om_1} U^2 \phi_k^2 dx = \int_{\RR^2} U^2 \Big( \frac{\phi_k}{\sqrt{k}} \Big)^2 dx\,. 
$$
Hence, if we prove that
\begin{equation} \label{E.Step5goal}
\int_{\RR^2} U^2 \Big( \frac{\phi_k}{\sqrt{k}} \Big)^2 dx = o(1)\,, \quad \textup{as } k \to \infty\,,
\end{equation}
\eqref{uphi} will be in contradiction with \eqref{E.Step5LTbeginning}, and the result will follow. The rest of the proof is devoted to showing that \eqref{E.Step5goal} holds. To that end, we will prove that, up to a subsequence, $k^{-\frac12} \phi_k \rightharpoonup 0$ in $\Hud$. 

First, by \eqref{E.startingR2},  
$$
\langle k^{-\frac12} \phi_k, \psi \rangle - 3 \int_{\RR^2} U_{k,t_k}^2 k^{-\frac12} \phi_k \psi\, dx = o(\|\psi\|_{\Hud})\,, \quad \textup{as } k \to \infty\,,
$$
uniformly in $\psi \in \cK_{k,t_k}^{\perp}$. Also, observe that, setting $\psi_k := \sqrt{k} \,\psi$ for all $k \in \NN$,
\begin{align*}
& \int_{\RR^2} (U_{k,t_k}^2 -U^2) k^{-\frac12} \phi_k \psi\, dx  = \int_{\Om_1} (U_{k,t_k}^2 - U^2) \phi_k \psi_k\, dx = \int_{\Om_1} \bigg( \bigg(\Sjk U_{j,k,t_k} \bigg)^2 - 2U \Sjk U_{j,k,t_k} \bigg) \phi_k \psi_k\, dx  \\
& \quad =  \int_{B_{\sqrt{t_k\de_k}}(\xi_{1,k,t_k})} U_{1,k,t_k}^2 \phi_k \psi_k\, dx + \int_{\Om_1 \setminus B_{\sqrt{t_k\de_k}}(\xi_{1,k,t_k})} U_{1,k,t_k}^2 \phi_k \psi_k\, dx \\
& \qquad + \int_{\Om_1} \bigg(\bigg(\sum_{j=2}^k U_{j,k,t_k}\bigg)^2 + 2 U_{1,k,t_k} \sum_{j=2}^k U_{j,k,t_k}
-2UU_{1,k,t_k}-2U\sum_{j=2}^kU_{j,k,t_k}\bigg)\phi_k \psi_k\, dx\,.
\end{align*}
Hence, arguing as in Step 2, we get that
$$
\int_{\RR^2} (U_{k,t_k}^2 -U^2) k^{-\frac12} \phi_k \psi\, dx =  \int_{B_{\sqrt{t_k\de_k}}(\xi_{1,k,t_k})} U_{1,k,t_k}^2 \phi_k \psi_k\, dx + O \bigg( \frac{\|\psi\|_{\Hud}}{\log k} \bigg)\,, \quad \textup{as } k \to \infty\,.
$$

On the other hand,  
$$
 \int_{B_{\sqrt{t_k\de_k}}(\xi_{1,k,t_k})} U_{1,k,t_k}^2 \phi_k \psi_k\, dx = \int_{B_{\frac{1}{\sqrt{t_k \de_k}}}(0)} U^2 \widetilde{\phi}_k \widetilde{\psi}_k\, dy\,,
$$
where, for all $k \in \NN$, $\widetilde{\phi}_k$ was defined in \eqref{E.phikTilde} and, similarly, we set
\begin{equation} \label{E.psikTilde}
    \widetilde{\psi}_k(y):=\sqrt{t_k\delta_k}\eta(\sqrt{t_k\delta_k} |y|)\psi_k(t_k\delta_k y +\xi_{1,k,t_k})\,, \quad \textup{for all } k \in \NN\,.
\end{equation}
By Steps 3 and 4, we know that $(\widetilde{\phi}_k)_k$ is bounded in $\Hud$ and that $\widetilde{\phi}_k \to 0$ a.e. in $\RR^2$. Furthermore, arguing exactly as in Step 3, we get that $(\widetilde{\psi}_k)_k$ is bounded in $\Hud$ and thus in $L^4$. Hence, we conclude that
$$
\int_{B_{\sqrt{t_k\de_k}}(\xi_{1,k,t_k})} U_{1,k,t_k}^2 \phi_k \psi_k\, dx = o(\|\psi\|_{\Hud})\,, \quad \textup{as } k \to \infty\,,
$$
and so that
\begin{equation*}
    \int_{\RR^2} (U_{k,t_k}^2 -U^2) k^{-\frac12} \phi_k \psi\, dx = o(\|\psi\|_{\Hud})\,, \quad \textup{as } k \to \infty\,.
\end{equation*}
This in turn implies that
\begin{equation} \label{E.phisqrtk-limitlinearized}
    \langle k^{-\frac12}\phi_k, \psi \rangle - 3 \int_{\RR^2} U^2 k^{-\frac12} \phi_k \psi\, dx = o(\|\psi\|_{\Hud})\,, \quad \textup{as } k \to \infty\,.
\end{equation}

Next, observe that $(k^{-\frac12}\phi_k)_k$ is bounded in $\Hud$. Hence, up to a subsequence, 
$$
k^{-\frac12} \phi_k \rightharpoonup \phi_\infty \quad \textup{in } \Hud \quad \textup{and} \quad k^{-\frac12} \phi_k \to \phi_\infty \quad \textup{in } L_{\rm loc}^p\,, \ 2 \leq p < 4\,, \quad \textup{as } k \to \infty\,,
$$
for some $\phi_{\infty} \in \Hud$. Moreover, using \eqref{E.phisqrtk-limitlinearized}, we get that $\phi_\infty$ is a solution to
$$
(-\Delta)^{\frac12} \phi_{\infty} - 3 U^2 \phi_{\infty} = 0 \quad \textup{in } \RR^2\,.
$$
Also, let us point out that, since $(\phi_k)_k \subset \cH_k$, $\phi_\infty$ is invariant under the action of the Kelvin transform and  radial. Hence, we have that
$$
3 \int_{\RR^2} U^2   Z^{(2)} \phi_\infty \, dx  = 3 \int_{\RR^2} \underbrace{\frac{1}{1+|x|^2} \, \phi_\infty(x)}_{\textup{Even in }x_2} \, \underbrace{\frac{x_2}{(1+|x|^2)^{\frac32}}}_{\textup{Odd in }  x_2} dx = 0 \,.
$$
Likewise, using that $\bK[U]=U$ and that $\bK[Z^{(0)}] = - Z^{(0)}$, we get that
\begin{align*}
 & \int_{\RR^2} U^2  Z^{(0)} \phi_\infty\, dx = \int_{\RR^2} \frac{1}{|z|^4} U\Big( \frac{z}{|z|^2} \Big)^2 Z^{(0)} \Big( \frac{z}{|z|^2} \Big) \phi_\infty\left(\frac{z}{|z|^2}\right) \, dz \\
 & \quad = \int_{\RR^2} (\bK[U](z))^2 \bK[Z^{(0)}](z) \bK[\phi_\infty](z) \, dz   = - \int_{\RR^2} U^2  Z^{(0)} \phi_\infty\, dz\,,
\end{align*}
and so that
$$
3 \int_{\RR^2} U^2   Z^{(0)} \phi_\infty \, dx = 0\,.
$$
Finally, let $\te \in (0,2\pi)$ be fixed but arbitrary. Then,
\begin{align*}
    & \int_{\RR^2} U^2 Z^{(1)} \phi_\infty\, dx = \int_{\RR^2} U^2 Z^{(1)}(\cR_{\theta} \cdot) \phi_\infty\, dx = \cos(\theta) \int_{\RR^2} U^2 Z^{(1)} \phi_\infty\, dx - \sin(\theta) \int_{\RR^2} U^2 Z^{(2)} \phi_\infty\, dx\,.
\end{align*}
Since we know that 
$$
\int_{\RR^2} U^2 Z^{(2)} \phi_{\infty}\, dx = 0\,,
$$
we conclude that
$$
3 \int_{\RR^2} U^2 Z^{(1)} \phi_{\infty}\, dx = 0\,. 
$$

In short, we have that
$$
\langle \phi_{\infty},\, Z^{(\ell)} \rangle = 3 \int_{\RR^2} U^2 Z^{(\ell)} \phi_{\infty}\, dx = 0\,, \quad \ell \in \{0,1,2\}\,,
$$
which in turn implies that $\phi_\infty \equiv 0$. This allows us to conclude that \eqref{E.Step5goal} holds, and so that
$$
\int_{\Om_1} U_{k,t_k}^2 \phi_k^2 dx = o(1)\,, \quad \textup{as } k \to \infty\,,
$$
reaching a contradiction with \eqref{uphi}. The proof is concluded. 
\end{proof}

\begin{corollary} \label{C.aprioriEstimate}
For every $0 < a < b < \infty$, let $k_1 \in \NN$ and $\cC_1 > 0$ be as in Lemma \ref{L.aprioriEstimate}. Then, for all $k \geq k_1$ and all $t \in [a,b]$, the inverse operator $(\cL_{k,t}^{\perp,*})^{-1}: \cK_{k,t}^{\perp} \to \cK_{k,t}^{\perp}$ exists and is continuous. In particular,
$$
\|(\cL_{k,t}^{\perp,*})^{-1}\phi\|_{\dot{H}^{\frac12}} \leq \cC_1 \|\phi\|_{\dot{H}^{\frac12}},  \quad \textup{for all } \phi \in \cK_{k,t}^{\perp}\,.
$$
\end{corollary}
\begin{proof} 
First of all, fix $k \geq k_1$ and $t \in [a,b]$, and observe that
$$
\cL_{k,t}^{\perp,*}\phi = \phi - \Pi^{\perp}(-\Delta)^{-\frac12} (3 U_{k,t}^2 \phi)\,, \quad \textup{for all } \phi \in \cK_{k,t}^{\perp}\,.
$$
Since $\Pi^{\perp}: \cH_k \subset \Hud \to \cK_{k,t}^{\perp}$ is continuous, if we show that
$$
\phi \mapsto (-\Delta)^{-\frac12}(3 U_{k,t}^2 \phi)\,,
$$
is compact in $\Hud$, then it follows that 
$$
\phi \mapsto \Pi^{\perp}(-\Delta)^{-\frac12} (3 U_{k,t}^2 \phi)\,,
$$
is compact in $\cK_{k,t}^{\perp}$, and so that $\cL_{k,t}^{\perp,*}:  \cK_{k,t}^{\perp} \to \cK_{k,t}^{\perp}$ is Fredholm of index zero. 

Assume this is indeed the case. Since by Lemma \ref{L.aprioriEstimate} we know that $\cL_{k,t}^{\perp,*}$ is an injective map, it then follows from the fact that it is Fredholm of index zero that it is also onto, and therefore bijective. Once we have the invertibility of $\cL_{k,t}^{\perp,*}$, the estimate for the inverse is again an immediate corollary of Lemma \ref{L.aprioriEstimate}. Hence, to conclude the proof, we are just missing to show that 
\begin{equation} \label{E.compactnessAim}
\phi \mapsto (-\Delta)^{-\frac12}(3 U_{k,t}^2 \phi)\,,\ \textup{ is compact in } \Hud\,.
\end{equation}

Taking into account the decay properties of $U_{k,t}$, and the fact that $U_{k,t}^2 \in L^2 \cap L^4$, one can easily prove that this operator is compact in $\Hud$. Indeed, let $(\phi_n)_{n=1}^\infty \subset \Hud$ be a bounded sequence. Using Kolmogorov-Riesz compactness theorem (see e.g. \cite{HH10}), one can check that
$$
\big((-\Delta)^{-\frac14}(3U_{k,t}^2 \phi_n)\big)_{n=1}^{\infty}\,,
$$
admits a strongly convergent subsequence in $L^2$. This implies that \eqref{E.compactnessAim} holds and concludes the proof.
\end{proof}

\subsection{Nonlinear theory} Now, using the error estimate we proved in Lemma \ref{L.error} and the linear theory we established in the previous subsection, we deal with the second equation in \eqref{E.phiSystem} and prove the following.

\begin{proposition}  \label{P.contraction}
Let $0 < a < b < \infty$ and let $k_1 \in \NN$ be as in Lemma \ref{L.aprioriEstimate}. There exist $k_2 \in \NN$ with $k_2 \geq k_1$ and $\cC_2 > 0$ such that, for all $k \geq k_2$ and all $t \in [a,b]$, there exists a unique solution $\phi = \phi_{k,t}$ to
\begin{equation} \label{E.nonlinearProjected}
\cL_{k,t}^{\perp,*} \phi - \cE_{k,t}^{\perp,*} - \cN_{k,t}^{\perp,*}(\phi) = 0\,, \quad \phi \in \cK_{k,t}^{\perp}\,,
\end{equation}
such that
\begin{equation} \label{E.nonlinearSize}
\|\phi_{k,t}\|_{\dot{H}^{\frac12}}^2 + \int_{\RR^2} U_{k,t}^2 \phi_{k,t}^2 dx \leq \frac{\cC_2}{\log^2(k)}\,.
\end{equation}
Moreover, the map $t \mapsto \phi_{k,t}$ is continuously differentiable. 
\end{proposition}

\begin{proof}
For every $\cC >0$, $k \in \NN$ with $k \geq 2$, and $a < t < b$, we set
    $$
    \cX_{k,t}(\cC) := \left\{ \phi \in \cK_{k,t}^{\perp} : \|\phi\|_{\dot{H}^{\frac12}}^2 + \int_{\RR^2} U_{k,t}^2 \phi^2 dx \leq \frac{\cC}{\log^2(k)} \right\}.
    $$
To show the existence of a unique solution to \eqref{E.nonlinearProjected} satisfying \eqref{E.nonlinearSize}, we will prove that if $k > 0$ and $\cC> 0$ are large enough, the operator
    $$
    \cF_{k,t}: \cK_{k,t}^{\perp} \to \cK_{k,t}^{\perp}\,, \qquad \phi \mapsto (\cL_{k,t}^{\perp,*})^{-1} \Big(\cE_{k,t}^{\perp,*} + \cN_{k,t}^{\perp,*}(\phi) \Big)\,,
    $$
has a unique fixed point in $\cX_{k,t}(\cC)$. To that end, we will use the contraction mapping theorem. 

First of all, by Corollary \ref{C.aprioriEstimate}, we know that
    $$
    \|\cF_{k,t}(\phi)\|_{\Hud} \leq 2 \cC_1 \big( \|\cE_{k,t}^{*}\|_{\Hud} + \|\cN_{k,t}^{*}(\phi)\|_{\Hud} \big),
    $$
and that
    \begin{align*}
    & \int_{\RR^2} U_{k,t}^2 (\cF_{k,t}(\phi))^2 dx  = \frac13 \langle \cF_{k,t}(\phi) - \cE_{k,t}^{*}- \cN_{k,t}^{*}(\phi), \cF_{k,t}(\phi) \rangle \\
    & \quad \leq \frac13 \Big( \|\cF_{k,t}(\phi)\|_{\Hud}^2 + \|\cF_{k,t}(\phi)\|_{\Hud} \big( \|\cE_{k,t}^{*}\|_{\Hud} + \|\cN_{k,t}^{*}(\phi)\|_{\Hud} \big) \Big) \\
    & \quad \leq \frac{2\cC_1(2\cC_1+1)}{3} \big( \|\cE_{k,t}^{*}\|_{\Hud} + \|\cN_{k,t}^{*}(\phi)\|_{\Hud} \big)^2,
    \end{align*}
for all $k \geq k_1$, all $t \in [a,b]$ and all $\phi \in \cK_{k,t}^{\perp}$.

Next, observe that
\begin{align*}
    & \|\cN_{k,t}^{*}(\phi)\|_{\Hud} \lesssim \|\cN_{k,t}(\phi)\|_{L^{\frac43}} = \|3 U_{k,t} \phi^2 + \phi^3\|_{L^{\frac43}} \leq 3\left( \int_{\RR^2} U_{k,t}^2 \phi^2 dx \right)^{\frac12} \|\phi\|_{L^4} + \|\phi\|_{L^4}^3 \\
    & \quad \lesssim \left( \int_{\RR^2} U_{k,t}^2 \phi^2 dx \right)^{\frac12} \|\phi\|_{\Hud} + \|\phi\|_{\Hud}^3 \leq \frac{\cC}{\log^2(k)} \left( 1 + \frac{\cC^{\frac12}}{\log k} \right),
\end{align*}
for all $k \geq k_1$, all $t \in [a,b]$ and all $\phi \in \cX_{k,t}(\cC)$. Let us stress that the implicit constants in the above chain of inequalities are independent of $k$, $t$ and $\cC$. Hence, taking into account Lemma \ref{L.error}, we get the existence of $k_2 \in \NN$ with $k_2 \geq k_1$, and $\cC_2 > 0$ such that
$$
\cF_{k,t}(\cX_{k,t}(\cC_2)) \subset \cX_{k,t}(\cC_2)\,, \quad \textup{for all } k \geq k_2 \textup{ and all } t \in [a,b]\,.
$$

Hence, to apply the contraction mapping theorem, it remains to show that,  for all $k \geq k_2$ and all $t \in [a,b]$, $\cF_{k,t}$ is a contraction mapping in $\cX_{k,t}(\cC_2)$. By Corollary \ref{C.aprioriEstimate} we know that
$$
\|\cF_{k,t}(\phi_1) - \cF_{k,t}(\phi_2)\|_{\Hud}\leq 2\cC_1 \|\cN_{k,t}^{*}(\phi_1) - \cN_{k,t}^{*}(\phi_2)\|_{\Hud}\,, \quad \textup{for all } \phi_1, \phi_2 \in \cX_{k,t}(\cC_2),\ k \geq k_2 \textup{ and } t \in [a,b]\,.
$$
Also, observe that
\begin{align*}
    & \|\cN_{k,t}^{*}(\phi_1) - \cN_{k,t}^{*}(\phi_2)\|_{\Hud} \lesssim \|\cN_{k,t}(\phi_1) - \cN_{k,t}(\phi_2)\|_{L^{\frac43}} \\
    & \quad  \leq 3 \|  U_{k,t} (\phi_1+\phi_2)(\phi_1-\phi_2)\|_{L^{\frac43}} + \| (\phi_1^2 + \phi_1\phi_2 + \phi_2^2)(\phi_1-\phi_2)\|_{L^{\frac43}} \\
    & \quad \lesssim \left( \int_{\RR^2} U_{k,t}^2 \phi_1^2 dx + \int_{\RR^2} U_{k,t}^2 \phi_2^2 dx \right)^{\frac12} \|\phi_1-\phi_2\|_{L^4} + \left(\|\phi_1\|_{L^4}^2 + \|\phi_2\|_{L^4}^2 \right)  \|\phi_1-\phi_2\|_{L^4} \\
    & \quad \lesssim  \bigg( \left( \int_{\RR^2} U_{k,t}^2 \phi_1^2 dx \right)^{\frac12}  + \left(  \int_{\RR^2} U_{k,t}^2 \phi_2^2 dx \right)^{\frac12} + \|\phi_1\|_{\Hud}^2 + \|\phi_2\|_{\Hud}^2  \bigg)  \|\phi_1-\phi_2\|_{\Hud} \\
    & \quad \leq \left( \frac{(2\cC_2)^{\frac12}}{\log k} + \frac{2 \cC_2}{\log^2(k)} \right) \|\phi_1-\phi_2\|_{\Hud}\,, 
\end{align*}
for all $\phi_1, \phi_2 \in \cX_{k,t}(\cC_2)$, $k \geq k_2$ and $t \in [a,b]$.

As before, the implicit constants in these chains of inequalities are independent of $k$, $t$ and $\cC_2$. Taking $k_2$ bigger if necessary we conclude that, for all $k \geq k_2$ and all $t \in [a,b]$, $\cF_{k,t}$ is indeed a contraction mapping in $\cX_{k,t}(\cC_2)$.

At this point, the contraction mapping theorem guarantees the existence of a unique solution $\phi_{k,t}$ to \eqref{E.nonlinearProjected} belonging to $\cX_{k,t}(\cC_2)$. Also, in this setting, the continuous differentiability of $t \mapsto \phi_{k,t}$ is by now standard (see for instance \cite{RV11}). 
\end{proof}

\subsection{Energy expansion} Having Proposition \ref{P.contraction} at hand, it remains to treat the first equation in \eqref{E.phiSystem}, which is one dimensional. To that end, a suitable energy expansion will be crucial. More precisely, we introduce the energy functional associated to the Yamabe equation
$$
I : \cH_k \to \RR\,, \qquad u \mapsto \frac12 \int_{\RR^2} |(-\Delta)^{\frac14} u|^2 dx - \frac14 \int_{\RR^2} u^4 dx\,,
$$
and devote this section to proving the following:

\begin{proposition} \label{P.energyExpansion}
Let $0 < a < b < \infty$ and let $k_2 \in \NN$ be as in Proposition \ref{P.contraction}. Also, for $k \in \NN$ with $k \geq k_2$, let $\phi_{k,t} \in \cK_{k,t}^{\perp}$ be as in Proposition \ref{P.contraction}. Then
\begin{equation} \label{E.energyExpansionFinal}
I(U_{k,t}+\phi_{k,t}) =  \frac{k+1}{4} \|U\|_{L^4}^4 + \frac{1}{\log k} \Big( \pi \sqrt{2t}  - t \Big) +  o \Big( \frac{1}{\log k} \Big)\,, \quad \textup{as } k \to \infty
\,,
\end{equation}
uniformly in $t \in [a,b]$.
\end{proposition}

\begin{remark} \label{R.strictGlobalMaximum}
    The function $f: \RR_+ \to \RR$ given by $f(t) := \pi \sqrt{2t}  - t$ has a global maximum at $t_{\rm max} := \pi^2/2$. This global maximum is non-degenerate ($f''(t_{\rm max}) = - \frac{1}{\pi^2} < 0$) and thus isolated. Actually, any function of the form $t \mapsto \mathbf{c}_1 \sqrt{t} - \mathbf{c}_2 t$ with $\mathbf{c}_1 > 0$ and $\mathbf{c}_2 > 0$ has a non-degenerate global maximum at $\mathbf{t}_{\rm max} := \mathbf{c}_1^2/(4 \mathbf{c}_2^2).$
\end{remark}

We start by expanding the energy of the ansatz $U_{k,t}$. This will give us the main order in \eqref{E.energyExpansionFinal}.

\begin{lemma} \label{L.energyExpansionAnsatz}
Let $0 < a < b < \infty$. Then
$$
I(U_{k,t}) = \frac{k+1}{4} \|U\|_{L^4}^4 + \frac{1}{\log k} \Big( \pi \sqrt{2t}  - t \Big) +  o \Big( \frac{1}{\log k} \Big) \,, \quad \textup{as } k \to \infty\,,
$$
uniformly in $t \in [a,b]$.

\end{lemma}
 
\begin{proof}
First, using that $U$ and $U_{j,k,t}$ are solutions to \eqref{E.Yamabe}, we directly expand the energy functional as
\begin{align*}
 I(U_{k,t})& = \frac{1}{2}\int_{\mathbb{R}^2}\left|(-\Delta)^{\frac{1}{4}}U_{k,t}\right|^2 dx  -\frac{1}{4}\int_{\mathbb{R}^2 } U_{k,t}^4dx\\
 &=\frac{1}{2}\int_{\RR^2}U_{k,t}\Dud U_{k,t}dx-\frac{1}{4}\int_{\RR^2}U_{k,t}^4 dx\\
 &= \frac{1}{2}\int_{\RR^2}\bigg(U^4-U\Sjk U_{j,k,t}^3-U^3\Sjk U_{j,k,t}+\Sjk\sum_{i=1}^k U_{j,k,t}U^3_{i,k,t}-\frac{1}{2} U_{k,t}^4\bigg)dx\\
 &= -\frac{1}{4}\int_{\RR^2}\bigg(U_{k,t}^4-U^4-\Sjk U_{j,k,t}^4+4 U\Sjk U_{j,k,t}^3+4U^3\Sjk U_{j,k,t}-4\Sjk\sum_{i\neq j}U_{j,k,t}^3U_{i,k,t}\bigg)dx\\
 &\quad +\int_{\mathbb{R}^2}\bigg(-\frac{1}{4}U^4-\frac{1}{4}\Sjk U_{j,k,t}^4+U\Sjk U_{j,k,t}^3+ U^3\Sjk U_{j,k,t}-\Sjk\sum_{i\neq j}U_{j,k,t}^3U_{i,k,t}\bigg)dx\\
 & \quad +\int_{\RR^2}\bigg(\frac{1}{2}U^4-\frac{1}{2}U\Sjk U_{j,k,t}^3-\frac{1}{2} U^3\Sjk U_{j,k,t}+\frac{1}{2}\Sjk U_{j,k,t}^4+\frac{1}{2}\Sjk\sum_{i\neq j}U_{j,k,t}^3U_{i,k,t}\bigg)dx\\
 & =\frac{k+1}{4}\int_{\RR^2}U^4 dx\\
 & \quad -\frac{1}{4}\int_{\RR^2}\bigg(U_{k,t}^4-U^4-\Sjk U_{j,k,t}^4+4 U\Sjk U_{j,k,t}^3+4U^3\Sjk U_{j,k,t}-4\Sjk\sum_{i\neq j}U_{j,k,t}^3U_{i,k,t}\bigg)dx\\
 & \quad +\frac{1}{2}\int_{\RR^2}\bigg( U\Sjk U_{j,k,t}^3-\Sjk \sum_{i\neq j} U_{j,k,t}^3U_{i,k,t}+U^3\Sjk U_{j,k,t}\bigg)dx.
\end{align*}
Moreover, we further rewrite the second last integral as 
\begin{align*}
\int_{\RR^2}&\bigg(U_{k,t}^4-U^4-\Sjk U_{j,k,t}^4+4 U\Sjk U_{j,k,t}^3+4U^3\Sjk U_{j,k,t}-4\Sjk\sum_{i\neq j}U_{j,k,t}^3U_{i,k,t}\bigg)dx\\
& = \int_{\RR^2}\bigg(U_{k,t}^4-U^4-\bigg(\Sjk U_{j,k,t}\bigg)^4+4 U\bigg(\Sjk U_{j,k,t}\bigg)^3+4U^3\Sjk U_{j,k,t}\bigg)dx\\
& \quad + \int_{\RR^2}\bigg(\bigg(\Sjk U_{j,k,t}\bigg)^4-\Sjk U_{j,k,t}^4-4\Sjk\sum_{i\neq j}U_{j,k,t}^3U_{i,k,t}\bigg)dx\\
& \quad -4\int_{\RR^2}U\bigg(\bigg(\Sjk U_{j,k,t}\bigg)^3-\Sjk U_{j,k,t}^3\bigg)dx,
\end{align*}
and notice that
\begin{align*}
    &\int_{\RR^2}\bigg(U_{k,t}^4-U^4-\bigg(\Sjk U_{j,k,t}\bigg)^4+4 U\bigg(\Sjk U_{j,k,t}\bigg)^3+4U^3\Sjk U_{j,k,t}\bigg)dx = 6 \int_{\RR^2}U^2\bigg(\Sjk U_{j,k,t}\bigg)^2 dx,
\end{align*}
and that
$$
\int_{\RR^2} U^3 \Sjk U_{j,k,t} dx = \int_{\RR^2} \Dud U \Sjk U_{j,k,t} dx  = \int_{\RR^2} U \Sjk \Dud U_{j,k,t } = \int_{\RR^2} U \Sjk U_{j,k,t}^3 dx.
$$
Hence, it follows that
\begin{equation} \label{E.energyExpansion}
\begin{aligned}
I(U_{k,t})& =\frac{k+1}{4}\|U\|_{L^4}^4+\frac{1}{2}\int_{\RR^2}\bigg(2U\Sjk U^3_{j,k,t} -\Sjk \sum_{i\neq j}U_{j,k,t}^3 U_{i,k,t} \bigg) dx - \frac32 \int_{\RR^2}U^2\bigg(\Sjk U_{j,k,t}\bigg)^2 dx\\
& \quad - \frac{1}{4}\int_{\RR^2}\bigg(\bigg(\Sjk U_{j,k,t}\bigg)^4-\Sjk U_{j,k,t}^4-4\Sjk\sum_{i\neq j}U_{j,k,t}^3U_{i,k,t}\bigg)dx\\
& \quad + \int_{\RR^2}U\bigg(\bigg(\Sjk U_{j,k,t}\bigg)^3-\Sjk U_{j,k,t}^3\bigg)dx.
\end{aligned}   
\end{equation}

Having this expression of $I(U_{k,t})$ at hand, we can now prove the desired expansion. We first estimate the higher order terms. First, observe that
\begin{equation*}
    \int_{\RR^2} U^2 \bigg( \Sjk U_{j,k,t} \bigg)^2 dx = k \int_{\Om_1}  U^2 \bigg( \Sjk U_{j,k,t} \bigg)^2 dx \lesssim k \int_{\Om_1} U^2 U_{1,k,t}^2 dx + k \int_{\Om_1} U^2 \bigg( \sum_{j=2}^k U_{j,k,t} \bigg)^2 dx.
\end{equation*}
Then, applying Lemma \ref{L.PremoselliVetois} twice, first with $\al = 2$ and $\be = 0$, and then with $\al = 0 $ and $\be =2$, we get that
\begin{equation} \label{E.energyR1}
     \int_{\RR^2} U^2 \bigg( \Sjk U_{j,k,t} \bigg)^2 dx \lesssim k \de_k \bigg( \frac{\log k}{k} + |\log \de_k| \bigg) + k \de_k \bigg( k + \frac{k^2 \log^2(k)}{k^2} \bigg) \lesssim \frac{1}{\log^2(k)}\,.
\end{equation}

Now, by direct but tedious computations, one can check that
\begin{align*}
    \int_{\RR^2}& \bigg(\bigg(\Sjk U_{j,k,t}\bigg)^4-\Sjk U_{j,k,t}^4-4\Sjk\sum_{i\neq j}U_{j,k,t}^3U_{i,k,t}\bigg)dx \\
    & \lesssim k \int_{\Om_1} U_{1,k,t} \bigg( \sum_{j=2}^k U_{j,k,t} \bigg)^3 dx + k \int_{\Om_1} U_{1,k,t}^2 \bigg( \sum_{j=2}^k U_{j,k,t} \bigg)^2 dx + k \int_{\Om_1} \bigg( \sum_{j=2}^k U_{j,k,t} \bigg)^4 dx. 
\end{align*}
The three terms on the right hand side can be estimated with Lemma \ref{L.PremoselliVetois}. Indeed, for the first one, we apply the lemma with $\al = 1$ and $\be = 3$; for the second term, we choose $\al = \be = 2$; and for the last, we take $\al = 0$ and $\be = 4$. Then, we conclude that
\begin{equation} \label{E.energyR2}
     \int_{\RR^2} \bigg(\bigg(\Sjk U_{j,k,t}\bigg)^4-\Sjk U_{j,k,t}^4-4\Sjk\sum_{i\neq j}U_{j,k,t}^3U_{i,k,t}\bigg)dx \lesssim \frac{1}{\log^4(k)}.
\end{equation}

Similarly, it follows that
$$
\begin{aligned}
\int_{\RR^2} & U\bigg(\bigg(\Sjk U_{j,k,t}\bigg)^3-\Sjk U_{j,k,t}^3\bigg)dx \\
& \lesssim k \int_{\Om_1} U U_{1,k,t}^2 \bigg( \sum_{j=2}^k U_{j,k,t} \bigg)  dx + k \int_{\Om_1} U U_{1,k,t} \bigg( \sum_{j=2}^k U_{j,k,t} \bigg)^2  dx + k \int_{\Om_1} U  \bigg( \sum_{j=2}^k U_{j,k,t} \bigg)^3  dx.
\end{aligned}
$$
Using again Lemma \ref{L.PremoselliVetois} we can estimate the three terms on the right hand side and conclude that
\begin{equation} \label{E.energyR3}
  \int_{\RR^2} U\bigg(\bigg(\Sjk U_{j,k,t}\bigg)^3-\Sjk U_{j,k,t}^3\bigg)dx  \lesssim \frac{1}{\log^3(k)}.  
\end{equation}
Thus, plugging \eqref{E.energyR1}, \eqref{E.energyR2} and \eqref{E.energyR3} into \eqref{E.energyExpansion}, we infer that
\begin{equation} \label{E.energyExpansion2}
    \begin{aligned}
        I(U_{k,t}) = \frac{k+1}{4}\|U\|_{L^4}^4+\frac{1}{2}\int_{\RR^2}\bigg(2U\Sjk U^3_{j,k,t} -\Sjk \sum_{i\neq j}U_{j,k,t}^3 U_{i,k,t} \bigg) dx + O \bigg( \frac{1}{\log^2(k)} \bigg)\,, \quad \textup{as } k \to \infty,
    \end{aligned}
\end{equation}
uniformly in $t \in [a,b]$. 

Next, we fix $\eta \in (0,\tfrac12)$ independent of $k$ such that $B_{\frac{\eta}{k}}(\xi_{1,k,t}) \cap B_{\frac{\eta}{k}}(\xi_{j,k,t}) = \emptyset$ for all $j \in \{2, \ldots,k\}$. Then, we rewrite the second term in \eqref{E.energyExpansion2} as
\begin{align*}
    & \int_{\RR^2} U \Sjk U_{j,k,t}^3 dx  = k \int_{\Om_1} U \Sjk U_{j,k,t}^3 dx = k \int_{\Om_1} U U_{1,k,t}^3 dx + k \int_{\Om_1} U \sum_{j=2}^k U_{j,k,t}^3 \\
    & \ = k \int_{B_{\frac{\eta}{k}}(\xi_{1,k,t})} U U_{1,k,t}^3 dx + k \int_{(\Om_1 \setminus B_{\frac{\eta}{k}}(\xi_{1,k,t})) \cap B_2(0)} U U_{1,k,t}^3 dx + k \int_{\Om_1 \setminus B_2(0)} U U_{1,k,t}^3 dx + k  \int_{\Om_1} U \sum_{j=2}^k U_{j,k,t}^3 dx.
\end{align*}
Then, note that, as in the proof of \eqref{E.energyR3}, it follows that
$$
k \int_{\Om_1} U \sum_{j=2}^k U_{j,k,t}^3 dx \leq k \int_{\Om_1} U \bigg(\sum_{j=2}^k U_{j,k,t} \bigg)^3 dx \lesssim \frac{1}{\log^3(k)}.
$$
Likewise, observe that
$$
U_{1,k,t}(x) \leq  \frac{k}{\eta} \sqrt{t \de_k} \lesssim \frac{1}{\log k} \qquad \textup{in } \Om_1 \setminus B_{\frac{\eta}{k}}(\xi_{1,k,t}), 
$$
and that
$$
U_{1,k,t}(x) \leq \frac{2}{|x|} \sqrt{t \de_k} \lesssim \frac{1}{ |x| k \log k } \qquad \textup{in } \Om_1 \setminus B_2(0).
$$
Hence, it follows that
$$
k \int_{(\Om_1 \setminus B_{\frac{\eta}{k}}(\xi_{1,k,t})) \cap B_2(0)} U U_{1,k,t}^3 dx + k \int_{\Om_1 \setminus B_2(0)} U U_{1,k,t}^3 dx + k  \int_{\Om_1} U \sum_{j=2}^k U_{j,k,t}^3 dx \lesssim \frac{1}{\log^3(k)}. 
$$

On the other hand, we have that
\begin{align*}
    & k \int_{B_{\frac{\eta}{k}}(\xi_{1,k,t})} U U_{1,k,t}^3 dx = k \sqrt{t\de_k} \int_{B_{\frac{\eta}{k t \de_k}}(0)} \frac{1}{(1+|t \de_k y + \xi_{1,k,t}|^2)^{\frac12}} \frac{1}{(1+|y|^2)^{\frac32}} dy  \\
    & \quad =  k \frac{\sqrt{t \de_k}}{(1+|\xi_{1,k,t}|^2)^{\frac12}} \int_{B_{\frac{\eta}{k t \de_k }}(0)} \frac{1}{(1+|y|^2)^{\frac32}} dy \\
    & \qquad 
    + k \sqrt{t \de_k}   \int_{B_{\frac{\eta}{k t \de_k }}(0)} \bigg( \frac{1}{(1+|t \de_k y + \xi_{1,k,t}|^2)^{\frac12}}  - \frac{1}{(1+|\xi_{1,k,t}|^2)^{\frac12}} \bigg) \frac{1}{(1+|y|^2)^{\frac32}} dy.
\end{align*}
Now, using polar coordinates, we get that
\begin{equation} \label{E.integralMainOrder}
\int_{B_{\frac{\eta}{k t \de_k }}(0)} \frac{1}{(1+|y|^2)^{\frac32}} dy = 2\pi + O \bigg( \frac{1}{k \log^2(k)} \bigg), \quad \textup{as } k \to \infty, \quad \textup{uniformly in } t \in [a,b].
\end{equation}
Moreover, observe that
$$
\frac{1}{(1+|t \de_k y + \xi_{1,k,t}|^2)^{\frac12}}  - \frac{1}{(1+|\xi_{1,k,t}|^2)^{\frac12}} = O \big( \de_k^2 |y|^2 + \de_k |y| \big),  \quad  \textup{as } k \to \infty, \quad \textup{uniformly in } t \in [a,b], 
$$
and so that
$$
k \sqrt{t \de_k}   \int_{B_{\frac{\eta}{k t \de_k }}(0)} \bigg( \frac{1}{(1+|t \de_k y + \xi_{1,k,t}|^2)^{\frac12}}  - \frac{1}{(1+|\xi_{1,k,t}|^2)^{\frac12}} \bigg) \frac{1}{(1+|y|^2)^{\frac32}} dy \lesssim \frac{1}{k^2 \log^2(k)}.
$$

Combining all these estimates, we conclude that the second term in \eqref{E.energyExpansion2} can be expanded as
\begin{equation} \label{E.energyMain1}
\int_{\RR^2} U \Sjk U_{j,k,t}^3 dx = \frac{\pi}{\log k} \sqrt{2t} + O \bigg( \frac{1}{\log^3(k)} \bigg), \quad \textup{as } k \to \infty, \quad \textup{uniformly in } t \in [a,b].
\end{equation}

Finally, we deal with the third term in \eqref{E.energyExpansion2}. First, observe that
\begin{align*}
    & \Sjk \sum_{i \neq j} \int_{\RR^2} U_{j,k,t}^3 U_{i,k,t} = k \Sjk \sum_{i \neq j} \int_{\Om_1} U_{j,k,t}^3 U_{i,k,t} dx = k \sum_{i=2}^k \int_{\Om_1} U_{1,k,t}^3 U_{i,k,t} dx + k \sum_{j=2}^k \sum_{i\neq j} \int_{\Om_1} U_{j,k,t}^3 U_{i,k,t} dx  \\
    & \quad = k \sum_{i=2}^k \int_{\Om_1} U_{1,k,t}^3 U_{i,k,t} dx + k \int_{\Om_1} U_{1,k,t} \bigg( \sum_{j=2}^k U_{j,k,t}^3 \bigg) dx + k \sum_{j=2}^k \sum_{\substack{i=2\\i \neq j}}^k \int_{\Om_1} U_{j,k,t}^3 U_{i,k,t} dx.
\end{align*}
The last two terms can be directly handled using Lemma \ref{L.PremoselliVetois}. Indeed, observe that
\begin{align*}
    & k \int_{\Om_1} U_{1,k,t} \bigg( \sum_{j=2}^k U_{j,k,t}^3 \bigg) dx + k \sum_{j=2}^k \sum_{\substack{i=2\\i \neq j}}^k \int_{\Om_1} U_{j,k,t}^3 U_{i,k,t} dx \\
    & \quad \leq  k \int_{\Om_1} U_{1,k,t} \bigg( \sum_{j=2}^k U_{j,k,t} \bigg)^3 dx + k \int_{\Om_1} \bigg( \sum_{j=2}^k U_{j,k,t} \bigg)^4 dx. 
\end{align*}
Then, for the first term on the right hand side, we use Lemma \ref{L.PremoselliVetois} with $\alpha=1$ and $\beta=3$, and get that
\begin{align*}
    k \int_{\Om_1} U_{1,k,t} \bigg( \sum_{j=2}^k U_{j,k,t} \bigg)^3 dx\ls \frac{1}{k}. 
\end{align*}

Similarly, by Lemma \ref{L.PremoselliVetois} with $\alpha=0$ and $\beta=4$, it follows that
\begin{align*}
k \int_{\Om_1} \bigg( \sum_{j=2}^k U_{j,k,t} \bigg)^4 dx\ls \frac{1}{\log^4(k)}. 
\end{align*}

Next, fixed $\eta \in (0,\tfrac12)$ as in the proof of \eqref{E.energyMain1}, we write
\begin{align*}
    k  \int_{\Om_1} U_{1,k,t}^3 \sum_{j=2}^k U_{j,k,t} dx = k \int_{B_{\frac{\eta}{k}}(\xi_{1,k,t})} U_{1,k,t}^3  \sum_{j=2}^k U_{j,k,t} dx + k \int_{\Om_1 \setminus B_{\frac{\eta}{k}}(\xi_{1,k,t}) } U_{1,k,t}^3  \sum_{j=2}^k U_{j,k,t} dx.
\end{align*}
On the one hand, combining \eqref{id-intklogk} and \eqref{E.I2}, \eqref{E.I3} with $\al = 3$ and $\be = 1$, we see that 
$$
k \int_{\Om_1 \setminus B_{\frac{\eta}{k}}(\xi_{1,k,t}) } U_{1,k,t}^3  \sum_{j=2}^k U_{j,k,t} dx \lesssim \frac{1}{k \log^3(k)}.
$$

\noindent On the other hand, observe that
\begin{align*}
    & k \int_{B_{\frac{\eta}{k}}(\xi_{1,k,t})} U_{1,k,t}^3  \sum_{j=2}^k U_{j,k,t} dx = k t \de_k \int_{B_{\frac{\eta}{kt\de_k}}(0)} \frac{1}{(1+|y|^2)^{\frac32}} \sum_{j=2}^k \frac{1}{(t^2 \de_k^2 + |t\de_k y + \xi_{1,k,t}-\xi_{j,k,t}|^2)^{\frac12}} dy \\
    & \quad = kt\de_k \bigg(  \int_{B_{\frac{\eta}{kt\de_k}}(0)} \frac{1}{(1+|y|^2)^{\frac32}} dy \bigg) \sum_{j=2}^k \frac{1}{|\xi_{1,k,t}-\xi_{j,k,t}|} \\
    & \qquad + k t \de_k \int_{B_{\frac{\eta}{kt\de_k}}(0)} \frac{1}{(1+|y|^2)^{\frac32}} \sum_{j=2}^k \bigg( \frac{1}{(t^2 \de_k^2 + |t\de_k y + \xi_{1,k,t}-\xi_{j,k,t}|^2)^{\frac12}}  - \frac{1}{|\xi_{1,k,t}-\xi_{j,k,t}|} \bigg) dy.
\end{align*}
Then, combining \eqref{E.integralMainOrder} with the following expansion (see, for instance, \cite[Appendix A]{MM19}):
$$
\sum_{j=2}^k \frac{1}{|\xi_{1,k,t}-\xi_{j,k,t}|} = \frac{k}{\pi} (\log k + o(\log k)), \quad \textup{as } k \to \infty,
$$
we get that
$$
kt\de_k \bigg(  \int_{B_{\frac{\eta}{kt\de_k}}(0)} \frac{1}{(1+|y|^2)^{\frac32}} dy \bigg) \sum_{j=2}^k \frac{1}{|\xi_{1,k,t}-\xi_{j,k,t}|} = \frac{2t}{\log k} (1+o(1)), \quad \textup{as } k \to \infty.
$$
Similarly, using that for $p > 1$ (see, for instance, \cite[(3.73)]{PV19}),
$$
\sum_{j=2}^k \frac{1}{|\xi_{1,k,t}-\xi_{j,k,t}|^p} \lesssim k^p, \quad \textup{as } k \to \infty,
$$
and arguing as in the proof of \eqref{E.energyMain1} we conclude that 
$$
k t \de_k \int_{B_{\frac{\eta}{kt\de_k}}(0)} \frac{1}{(1+|y|^2)^{\frac32}} \sum_{j=2}^k \bigg( \frac{1}{(t^2 \de_k^2 + |t\de_k y + \xi_{1,k,t}-\xi_{j,k,t}|^2)^{\frac12}}  - \frac{1}{|\xi_{1,k,t}-\xi_{j,k,t}|} \bigg) dy \lesssim \frac{1}{k \log^3(k)}. 
$$

Combining all these estimates, we conclude that
\begin{equation} \label{E.energyMain2}
    -\frac12  \Sjk \sum_{i \neq j} \int_{\RR^2} U_{j,k,t}^3 U_{i,k,t} = - \frac{t}{\log k} (1+ o(1)), \quad \textup{as } k \to \infty. 
\end{equation}
\noindent The result follows combining \eqref{E.energyExpansion2}, \eqref{E.energyMain1} and \eqref{E.energyMain2}. 
\end{proof}

Having this lemma at hand, we can now prove Proposition \ref{P.energyExpansion}.

\begin{proof}[Proof of Proposition \ref{P.energyExpansion}]
    First note that, by the fundamental theorem of calculus, 
    $$
    I(U_{k,t}+ \phi_{k,t}) - I(U_{k,t}) = \int_{0}^1 \frac{d}{ d s} I(U_{k,t} + s \phi_{k,t} )\, ds. 
    $$
    Also, using Lemma \ref{L.error} and Proposition \ref{P.contraction}, we get that
    \begin{equation} \label{E.erroragaistphi}
    \left| \frac{d}{ds} I(U_{k,t} + s \phi_{k,t} ) \Big|_{s=0} \right|  = \left| \int_{\RR^2} \cE_{k,t} \phi_{k,t} dx \right| \leq \|\cE_{k,t}\|_{\frac43} \|\phi_{k,t}\|_4 \lesssim \|\cE_{k,t}\|_{\frac43} \|\phi_{k,t}\|_{\dot{H}^{\frac12}} \lesssim \frac{1}{\log^2(k)}\,.
    \end{equation}
    Thus, using again the fundamental theorem of calculus and changing the order of integration, we infer that
    \begin{equation} \label{E.energyExpansion1minuss}
    \begin{aligned}
        & I(U_{k,t}+ \phi_{k,t}) - I(U_{k,t}) = \int_0^1 \left( \int_0^s \frac{d^2}{d\sigma^2} I(U_{k,t} + \sigma \phi_{k,t} ) \,d\sigma \right) ds +  \frac{d}{ds} I(U_{k,t} + s \phi_{k,t} ) \Big|_{s=0} \\
        & \quad = \int_{0}^1 (1-\sigma) \frac{d^2}{d \sigma^2}  I(U_{k,t} + \sigma \phi_{k,t} ) \, d\sigma +  \frac{d}{ds} I(U_{k,t} + s \phi_{k,t} ) \Big|_{s=0} \\
        & \quad  = \int_{0}^1 (1-\sigma) \frac{d^2}{d \sigma^2}  I(U_{k,t} + \sigma \phi_{k,t} ) \, d\sigma +  O \Big( \frac{1}{\log^2 (k)} \Big), \quad \textup{as } k \to \infty, \textup{ uniformly in } t \in [a,b]\,.
    \end{aligned}
    \end{equation}
    
    Next, using that $\phi_{k,t}$ is a solution to \eqref{E.nonlinearProjected}, we get that
    $$
    \frac{d^2}{d\sigma^2} I(U_{k,t} + \sigma \phi_{k,t}) = \int_{\RR^2} \big( \cE_{k,t} + \cN_{k,t}(\phi_{k,t}) \big) \phi_{k,t} \, dx - 3 \int_{\RR^2} \big( (U_{k,t}+\sigma\phi_{k,t})^2- U_{k,t}^2 \big) \phi_{k,t}^2\, dx \,.
    $$
    Then, arguing as in the proof of Proposition \ref{P.contraction} and in \eqref{E.erroragaistphi}, we can estimate each term on the right hand side and conclude that
    $$
    \left|  \frac{d^2}{d\sigma^2} I(U_{k,t} + s \phi_{k,t}) \right|  \lesssim \frac{1}{\log^2(k)}\,.
    $$
    The result follows from the expansion in \eqref{E.energyExpansion1minuss} and Lemma \ref{L.energyExpansionAnsatz}.
\end{proof}

\subsection{Proof of Theorem \ref{T.mainThmYamabe}}\label{sec:proof-main-result1} We now have all the necessary ingredients to prove Theorem \ref{T.mainThmYamabe}. Here, for $0 < a < b < \infty$, $t \in [a,b]$, $k_2 \in \NN$ as in Proposition \ref{P.contraction} and $k \in \NN$ with $k \geq k_2$, we let $\phi= \phi_{k,t}$ denote the unique solution to \eqref{E.nonlinearProjected} satisfying \eqref{E.nonlinearSize} provided by Proposition \ref{P.contraction}. 

\begin{proof}[Proof of Theorem \ref{T.mainThmYamabe}]
Let $f : \RR \to \RR$ be the one-dimensional function given in Remark \ref{R.strictGlobalMaximum}. We know that $f$ has a strict global maximum at $t_{\rm max} := \pi^2/2$. Hence, it follows from Proposition \ref{P.energyExpansion} that, for $k$ sufficiently large, the function
$
t \mapsto I(U_{k,t} + \phi_{k,t})
$
has a critical point $t_k > 0$ such that $t_k \to t_{\rm max}$ as $k \to \infty$. We then set
\begin{equation} \label{E.psikSolution}
\psi_k := U_{k,t_k} + \phi_{k,t_k}\,,
\end{equation} 
and prove that, for all $k$ sufficiently large, $\psi_k \in \Hud$ is the solution to \eqref{E.Yamabe} we were looking for. 

First, since $\phi_{k,t_k}$ is a solution to \eqref{E.nonlinearProjected}, we get the existence of $c_k \in \RR$ such that
\begin{equation} \label{E.th21conclusion}
I'(\psi_k) = c_k \langle Z_{k,t_k}, \cdot \rangle \,.
\end{equation}
Hence, since $t_k$ is a critical point of the function $t \mapsto I(U_{k,t} + \phi_{k,t})$, it follows that
\begin{equation} \label{E.0equalprojection}
0 =  I'(\psi_k) \frac{d}{dt} (U_{k,t}+\phi_{k,t}) \Big|_{t = t_k}  = c_k  \Big\langle Z_{k,t_k},  \frac{d}{dt} (U_{k,t}+\phi_{k,t}) \Big|_{t = t_k} \Big\rangle\,.
\end{equation}

Now, observe that
$$
\frac{d}{dt} U_{k,t}(x) = \frac1t \left( Z_{k,t}(x) - \frac{t^{\frac52}\de_k^{\frac52}}{1-t^2\de_k^2} \Sjk \frac{(x-\xi_{j,k,t})\cdot \xi_{j,k,t}}{(|x-\xi_{j,k,t}|^2 + t^2 \de_k^2)^{\frac32}} \right).
$$
Thus, arguing as we did in the proof of Lemma \ref{L.aprioriEstimate} to estimate $\|Z_{k,t} + \bK[Z_{k,t}]\|_{\Hud}$, we get that 
\begin{align*}
    & t_k \Big\langle Z_{k,t_k}, \frac{d}{dt} U_{k,t}\Big|_{t=t_k} \Big\rangle =  \|Z_{k,t_k}\|_{\Hud}^2- \frac{t_k^{\frac52}\de_k^{\frac52}}{1-t_k^2\de_k^2} \Big \langle Z_{k,t_k}, \Sjk \frac{(x-\xi_{j,k,t_k})\cdot \xi_{j,k,t_k}}{(|x-\xi_{j,k,t}|^2 + t_k^2 \de_k^2)^{\frac32}} \Big \rangle  = \frac{k\pi}{4} +O\left(\frac{1}{\log k}\right)\,,
\end{align*}
as $k \to \infty$, and so that
\begin{equation} \label{E.ZktagainsddtUkt}
\Big\langle Z_{k,t_k}, \frac{d}{dt} U_{k,t}\Big|_{t=t_k} \Big\rangle = \frac{k}{2\pi} + o(k)\,, \quad \textup{as } k \to \infty\,.
\end{equation}

On the other hand, since $\phi_{k,t} \in \cK_{k,t}^{\perp}$, we have that
$$
0=\frac{d}{dt} \langle Z_{k,t}, \phi_{k,t}\rangle = \Big \langle \frac{d}{dt}Z_{k,t}, \phi_{k,t} \Big \rangle + \Big \langle Z_{k,t}, \frac{d}{dt} \phi_{k,t} \Big \rangle,
$$
and so that
$$
\Big \langle Z_{k,t_k}, \frac{d}{dt} \phi_{k,t} \Big|_{t = t_k} \Big \rangle = - \Big \langle \frac{d}{dt} Z_{k,t}\Big|_{t = t_k}, \phi_{k,t_k} \Big \rangle\,.
$$
Hence, we have that
$$
\left| \Big \langle Z_{k,t_k}, \frac{d}{dt} \phi_{k,t} \Big|_{t = t_k} \Big \rangle \right| \leq \|\phi_{k,t_k}\|_{\Hud} \left\| \frac{d}{dt} Z_{k,t} \Big|_{t = t_k} \right\|_{\Hud} \lesssim \frac{1}{\log k} \left\| \frac{d}{dt} Z_{k,t} \Big|_{t = t_k} \right\|_{\Hud} \,.
$$
Moreover, observe that
\begin{align*}
    & \left\| \frac{d}{dt} Z_{k,t} \Big|_{t = t_k} \right\|_{\Hud} \lesssim \sum_{j=1}^k \left\| \frac{d}{dt}(3U_{j,k,t}^2 Z^{(0)}_{t \de_k, \xi_{j,k,t}} )\Big|_{t = t_k} \right\|_{L^{\frac43}} \\
    & \quad = 3 \sum_{j=1}^k \left\| 2 U_{j,k,t_k} \Big( \frac{d}{dt} U_{j,k,t} \Big|_{t=t_k} \Big) Z^{(0)}_{t_k \de_k, \xi_{j,k,t_k}} + U_{j,k,t_k}^2 \Big( \frac{d}{dt}Z^{(0)}_{t \de_k, \xi_{j,k,t}}\Big|_{t=t_k} \Big) \right\|_{L^{\frac43}} \\
    & \quad  \leq 6 \sum_{j=1}^k \left\| U_{j,k,t_k} \Big(\frac{d}{dt} U_{j,k,t} \Big|_{t=t_k} \Big) Z^{(0)}_{t_k \de_k, \xi_{j,k,t_k}}\right\|_{L^{\frac43}} + 3 \sum_{j=1}^k \left\|U^2_{j,k,t_k} \Big( \frac{d}{dt}Z^{(0)}_{t \de_k, \xi_{j,k,t}}\Big|_{t=t_k} \Big) \right\|_{L^{\frac43}} \lesssim k \,,
\end{align*}
and so that
\begin{equation} \label{E.Zktagainsddtphikt}
    \left| \Big \langle Z_{k,t_k}, \frac{d}{dt} \phi_{k,t} \Big|_{t = t_k} \Big \rangle \right| \lesssim \frac{k}{\log k}\,.
\end{equation}

Gathering \eqref{E.0equalprojection}--\eqref{E.Zktagainsddtphikt}, we get that
$$
0 = c_k \left( \frac{k}{2\pi} + o(k) \right)\,,
$$
and so that, for $k$ sufficiently large, $c_k = 0$. The result follows combining this fact with \eqref{E.th21conclusion}.
\end{proof}

\section{Nonradial stationary solutions to the SQG equation. Proof of Theorem \ref{T.mainThmSQG}} \label{S.nonradial}

Having Theorem \ref{T.mainThmYamabe} at hand, we can now prove Theorem \ref{T.mainThmSQG}. To that end, we just have to prove some extra properties of the solutions to \eqref{E.Yamabe} we constructed in Theorem \ref{T.mainThmYamabe}. 

\subsection{The half-Yamabe solution is nonradial and smooth} 

We start by proving that the constructed solutions are nonradial, bounded and smooth. Since these results are of independent interest, we treat them in a separate subsection. 

\begin{proposition} \label{P.nonradial}
    Let $k_0 \in \NN$ and $\psi_k \in \dot{H}^{\frac12}$ be as in Theorem \ref{T.mainThmYamabe}. There exists $k_{\rm rad} \in \NN$ with $k_{\rm rad} \geq k_0$ such that, for all $k \geq k_{\rm rad}$, $\psi_k$ is nonradial. 
\end{proposition}

\begin{remark}
    In the local case \cite{dPMPP11}, the nonradiality of the constructed solutions immediately follows from the Pohozaev identity. Nevertheless, this is not the case for the half-Yamabe equation. Indeed, to prove that the solutions to \eqref{E.Yamabe} constructed in Theorem \ref{T.mainThmYamabe} are nonradial, we use a completely different approach based on subtle energy estimates. 
\end{remark}

\begin{proof}
We assume by contradiction that 
$
\psi_k = \Sjk U_{j,k,t} - U + \phi_k
$
is radial. Since $U$ is radial, this implies
\begin{equation} \label{E.phikUk}
\phi_k (\cR_\theta\, \cdot) - \phi_k = \Sjk \big(U_{j,k,t} - U_{j,k,t}(\cR_\theta \,\cdot) \big)\,,
\end{equation}
for every rotation matrix $\cR_\theta$ of angle $\theta \in [0,2\pi)$. 

On the one hand, taking into account Proposition \ref{P.contraction}, we get that
\begin{equation} \label{E.phikUpper}
    \|\phi_k(\cR_\theta\,\cdot)-\phi_k\|_{\Hud}^2\lesssim \frac{1}{\log^2(k)}, \quad \textup{for all } \theta \in [0,2\pi).
\end{equation}

On the other hand, choosing $\theta = \frac{\pi}{k}$, it follows that
\begin{align*}
    & \left\| \Sjk U_{j,k,t} - \Sjk U_{j,k,t} (\cR_{\frac{\pi}{k}}\, \cdot) \right\|_{\dot{H}^{\frac12}}^2 = \int_{\RR^2} \bigg|\Sjk (-\Delta)^\frac{1}{4}U_{j,k,t}-(-\Delta)^{\frac14} U_{j,k,t}(\cR_{\frac{\pi}{k}}\, \cdot)\bigg|^2 dx \\
    & \quad = \int_{\RR^2} \bigg| \Sjk(-\Delta)^{\frac14} U_{j,k,t}\bigg|^2 dx + \int_{\RR^2} \bigg| \Sjk(-\Delta)^{\frac14} U_{j,k,t}(\cR_{\frac{\pi}{k}}\, \cdot)\bigg|^2 dx \\
    & \qquad - 2 \int_{\RR^2} \bigg(\Sjk(-\Delta)^{\frac14} U_{j,k,t} \bigg) \bigg( \sum_{i=1}^k(-\Delta)^{\frac14} U_{i,k,t}(\cR_{\frac{\pi}{k}}\, \cdot) \bigg) dx \\
    & \quad = 2 \int_{\RR^2} \bigg| \Sjk(-\Delta)^{\frac14} U_{j,k,t}\bigg|^2 dx - 2 \int_{\RR^2} \bigg(\Sjk(-\Delta)^{\frac14} U_{j,k,t} \bigg) \bigg( \sum_{i=1}^k(-\Delta)^{\frac14} U_{i,k,t}(\cR_{\frac{\pi}{k}}\, \cdot) \bigg) dx. 
\end{align*}
Moreover, using that $U_{j,k,t}$ solves \eqref{E.Yamabe}, we infer that
\begin{equation*} 
\begin{aligned}
    & \int_{\RR^2} \bigg| \Sjk(-\Delta)^{\frac14} U_{j,k,t}\bigg|^2 dx = \Sjk \sum_{i=1}^k \int_{\RR^2} (-\Delta)^{\frac12} U_{j,k,t} U_{i,k,t} dx =  \Sjk \sum_{i=1}^k  \int_{\RR^2} U_{j,k,t}^3 U_{i,k,t}dx \\
    & \quad = k \|U\|_{L^4}^4 + \Sjk \sum_{i \neq j} \int_{\RR^2} U_{j,k,t}^3 U_{i,k,t}dx,
\end{aligned}
\end{equation*}
and so that
\begin{equation} \label{E.decompositionRotation}
\begin{aligned}
& \frac12\left\| \Sjk U_{j,k,t} - \Sjk U_{j,k,t} (\cR_{\frac{\pi}{k}}\, \cdot) \right\|_{\dot{H}^{\frac12}}^2 \\
& \quad = k \|U\|_{L^4}^4 + \Sjk \sum_{i \neq j} \int_{\RR^2} U_{j,k,t}^3 U_{i,k,t}
 -  \int_{\RR^2} \bigg(\Sjk(-\Delta)^{\frac14} U_{j,k,t} \bigg) \bigg( \sum_{i=1}^k(-\Delta)^{\frac14} U_{i,k,t}(\cR_{\frac{\pi}{k}}\, \cdot) \bigg) dx.
\end{aligned}
\end{equation}

Next, observe that
\begin{align*}
    & \Sjk \sum_{i \neq j} \int_{\RR^2} U_{j,k,t}^3 U_{i,k,t} = k \Sjk \sum_{i \neq j} \int_{\Om_1} U_{j,k,t}^3 U_{i,k,t} dx = k \sum_{i=2}^k \int_{\Om_1} U_{1,k,t}^3 U_{i,k,t} dx + k \sum_{j=2}^k \sum_{i\neq j} \int_{\Om_1} U_{j,k,t}^3 U_{i,k,t} dx  \\
    & \quad = k \sum_{i=2}^k \int_{\Om_1} U_{1,k,t}^3 U_{i,k,t} dx + k \int_{\Om_1} U_{1,k,t} \bigg( \sum_{j=2}^k U_{j,k,t}^3 \bigg) dx + k \sum_{j=2}^k \sum_{\substack{i=2\\i \neq j}}^k \int_{\Om_1} U_{j,k,t}^3 U_{i,k,t} dx  \\
    & \quad \leq k \int_{\Om_1} U_{1,k,t}^3 \bigg( \sum_{j=2}^k U_{j,k,t} \bigg) dx + k \int_{\Om_1} U_{1,k,t} \bigg( \sum_{j=2}^k U_{j,k,t} \bigg)^3 dx + k \int_{\Om_1} \bigg( \sum_{j=2}^k U_{j,k,t} \bigg)^4 dx. 
\end{align*}
For the first integral we use Lemma \ref{L.PremoselliVetois} with $\alpha=3$ and $\beta=1$,
\begin{align*}
    k \int_{\Om_1} U_{1,k,t}^3 \bigg( \sum_{j=2}^k U_{j,k,t} \bigg) dx\ls k\delta_k^2(k\log(k)+k\log(k)\delta_k^{-1})\ls \frac{1}{\log(k)}.
\end{align*}
For the second integral we use Lemma \ref{L.PremoselliVetois} with $\alpha=1$ and $\beta=3$,
\begin{align*}
    k \int_{\Om_1} U_{1,k,t} \bigg( \sum_{j=2}^k U_{j,k,t} \bigg)^3 dx\ls \frac{1}{k}.
\end{align*}
For the third and last integral we use Lemma \ref{L.PremoselliVetois} with $\alpha=0$ and $\beta=4$,
\begin{align*}
k \int_{\Om_1} \bigg( \sum_{j=2}^k U_{j,k,t} \bigg)^4 dx\ls k\delta_k^2(k^3+k^4\log^4(k)k^{-2})\ls \frac{1}{\log^4(k)}.
\end{align*}
Thus, we conclude that
\begin{equation} \label{E.Uj3Ui}
 \Sjk \sum_{i \neq j} \int_{\RR^2} U_{j,k,t}^3 U_{i,k,t} dx \lesssim \frac{1}{\log(k)}.
\end{equation}

Likewise, observe that
\begin{align*}
    & \int_{\RR^2} \bigg(\Sjk(-\Delta)^{\frac14} U_{j,k,t} \bigg) \bigg( \Sjk(-\Delta)^{\frac14} U_{j,k,t}(\cR_{\frac{\pi}{k}}\cdot) \bigg) dx = \int_{\RR^2} \bigg(\Sjk U_{j,k,t}^3 \bigg) \bigg( \sum_{i=1}^k U_{i,k,t}(\cR_{\frac{\pi}{k}} \cdot) \bigg) dx \\
    & \quad \leq  \sum_{j=1}^{2k} \sum_{i\neq j}  \int_{\RR^2} U_{t\delta_k, \widetilde{\xi}_j}^3 U_{t\delta_k, \widetilde{\xi}_i} dx, 
\end{align*}
where
$$
\widetilde{\xi}_j = \sqrt{1-t^2\delta_k^2} \big( \cos\big( \tfrac{\pi(j-1)}{k}\big), \sin\big( \tfrac{\pi(j-1)}{k}\big) \big)\,.
$$
The integral on the right hand side can be estimated exactly as we did to prove \eqref{E.Uj3Ui}. Indeed, $\sum_{j=1}^{2k} U_{t\de_k, \widetilde{\xi}_j}$ can be seen as a generalization of $\sum_{j=1}^k U_{j,k,t}$ but with $2k$ peaks. Then, we get that
\begin{equation} \label{E.interactionRotation}
   \int_{\RR^2} \bigg(\Sjk(-\Delta)^{\frac14} U_{j,k,t} \bigg) \bigg( \Sjk(-\Delta)^{\frac14} U_{j,k,t}(\cR_{\frac{\pi}{k}}\cdot) \bigg) dx \lesssim \frac{1}{\log(k)}. 
\end{equation}
Combining \eqref{E.decompositionRotation} with \eqref{E.Uj3Ui} and \eqref{E.interactionRotation}, we infer that
$$
\left\| \Sjk U_{j,k,t} - \Sjk U_{j,k,t} (\cR_{\frac{\pi}{k} \cdot)} \right\|_{\dot{H}^{\frac12}}^2 \gtrsim k. 
$$
This lower bound gives us a contradiction with \eqref{E.phikUk} and \eqref{E.phikUpper} for $k$ large enough, and the result follows. 
\end{proof}

Next, we prove that the solutions to \eqref{E.Yamabe} we constructed in Theorem \ref{T.mainThmYamabe} are bounded and smooth. Actually, it is standard to show that any solution to \eqref{E.Yamabe} belonging to $\dot{H}^{\frac12}$ is bounded and smooth. 

\begin{proposition} \label{P.bounded}
    Let $\psi \in \dot{H}^{\frac12}$ be a solution to \eqref{E.Yamabe}. Then $\psi \in L^{\infty}$.
\end{proposition}

\begin{proof}
    For instance, this can be proved arguing as in \cite[Proposition 5.1.1]{DMV}. Indeed, one just has to change the definition of the function $\varphi$ introduced there as follows. Given $\be > 1$ and $T > 0$, we define
    $$
    \varphi(t) := \left\{
    \begin{aligned}
    & -\be T^{\be-1} (t+T) + T^{\be} \quad &&\textup{if } t \leq -T, \\
    & |t|^{\be} && \textup{if } -T < t < T, \\
    \, & \be T^{\be-1}(t-T) + T^{\be} && \textup{if } t \geq T.
    \end{aligned}
    \right.
    $$
    Having this new $\varphi$ at hand, the proof of this result follows as in \cite[Proposition 5.1.1]{DMV}. We refer to \cite[Section 3]{SVWY22} for some details concerning the sign-changing case. 
\end{proof}

\begin{corollary} \label{C.smooth}
    Let $\psi \in \dot{H}^{\frac12}$ be a solution to \eqref{E.Yamabe}. Then $\psi \in C^{\infty}$.
\end{corollary}

\begin{proof}
    By the previous lemma we know that $\psi \in L^p$ for all $4 \leq p \leq \infty$. Hence, the result follows by a standard bootstrap argument. 
\end{proof}

\subsection{Proof of Theorem \ref{T.mainThmSQG}} Taking into account the previous subsection, to conclude the proof of Theorem \ref{T.mainThmSQG} we are essentially missing to show that \eqref{E.distributionalConvergence} holds.

\begin{proof}[Proof of Theorem \ref{T.mainThmSQG}]
    Let $k_{\rm rad} \in \NN$ be as in Proposition \ref{P.nonradial}. By Theorem \ref{T.mainThmYamabe} we know that, for all $k \geq k_{\rm rad}$, there exists $\psi_k \in \dot{H}^{\frac12}$ solving \eqref{E.Yamabe}. Moreover, by Propositions \ref{P.nonradial} and \ref{P.bounded}, and Corollary \ref{C.smooth} we know that $\psi_k$ is nonradial, bounded and smooth. 

    Next, for all $k \geq k_{\rm rad}$, we set  $\theta_k := \psi_{k}^3$. Then, we have that $(\theta_k)_{k=k_{\rm rad}}^{\infty}$ is a sequence of smooth nonradial finite energy stationary solutions to \eqref{E.SQG}. Moreover, let us point out that, for all $k \geq k_{\rm rad}$, $\theta_k$ is $k$-fold symmetric (see \eqref{E.kfold}). Up to a relabeling, this is the sequence of solutions to \eqref{E.SQG} we were looking for. 

    At this point, to conclude the proof, we just have to show that \eqref{E.distributionalConvergence} holds. Let $\vp \in C_c^{\infty}(\RR^2)$ be fixed but arbitrary. Then, it follows that
    \begin{equation} \label{E.thetakU3}
         \int_{\RR^2} (\theta_k + U^3) \vp\, dx = \int_{\RR^2} (-\Delta)^{\frac12} (\psi_{k} + U) \vp\, dx =  \langle U_{k,t_k} + U + \phi_k, \vp \rangle =\Big \langle \Sjk U_{j,k,t_k}, \vp \Big\rangle + \langle \phi_k, \vp \rangle\,.
    \end{equation}
    
    Now, on the one hand, observe that
    \begin{align*}
        & \Big \langle \Sjk U_{j,k,t_k}, \vp \Big\rangle = \Sjk \int_{\RR^2} U_{j,k,t_k}^3 \vp dx = \sqrt{t_k \de_k} \, \Sjk \int_{\RR^2} U^3(y) \vp (t_k \de_k y + \xi_{j,k,t_k}) dy  \\ & \quad = \sqrt{t_k \de_k}\, \Sjk \Big( 2\pi \vp (\xi_{j,k,t_k}) + O\big(\|\vp\|_{C^2(\RR^2)}\,t_k \de_k\big) \Big) \\
        & \quad = \frac{2\pi\sqrt{t_k}}{\log k} \left( \frac{1}{k} \Sjk \vp(\xi_{j,k,t_k}) \right) + O\bigg( \frac{\|\vp\|_{C^2(\RR^2)}}{k^2 \log^3(k)} \bigg), \quad \textup{as } k \to \infty\,.
    \end{align*}
    Thus, taking into account that 
    $$
    t_k \to \frac{\pi^2}{2} \quad \textup{and} \quad 
    \frac{1}{k} \Sjk \de_{\xi_{j,k,t_k}} \rightharpoonup \frac{1}{2\pi}\sigma_{\partial \mathbb{D}}\,, \quad \textup{as } k \to \infty\,,
    $$
    we conclude that
    \begin{equation} \label{E.convergenceMain}
        \log k \Sjk U_{j,k,t_k}^3 \rightharpoonup \frac{\pi}{\sqrt{2}}\, \sigma_{\partial \mathbb{D}}\,, \quad \textup{as } k \to \infty\,.
    \end{equation}

    On the other hand, we know that $\|\phi_k\|_{\Hud} \lesssim (\log k)^{-1}$. Hence, setting $w_k := \log k \phi_k$ for all $k \in \NN$, we have that $(w_k)_k$ is a bounded sequence in $\Hud$. Thus, up to a subsequence, it follows that
    $$
   w_k \rightharpoonup w_{\infty}\ \textup{ in } \Hud \quad \textup{and} \quad w_k \to w_{\infty}\  \textup{ in } L_{loc}^p(\RR^2),\ 1 \leq p < 4\,, \quad \textup{as } k \to \infty\,,
    $$
    for some $w_\infty \in \dot{H}^{\frac12}$. Also, since $\phi_k$ is a solution to \eqref{E.YamabePhi}, we infer that $w_k \in \dot{H}^{\frac12}$ is a solution to 
    \begin{equation*}
        (-\Delta)^{\frac12} w_k - 3 U_{k,t_k}^2 w_k = \log k\, \cE_{k,t_k} + \log k\,\cN_{k,t_k}(\phi_k)\quad \textup{in } \RR^2\,.
    \end{equation*}
    In particular, taking into account the estimates for $\cE_{k,t_k}$ and $\cN_{k,t_k}(\phi_k)$ proven in Lemma \ref{L.error} and Proposition \ref{P.contraction} respectively, we get that, for all $\vp \in C_c^{\infty}(\RR^2)$,
    \begin{equation}
        \langle w_k, \varphi \rangle = 3 \int_{\RR^2} U_{k,t_k}^2 w_k \varphi\, dx + 3 \log k \sum_{j=1}^k \int_{\RR^2} U_{j,k,t_k} U^2 \varphi\, dx + O \bigg( \frac{\|\vp\|_{L^{\infty}(\RR^2)}}{\log k} \bigg)\,, \quad \textup{as } k \to \infty\,.
    \end{equation}
    
    Next, arguing as in Section \ref{S.2}, we get that
    $$
    3  U_{k,t_k}^2 w_k \rightharpoonup 3  U^2 w_{\infty}\,, \quad \textup{as } k \to \infty\,,
    $$
    and that
    $$
    \log k \Sjk U_{j,k,t_k} \rightharpoonup \frac{1}{2\sqrt{2}} \int_{\partial \mathbb{D}} \frac{1}{|\cdot-\,\xi|}\,d\sigma(\xi)\,, \quad \textup{as } k \to \infty\,.
    $$
    Hence, we conclude that $w_{\infty} \in \Hud$ is a solution to
    \begin{equation} \label{E.winfty}
    (-\Delta)^{\frac12} w_{\infty} - 3U^2 w_{\infty} = \frac{3}{2\sqrt{2}} U^2 \int_{\partial \mathbb{D}} \frac{1}{|\cdot-\,\xi|}\,d\sigma(\xi) \quad \textup{in } \RR^2\,,
    \end{equation}
    and so that the remainder in \eqref{E.distributionalConvergence} we were looking for is given by
    \begin{equation} \label{E.Rdistributional} 
    R = 3 U^2 \left( w_{\infty} + \frac{1}{2\sqrt{2}} \int_{\partial \mathbb{D}} \frac{1}{|\cdot-\xi|}d\sigma(\xi) \right).
    \end{equation}
    
    Finally, by a standard bootstrap argument in \eqref{E.winfty}, we see that the most singular part of $R$ is precisely
    $$
    \frac{3}{2\sqrt{2}} \, U^2 \int_{\partial \mathbb{D}} \frac{1}{|\cdot-\,\xi|}\,d\sigma(\xi)\,.
    $$
    It is then straightforward to check that $R \in \dot{H}^s \cap L^p$ for all $1 \leq p < \infty$ and all $0 < s < \frac12$. The distributional convergence in \eqref{E.distributionalConvergence} then follows combining \eqref{E.thetakU3}, \eqref{E.convergenceMain} and \eqref{E.Rdistributional}. 
\end{proof}

\section*{AI disclosure}
The idea of using Newtonian capacity in the proof of Proposition \ref{P.ImprovedSobolev} arose from discussions with ChatGPT. The proofs themselves were developed and written entirely by the authors, who assume full responsibility for the correctness of the results and for their presentation. 

\section*{Acknowledgments}
The authors thank M. Medina for an enlightening discussion concerning Proposition \ref{P.nonradial}.

A.C. acknowledges financial support from the Severo Ochoa Programme for Centers of Excellence (Grants \nolinkurl{CEX2019-000904-S} and \nolinkurl{CEX-2023-001347-S}), funded by 
\nolinkurl{MCIN/AEI/10.13039/501100011033}. A.C. was also supported by the grants \nolinkurl{PID2020-114703GB-I00}, \nolinkurl{PID2024-158418NB-I00}, and \nolinkurl{RED2024-153842-T}, all funded by \nolinkurl{MICIU/AEI/10.13039/501100011033}.

A.J.F. is partially supported by the grants \nolinkurl{PID2023-149451NA-I00 of MCIN/AEI/10.13039/501100011033/FEDER, UE;}  and \nolinkurl{RYC2024-049142-I}, \nolinkurl{MICINN} (Spain).

C.G. has been supported by \nolinkurl{RYC2022-035967-I (MCIU/AEI/10.13039/501100011033} and
\nolinkurl{FSE+)}, and partially by Grants \nolinkurl{PID2022-140494NA-I00} and \nolinkurl{PID2022-137228OB-I00} funded
by \nolinkurl{MCIN/AEI/10.13039/501100011033/FEDER, UE,} by Grant \nolinkurl{C-EXP-265-UGR23} funded by
Consejeria de Universidad, Investigacion e Innovacion \& ERDF/EU Andalusia Program, and
by Modeling Nature Research Unit, project \nolinkurl{QUAL21-011}. Proyecto realizado con la Beca Leonardo a
Investigadores y Creadores Culturales 2024 de la Fundaci\'on BBVA. 

\appendix

\section{Some useful calculations}\label{appendix}

\begin{lemma} \label{L.normZ+KZ}
    Let $0 < a < b < \infty$. Then
    $$
    \|Z_{k,t} + \bK[Z_{k,t}]\|_{\Hud}^2 = k\pi + O \Big( \frac{1}{\log k} \Big)\,, \quad \textup{as } k \to \infty\,.
    $$
\end{lemma}

\begin{proof}
    We start by proving that
    \begin{equation} \label{E.ZktExpansion}
    \|Z_{k,t}\|_{\Hud}^2 = \frac{k\pi}{4} + O \Big( \frac{1}{\log k} \Big)\,, \quad \textup{as } k \to \infty\,.
    \end{equation}
    Indeed, by direct but tedious calculations we have that
    \begin{align*}
        \|Z_{k,t}\|_{\Hud}^2 & = \int_{\RR^2} \Big| \Sjk (-\Delta)^{\frac14} Z_{j,k,t} \Big|^2 dx = 3 \Sjk \sum_{i=1}^k \int_{\RR^2} U_{j,k,t}^2 Z_{j,k,t} Z_{i,k,t} \,dx\\
        & = 3 \Sjk \int_{\RR^2} U(y)^2 Z^{(0)}(y)^2 dy + 3 \Sjk \sum_{i \neq j} \int_{\RR^2} U_{j,k,t}^2 Z_{j,k,t} Z_{i,k,t} \,dx \\
        & = \frac{k\pi}{4} + 3  \Sjk \sum_{i \neq j} \int_{\RR^2} U_{j,k,t}^2 Z_{j,k,t} Z_{i,k,t} \,dx\,.
    \end{align*}
    Also, using that 
    \begin{equation} \label{E.comparisonZU}
    |Z_{j,k,t}(x)| \leq \frac12 U_{j,k,t}(x) \quad \textup{for all } j \in \{1, \ldots,k\}\,,
    \end{equation}
    and \eqref{E.Uj3Ui}, we infer that
    $$
    \bigg|\Sjk \sum_{i \neq j } \int_{\RR^2} U_{j,k,t}^2 Z_{j,k,t} Z_{i,k,t} \,dx \bigg| \leq \frac12  \Sjk \sum_{i \neq j } \int_{\RR^2} U_{j,k,t}^3 U_{i,k,t} \,dx \lesssim \frac{1}{\log k}\,,
    $$
    and so that \eqref{E.ZktExpansion} holds. 

    Having \eqref{E.ZktExpansion} at hand, we can now conclude the proof. Arguing as in Lemma \ref{KZkt}, we get that
    \begin{equation} \label{E.Zkt+KelvinExpansion}
    \begin{aligned}
    & \|Z_{k,t} + \bK[Z_{k,t}]\|_{\Hud}^2 = 4 \|Z_{k,t}\|_{\Hud}^2 + 2 \Sjk \sum_{i =1}^k \bigg\langle Z_{j,k,t}, (t\de_k)^{\frac52} \frac{1-|x|^2}{(|x-\xi_{i,k,t}|^2+ (t\de_k)^2)^{\frac32}} \bigg \rangle \\
    & \quad = k\pi + 2 \Sjk \sum_{i =1}^k \bigg\langle Z_{j,k,t}, (t\de_k)^{\frac52} \frac{1-|x|^2}{(|x-\xi_{i,k,t}|^2+ (t\de_k)^2)^{\frac32}} \bigg \rangle + O \Big(\frac1{\log k} \Big)\,.
    \end{aligned}
    \end{equation}
    Hence, to conclude the proof, we just have to handle the scalar product. First, note that
    \begin{align*}
&\bigg\langle Z_{j,k,t}, \,(t\delta_k)^\frac{5}{2}\frac{1-|x|^2}{(|x-\xi_{i,k,t}|^2+(t\delta_k)^2)^\frac{3}{2}}\bigg\rangle= (t\delta_k)^\frac{5}{2}
\int_{\RR^2} (-\Delta)^\frac{1}{2}Z_{j,k,t}(x) \frac{1-|x|^2}{(|x-\xi_{i,k,t}|^2+(t\delta_k)^2)^\frac{3}{2}} \,dx \\
& \quad =3 (t\delta_k)^\frac{5}{2} \int_{\RR^2}U_{j,k,t}(x)^2 Z_{j,k,t}(x) \frac{1-|x|^2}{(|x-\xi_{i,k,t}|^2+(t\delta_k)^2)^\frac{3}{2}} \,dx\,,
\end{align*}
and so that
\begin{align*}
    & 2 \Sjk \sum_{i =1}^k \bigg\langle Z_{j,k,t}, (t\de_k)^{\frac52} \frac{1-|x|^2}{(|x-\xi_{i,k,t}|^2+ (t\de_k)^2)^{\frac32}} \bigg \rangle \\
    & \quad = 6 (t\de_k)^{\frac52} \Sjk \int_{\RR^2}U_{j,k,t}(x)^2 Z_{j,k,t}(x) \frac{1-|x|^2}{(|x-\xi_{j,k,t}|^2+(t\delta_k)^2)^\frac{3}{2}} \,dx \\
    & \qquad + 6(t\de_k)^{\frac52} \Sjk \sum_{i \neq j} \int_{\RR^2} U_{j,k,t}(x)^2 Z_{j,k,t}(x) \frac{1-|x|^2}{(|x-\xi_{i,k,t}|^2+(t\delta_k)^2)^\frac{3}{2}} \,dx =: J + I_{\neq j}\,.
\end{align*}
We deal with $J$ and $I_{\neq j}$ separately. First, for the $J$ term, we see that
$$
J = 6 \Sjk \int_{\RR^2} U(y)^2 Z^{(0)}(y)\, \frac{(t\de_k)^2(1-|y|^2) -2 t\de_k y \cdot \xi_{j,k,t}}{(1+|y|^2)^{\frac32}}\, dy =: J_1 + J_2\,.
$$
Moreover, one can easily check that
\begin{align*}
    & |J_1| = 3 k (t\de_k)^2 \int_{\RR^2} \frac{(|y|^2-1)^2}{(1+|y|^2)^4} \, dy \lesssim \frac{1}{k^3 \log^4 (k)}\,.\\
    & |J_2| \lesssim k t \de_k \int_{\RR^2} U(y)^2 Z^{(0)}(y) \frac{|y|}{(1+|y|^2)^{\frac32}}\, dy \lesssim \frac{1}{k \log^2(k)}\,.
\end{align*}

On the other hand, we combine the fact that
$$
|(t\de_k)^{\frac52} (1-|x|^2)| \lesssim \sqrt{t\de_k} ((t\de_k)^2 + |x-\xi_{i,k,t}|^2)\,, \quad \textup{for all } i \in \{1, \ldots,k\}\,,
$$
with \eqref{E.comparisonZU} and \eqref{E.Uj3Ui}, and get that
$$
|I_{\neq j}| \lesssim \Sjk \sum_{i \neq j } \int_{\RR^2} U_{j,k,t}^3 U_{i,k,t} \,dx \lesssim \frac{1}{\log k}\,.
$$
The result follows by substituting the estimates of $|J|$ and $|I_{\neq j}|$ into \eqref{E.Zkt+KelvinExpansion}.
\end{proof} 

\begin{proof}[Proof of \eqref{E.ckestimate}]
    First of all, observe that
    \begin{align*}
        & \langle \phi_k - (-\Delta)^{-\frac12} (3U_{k,t_k}^2 \phi_k), Z_{k,t_k} \rangle = \int_{\RR^2} \Big( (-\Delta)^{\frac12} Z_{k,t_k} - 3 U_{k,t_k}^2 Z_{k,t_k} \Big) \phi_k dx \\
        & \quad = 3 \int_{\RR^2} \bigg( \Sjk U_{j,k,t_k}^2 Z_{j,k,t_k} - \bigg( \Sjk U_{j,k,t_k} - U \bigg)^2  \sum_{i=1}^k Z_{i,k,t_k}  \bigg) \phi_k dx \\
        & \quad = 3 \int_{\RR^2} \bigg( \Sjk U_{j,k,t_k}^2 Z_{j,k,t_k} - \bigg( \Sjk U_{j,k,t_k} \bigg)^2 \sum_{i=1}^k Z_{i,k,t_k} \bigg) \phi_k dx  \\
        & \qquad - 3 \int_{\RR^2} U^2 \sum_{i=1}^k Z_{i,k,t_k} \phi_k dx + 6 \int_{\RR^2} U \sum_{j=1}^k \sum_{i=1}^k U_{j,k,t_k} Z_{i,k,t_k} \phi_k dx
    \end{align*}
    Moreover, using \eqref{E.comparisonZU}, we easily see that
    $$
    \bigg| \int_{\RR^2} U^2 \sum_{i=1}^k Z_{i,k,t_k} \phi_k dx \bigg| \lesssim \int_{\RR^2 } \cE_{k,t_k}^{(3)} |\phi_k| dx \,,
    $$
    and that
    $$
    \bigg| \int_{\RR^2} U \sum_{j=1}^k \sum_{i=1}^k U_{j,k,t_k} Z_{i,k,t_k} \phi_k dx \bigg| \lesssim \int_{\RR^2} \cE_{k,t_k}^{(2)} |\phi_k| dx\,,
    $$
where we recall that $\cE_{k,t}^{(\ell)}$, $\ell \in \{1,2,3\}$, were defined in the proof of Lemma \ref{L.error}. Even more, combining \eqref{E.Error3}, \eqref{E.error2} and \eqref{E.phikSectorL4} we easily get that
    $$
     \bigg| \int_{\RR^2} U^2 \sum_{i=1}^k Z_{i,k,t_k} \phi_k dx \bigg| +   \bigg| \int_{\RR^2} U \sum_{j=1}^k \sum_{i=1}^k U_{j,k,t_k} Z_{i,k,t_k} \phi_k dx \bigg| \lesssim \frac{\sqrt{k}}{\log k}\,.
    $$

    On the other hand, by direct but tedious calculations, using again \eqref{E.comparisonZU}, we get that 
    \begin{align*}
        & \bigg| \int_{\RR^2} \bigg( \Sjk U_{j,k,t_k}^2 Z_{j,k,t_k} - \bigg( \Sjk U_{j,k,t_k} \bigg)^2 \sum_{i=1}^k Z_{i,k,t_k} \bigg) \phi_k dx \bigg| \\
        & \quad = k \bigg| \int_{\Om_1} \bigg( \sum_{j=1}^k \sum_{i \neq j} U_{j,k,t_k}^2 Z_{i,k,t_k} + 2 \sum_{j=1}^k \sum_{i \neq j} U_{j,k,t_k} U_{i,k,t_k} Z_{j,k,t_k} + \Sjk \sum_{i \neq j} \sum_{l\neq i,j} U_{j,k,t_k} U_{i,k,t_k} Z_{l,k,t_k} \bigg) \phi_k dx \bigg| \\
        & \quad \lesssim k \int_{\Om_1} \bigg( U_{1,k,t_k}^2 \sum_{j=2}^k U_{j,k,t_k} + U_{1,k,t_k} \bigg(\sum_{j=2}^k U_{j,k,t_k} \bigg)^2 + \bigg( \sum_{j=2}^k U_{j,k,t_k} \bigg)^3 \bigg) |\phi_k| dx \,.
    \end{align*}
    Arguing exactly as we did to estimate $\cE_{k,t}^{(1)}$ in the proof of Lemma \ref{L.error}, we conclude that
    $$
    \bigg| \int_{\RR^2} \bigg( \Sjk U_{j,k,t_k}^2 Z_{j,k,t_k} - \bigg( \Sjk U_{j,k,t_k} \bigg)^2 \sum_{i=1}^k Z_{i,k,t_k} \bigg) \phi_k dx \bigg| \lesssim \frac{\sqrt{k}}{\log^3(k)}\,,
    $$
    and so that
    $$
    | \langle \phi_k - (-\Delta)^{-\frac12} (3U_{k,t_k}^2 \phi_k), Z_{k,t_k} \rangle| \lesssim \frac{\sqrt{k}}{\log k}\,.
    $$
    Combining this estimate with the definition of the coefficient $c_k$ (cf. \eqref{ck}) and Lemma \ref{L.normZ+KZ}, we conclude that \eqref{E.ckestimate} holds. 
\end{proof}

\bibliographystyle{amsplain}

\begin{thebibliography}{99}\frenchspacing



\bibitem{abejeongserrano} K. Abe, J. G\'omez-Serrano, I.-J. Jeong, Homogeneous steady states for the generalized surface quasi-geostrophic equations. \href{https://arxiv.org/abs/2510.03009}{arXiv:2510.03009}.
\bibitem{ADdPMW21} W. Ao, J. D\'avila, M. del Pino, M. Musso, J. Wei, Travelling and rotating solutions to the generalized
inviscid surface quasi-geostrophic equation, Trans. Amer. Math. Soc. 374(9) (2021) 6665--6689.

\bibitem{BSV19} T. Buckmaster, S. Shkoller, V. Vicol, Nonuniqueness of weak solutions to the SQG equation, Comm. Pure Appl. Math., 72(9) (2019) 1809--1874.


\bibitem{CCGS20} A. Castro, D. C\'ordoba, J. G\'omez-Serrano, Global smooth solutions for the inviscid SQG equation, Mem. Amer. Math. Soc. 266 (2020) 89pp.

\bibitem{CFMS25} A. Castro, D. Faraco, F. Mengual, M. Solera, Unstable vortices, sharp non-uniqueness with forcing, and global smooth solutions for the SQG equation. \href{https://arxiv.org/abs/2502.10274}{arXiv:2502.10274}.

\bibitem{CG11} S.-Y.~A. Chang, M.~M. Gonz\'alez, Fractional Laplacian in conformal geometry, Adv. Math. 226 (2) (2011) 1410--1432.
 
\bibitem{CKW25} H. Chen, S. Kim, J. Wei, Sharp quantitative stability estimates for critical points of fractional Sobolev inequalities, Int. Math. Res. Not. 12 (2025) 36pp. 

\bibitem{CLO06} W. Chen, C. Li, B. Ou, Classification of solutions for an integral equation. Comm. Pure Appl. Math. 59 (3) (2006) 330--343.

\bibitem{CKL21} X. Cheng, H. Kwon, D. Li, Non-uniqueness of steady-state weak solutions to the surface quasigeostrophic equations, Comm. Math. Phys. 388 (3) (2021) 1281--1295.


\bibitem{CMT94} P. Constantin, A. J. Majda, E. Tabak, Formation of strong fronts in the 2-D quasigeostrophic thermal
active scalar Nonlinearity 7(6) (1994), 1495–1533.

\bibitem{CLZpp} D. C\'ordoba, J. Lucas-Manch\'on, L. Mart\'inez-Zoroa, Strong ill-posedness and non-existence in Sobolev spaces for generalized-SQG, Nonlinearity 38 (8) (2025).

\bibitem{CordobaZoroa22} D. C\'ordoba, L. Mart\'inez-Zoroa, Non existence and strong ill-posedness in $C^k$ and Sobolev spaces for SQG, Adv. Math. 407 (2022) 74pp.

\bibitem{CZOpp} D. C\'ordoba, L. Mart\'inez-Zoroa, W. S. O\.za\'nski, Instantaneous continuous loss of regularity for the SQG equation, Adv. Math. 481 (2025) 49pp.



\bibitem{DP23} M. Dai, Q. Peng, Non-unique stationary solutions of forced SQG, \href{https://arxiv.org/abs/2302.03283}{arXiv:2302.03283}.
 
\bibitem{DdPS13} J. D\'{a}vila, M. del Pino, Y. Sire, Nondegeneracy of the bubble in the critical case for nonlocal equations, Proc. Amer. Math. Soc. 141 (2013) 3865--3870.

\bibitem{dPMPP11}  M. del Pino, M. Musso, F. Pacard, A. Pistoia, Large energy entire solutions for the Yamabe equation, J. Differential Equations 251(9) (2011) 2568--2597

\bibitem{DMV} S. Dipierro, M. Medina, E. Valdinoci, \textit{Fractional elliptic problems with critical growth in the whole of $\RR^n$,} volume 15 of Appunti. Scuola Normale Superiore di Pisa. Edizioni della Normale, Pisa (2017).

\bibitem{GM16} D. Garrido, M. Musso, Entire sign-changing solutions with finite energy to the fractional Yamabe equation, Pacific J. Math. 283 (1) (2016) 85--114. 

\bibitem{GSPSY21} J. G\'omez-Serrano, J. Park, J. Shi, Y. Yao, Symmetry in stationary and uniformly-rotating solutions of
active scalar equations, Duke Math. J. 170 (13) (2021) 2957--3038.

\bibitem{GQ13} M.~M. Gonz\'alez, J. Qing, Fractional conformal Laplacians and fractional Yamabe problems, Anal. PDE 6 (2013) 1535--1576.


\bibitem{GS19} P. Gravejat, D. Smets, Smooth travelling-wave solutions to the inviscid surface quasi-geostrophic equation, Int. Math. Res. Not. 6 (2019) 1744--1757.

\bibitem{HH10} H. Hance-Olsen, H. Holden, The Kolmogorov–Riesz compactness theorem, Expo. Math. 28 (2010) 385--394.


\bibitem{jeongkimsqg} I.-J. Jeong, J. Kim, Strong ill-posedness for SQG in critical Sobolev spaces, Anal. PDE 17 (1) (2024) 133--170.



\bibitem{L04} Y.Y. Li, Remark on some conformally invariant integral equations: the method of moving spheres. J. Eur. Math. Soc. 6 (2) (2004) 153--180.


\bibitem{LZ95} Y.Y. Li, M. Zhu, Uniqueness theorems through the method of moving spheres. Duke Math. J. 80 (2) (1995) 383--417.

\bibitem{M08} F. Marchand. Existence and regularity of weak solutions to the quasi-geostrophic equations in the spaces $L^p$ or $\dot{H}^{-\frac12}$ Comm. Math. Phys., 277(1) (2008), 45–67.

\bibitem{Mazya} V.G. Maz'ya, {\it Sobolev spaces with applications to elliptic partial differential equations}, second, revised and augmented edition,  Grundlehren der mathematischen Wissenschaften, 342, Springer, Heidelberg, 2011.

\bibitem{MM19} M. Medina, M. Musso, Doubling nodal solutions to the Yamabe equation in $\RR^n$ with maximal rank, J. Math. Pures. Appl. 152(9) (2021) 145--188.

\bibitem{pacual-caballo} M. M. G. Pascual-Caballo, Stationary homogeneous solutions for the inviscid SQG equation. \href{https://arxiv.org/abs/2510.03108}{arXiv:2510.03108}.

\bibitem{PV19} B. Premoselli, J. V\'etois, Compactness of sign-changing solutions to scalar curvature-type equations with bounded negative part, J. Differential Equations 266 (2019) 7416--7458.


\bibitem{R95} S. G. Resnick, Dynamical problems in non-linear advective partial differential equations. 1995. PhD thesis, University of Chicago, Department of Mathematics.

\bibitem{RV11} F. Robert, J. V\'etois, A general theorem for the construction of blowing-up solutions to some elliptic nonlinear equations via Lyapunov-Schmidt's reduction, \textit{Concentration compactness and profile decomposition, Bangalore} (2011), Trends Math., Springer, Basel (2013) 85--116.  
 

\bibitem{SVWY22} X. Su, E. Valdinoci, Y. Wei, J. Zhang, Regularity results for solutions of mixed local and nonlocal elliptic equations, Math. Z. 302(3) (2022) 1855--1878. 

\bibitem{TaylorII} M. Taylor, \textit{Partial Differential Equations II: Qualitative studies of linear equations.} Applied Mathematical Sciences. Springer
New York, 2010.

\bibitem{WY10} J. Wei, S. Yan, Infinitely many solutions for the prescribed scalar curvature problem on $\mathbb{S}^N$, J. Funct. Anal. 258 (9) (2010) 3048--3081.

\bibitem{W05} J. Wu, Solutions of the 2D quasi-geostrophic equation in Hölder spaces. Nonlinear Anal., 62(4) (2005), 579–594.

\end{thebibliography}

\end{document}